\documentclass[oneside]{amsart}
\usepackage{geometry}
\usepackage{amssymb, stmaryrd}
\usepackage{esint}
\usepackage{hyperref}
\usepackage{mathrsfs}
\usepackage[normalem]{ulem}
\usepackage{mathtools}
\usepackage{centernot}
\usepackage{bbm}
\usepackage{bm}
\usepackage{dirtytalk}
\usepackage[parfill]{parskip}
\usepackage[colorinlistoftodos,textsize=tiny]{todonotes}
\usepackage[utf8]{inputenc}
\usepackage{algorithm}
\usepackage{algpseudocode}

\usepackage[numbers,sort&compress]{natbib}

\newcommand{\R}{\mathbb{R}}	
\newcommand{\N}{\mathbb{N}}	

\newcommand{\metric}{\mathsf{d}}

\newcommand{\verticalvec}{\mathbf{v}}
\newcommand{\unitnormal}{\boldsymbol{n}}
\newcommand{\cellstrucP}{\mathcal{P}}
\newcommand{\cellstrucQ}{\mathcal{Q}}
\newcommand{\cellstrucW}{\mathcal{W}}
\newcommand{\partialmatching}{\mathfrak{R}}
\newcommand{\Fill}{\mathsf{F}}

\def\Xint#1{\mathchoice
{\XXint\displaystyle\textstyle{#1}}%
{\XXint\textstyle\scriptstyle{#1}}%
{\XXint\scriptstyle\scriptscriptstyle{#1}}%
{\XXint\scriptscriptstyle\scriptscriptstyle{#1}}%
\!\int}
\def\XXint#1#2#3{{\setbox0=\hbox{$#1{#2#3}{\int}$ }
\vcenter{\hbox{$#2#3$ }}\kern-.6\wd0}}

\def\dashint{\Xint-}

\DeclareMathOperator{\VolEx}{\mathscr{E}}
\DeclareMathOperator{\ProjE}{\mathscr{pE}}

\DeclareMathOperator{\variation}{\mathbf{V}}
\DeclareMathOperator{\essvariation}{\mathbf{ess V}}

\DeclareMathOperator{\PM}{\mathsf{PM}}
\DeclareMathOperator{\IM}{\mathsf{IM}}

\DeclareMathOperator{\BV}{\mathrm{BV}}
\DeclareMathOperator{\PC}{\mathrm{PC}}
\DeclareMathOperator{\Iso}{\mathrm{Iso}}
\DeclareMathOperator{\cyl}{\mathfrak{C}}
\DeclareMathOperator*{\aplim}{ap\,lim}

\newcommand{\allardell}{\,\hbox{\vrule height 6pt depth 0pt
		\vrule height 0.4pt depth 0pt width 5pt}\,}

\newcommand{\paren}[1]{\left( #1 \right)}
\newcommand{\brac}[1]{\left[ #1 \right]}

\newcommand{\abs}[1]{\left\vert#1\right\vert}

\newcommand{\set}[1]{\left\{#1\right\}}

\newcommand{\interior}[1]{%
  {\kern0pt#1}^{\mathrm{o}}%
}

\newcounter{condition}
\renewcommand{\thecondition}{\Alph{condition}}

\DeclareMathOperator{\im}{im}

\newtheorem{theorem}{Theorem}
\newtheorem{remark}[theorem]{Remark}
\newtheorem{lemma}[theorem]{Lemma}
\newtheorem{prop}[theorem]{Proposition}
\newtheorem{definition}[theorem]{Definition}

\newtheorem{corollary}[theorem]{Corollary}

\usepackage{tikz}
\usepackage{tikz-cd}
\usetikzlibrary{arrows.meta,calc,patterns.meta}

\title{The Volume Exchange Metric}
\author{Enrique Alvarado}
\address{Department of Mathematical Sciences, The University of Texas at El Paso, El Paso, TX 79968, USA} 
\author{Joseph Jung}
\address{Department of Mathematical Sciences, George Mason University, Fairfax, VA 22030, USA}
\author{Jeremy Mason}
\address{Department of Materials Science and Engineering, University of California, Davis, Davis, CA 95616, USA}
\author{Shrunal Pothagoni}
\address{Department of Mathematical Sciences, George Mason University, Fairfax, VA 22030, USA}
\author{Benjamin Schweinhart}
\address{Department of Mathematical Sciences, George Mason University, Fairfax, VA 22030, USA}
\date{\today}

\begin{document}

\begin{abstract}
The volume exchange metric between cell structures embedded in Euclidean space is introduced.
Intuitively, it measures the minimum amount of volume that must be swept out to continuously transform one cell structure into another.
The main result is that the volume exchange metric coincides with a discrete analogue defined as a minimum over matchings between cells.
Algorithms are described to compute the discrete volume exchange metric and illustrate its effectiveness at distinguishing polycrystalline microstructures in computational experiments. 
\end{abstract}

\maketitle 
%\tableofcontents

\section{Introduction}

\begin{figure}[b]
\center
		\includegraphics[height=0.3\textwidth]{%
			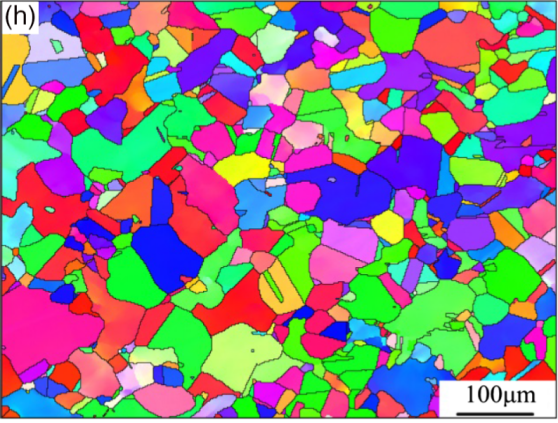} \qquad
	\label{fig:histopathology}{%
		\includegraphics[height=0.3\textwidth]{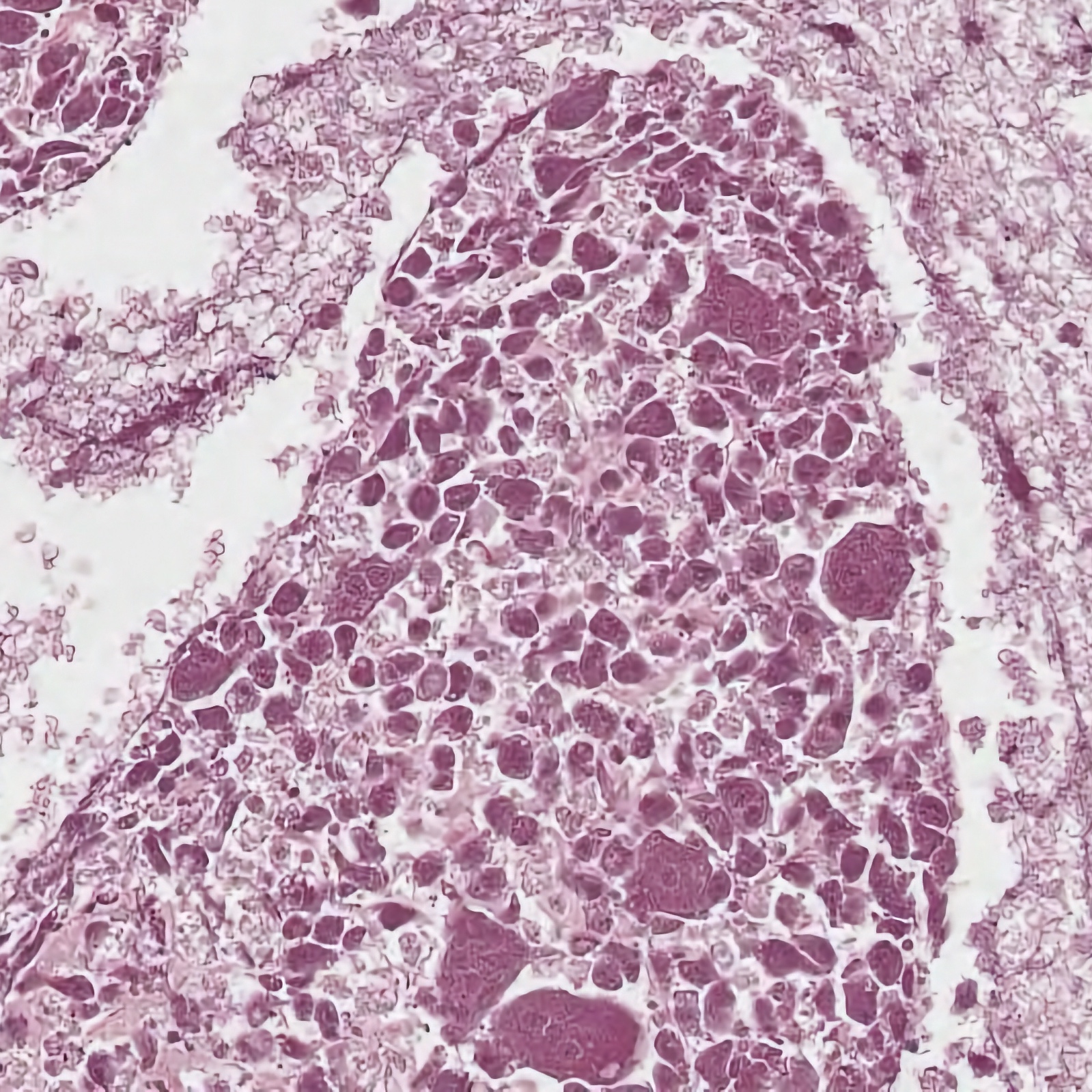}}
\caption{\label{fig:micrograph} (a) A micrograph of a metal showing the polycrystal microstructure, reprinted from Ref.\ \cite{2019jia}, Fig.\ 4, with permission from John Wiley and Sons.
(b) A histopathology slide of cancerous bone tissue~\cite{Leavey2019-js}.}
\end{figure}

Cell structures appear throughout the natural world.
Examples include the decomposition of a polycrystalline material into its constituent grains (Fig.\ \ref{fig:micrograph}a) and the representation of a biological tissue as a union of cells (Fig.\ \ref{fig:micrograph}b).
% See Figure~\ref{fig:micrograph} for two illustrations.
Recent advances in microscopy techniques have made possible the collection of large datasets of two-dimensional images of cell structures, and for some specific applications three-dimensional images are available as well \cite{vasarhelyi2020microcomputed,xu2021open,hu2024situ,oh2025taking}.
% More recently, three-dimensional scans have become available for specific applications.
The underlying physical processes are often stochastic though, and the analysis of the resulting images continues to be challenging.
Specifically, the scarcity of established techniques to compare such images, or even to identify natural classes, has significant practical consequences.
% These datasets present many natural problems for machine learning and artificial intelligence.
For example, it is difficult to evaluate whether a simulated material microstructure is a reasonable approximation of a manufactured one without the ability to quantitatively measure the extent of their differences.
% One might ask whether a computer vision model can be trained to distinguish micrographs of polycrystalline materials produced by different processes, or a manufactured microstructure from a simulation of one.
As a second example, there would be considerable implications for human health if a computer vision model could be trained on the histopathological classification problem to distinguish images of healthy and cancerous biological tissues.
% Another example is the histopathological classification problem: how can one distinguish images of healthy and cancerous biological tissues?

Previous literature on this subject has generally either (i) treated the image as an array of greyscale or colored pixels and applied general image analysis techniques, or (ii) compared ad hoc selections of geometric statistics for individual or small groups of cells.
% Most previous computational studies of cell structures either represented micrographs as a preprocessed images --- that is, as arrays of pixels --- or focused on geometric statistics of single grains or cells.
While the former does not require expert knowledge and can be highly accurate for specific applications \cite{holm2020overview,LIU2021104523}, the performance rapidly degrades for structures outside the support of the training set, and the analysis usually does not indicate which image features are most significant.
The latter instead uses expert knowledge to identify the geometric features most relevant to the properties of interest \cite{groeber2007framework,caicedo2017data}, but this effectively prohibits meaningful comparison with structure classes for which the selected features are irrelevant or undefined \cite{furrer2019industrial}.
It is also significant that both approaches are practically limited to comparatively small sets of cell structures for which the extent of geometric variation is known beforehand.
This can increase the difficulty of analyzing newly-discovered cell structures, discourage the intentional design and exploration of cell structures outside of well-established classes, and promote the fragmentation of the physical science communities concerned with such systems.
This situation has motivated several recent studies that have sought to develop comprehensive and general geometric summaries of the overall cell arrangement, ideally providing sufficient information to reliably identify structural differences and predict system properties while remaining flexible enough to be useful for the analysis of general cell structures.
% More recently, a few studies have sought to develop geometric summaries of arrangements of cells.
% The goal of these methods is to retain sufficient information to capture differences between cell structures and predict physical properties while remaining concrete and geometric enough to yield explainable results.

Here we introduce two dissimilarities between cell structures called the one-to-one volume exchange metric and the many-to-many volume exchange dissimilarity.
Intuitively, they measure the minimum amount of volume that must be swept out by the cell boundaries to transform one structure into the other, as for example in Fig.~\ref{fig:volume_flux}.
The one-to-one volume exchange metric additionally requires that each cell in one structure either transforms into exactly one cell in the other structure or disappears in the process.
The many-to-many volume exchange dissimilarity instead allows initially non-adjacent cells to merge during the transformation.
% Both are physically motivated: in the former, we view each cell as a distinct species that cannot be converted into another, whereas the latter assumes only that adjacent cells are distinct.
% In materials science, one might imagine that different inclusions in a large grain might merge to form a slightly larger inclusion.
% The advantage of these metrics over e.g. the Hausdorff metric is...
One advantage that these have over alternatives like the Hausdorff metric is that they are less sensitive to the geometric noise that is often introduced either by the imaging method or by the cell segmentation process.

\begin{figure}[t]
\center
\includegraphics[height=0.30\textwidth]{%
			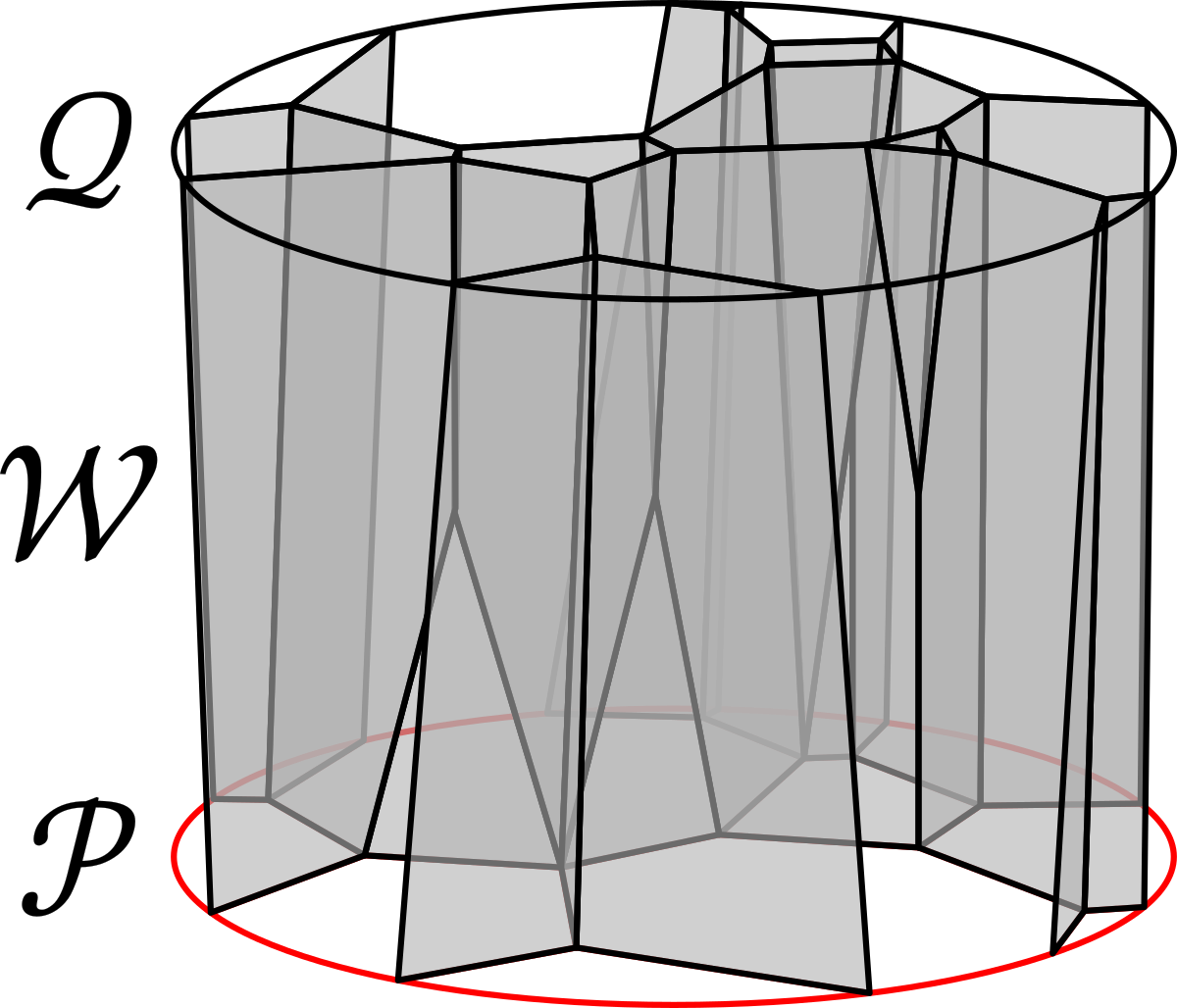} \quad
\includegraphics[height=0.30\textwidth]{%
			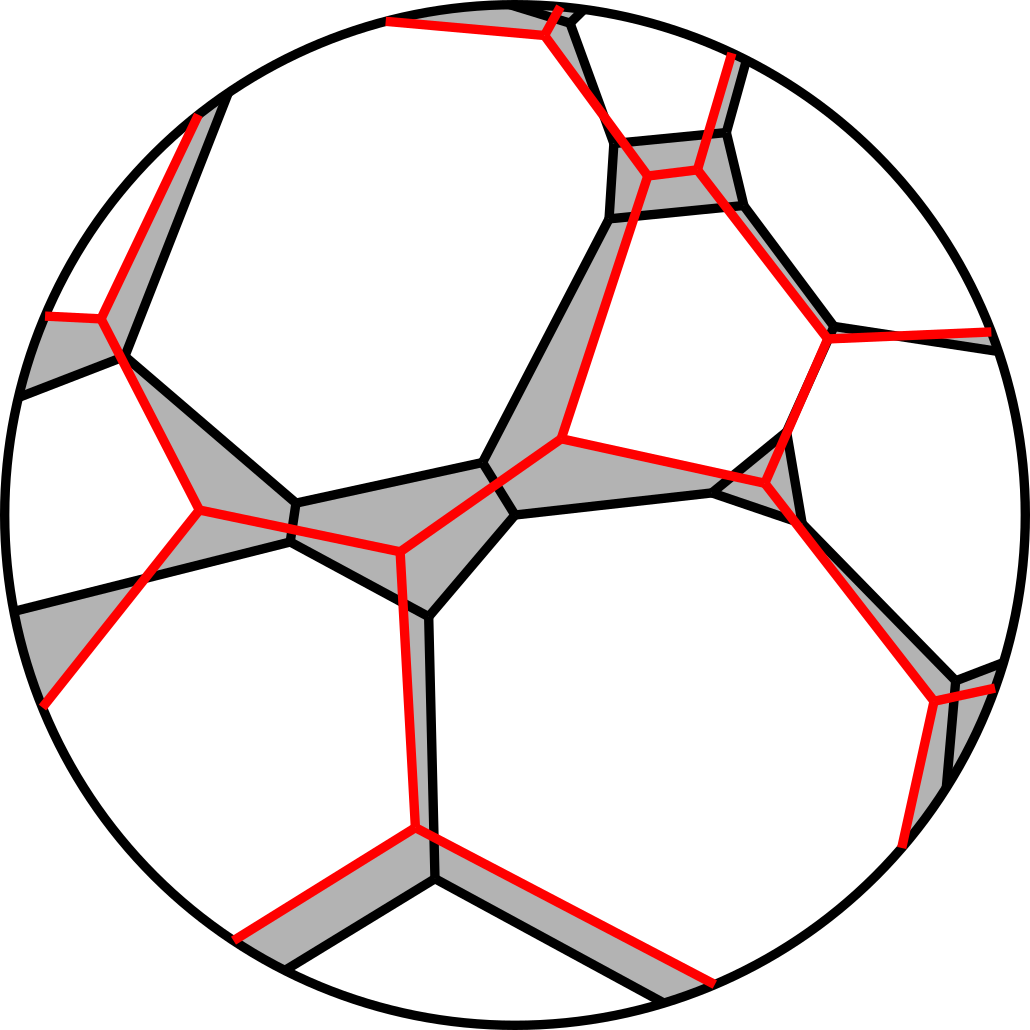}
\caption{\label{fig:volume_flux}
(Left) A cellular filling $\cellstrucW$ between $\cellstrucP$ and $\cellstrucQ$ that realizes the minimal one-to-one volume exchange.
(Right) The projection into the plane %$\R^2 \times 0$
shows the volume exchange distance between $\cellstrucP$ and $\cellstrucQ$ as the area in grey.}
\end{figure}

The underlying physical motivation though is that there are many examples of cell structures for which the transport of volume across the cell boundaries is the slowest (rate-limiting) process involved in changing the cell configuration, and geometric properties that are the slowest to change are likely to be highly significant to physical systems.
With reference to Fig.\ \ref{fig:micrograph}, cell boundary motion in a polycrystal occurs precisely by atoms leaving one grain, passing through the boundary, and joining the crystal lattice of the adjoining grain.
A biological cell of fixed volume is easily able to redistribute the contents of the cell, but the volume changes that accompany general cell configuration changes involve a slower process of water transport across the cell membrane.
Gas transport through the liquid boundaries of the bubbles in a dry soap foam is so slow relative to other processes that in the limit the cell configuration is entirely specified by the topology of the cell boundary network and the cell volumes.
These examples suggest that meaningful dissimilarities could be based directly on the amount of volume exchanged by adjoining cells rather than on, e.g., changes to the area of the cell boundaries or the amount of displacement of the cell boundaries.

The dissimilarities can also be motivated by considering distances on labeled cell structures, with there being many applications for which there is a natural set of labels in a metric space $\paren{Y,\delta}$.
% or could be labeled by the grain's phase identity which is discrete and of which there are usually only a handful.
For example, in materials science the grains of a polycrystalline material could be labeled by a crystal orientation that indicates the misorientation of the crystal lattice relative to a suitable reference, i.e., an element of the quotient $Y = SO(3) / H$ of the three-dimensional rotation group $SO(3)$ by the point symmetry group $H$ of the crystal lattice \cite{bunge2013texture}.
Such a labeled cell structure $\cellstrucP$ in an ambient space $\Omega\subset \R^n$ can be identified with a piecewise constant function $p\colon\Omega\to Y$  whose level sets are the cells.
A natural distance between cell structures $\cellstrucP$ and $\cellstrucQ$ is then given by the $L^1$ norm of $\delta\paren{p,q}$.
However, in the absence of specific knowledge about the cell labels or a natural distance between them, minimal assumptions could be that labels on adjacent cells are distinct and that distinct labels are a unit distance apart.
Given two cell structures $\cellstrucP$ and $\cellstrucQ$, the many-to-many volume exchange dissimilarity is then the minimal $L^1$ distance between $g_\cellstrucP$ and $g_\cellstrucQ$ over all possible labeling $g_{\cellstrucP}$ and $g_{\cellstrucQ}$ which assign distinct labels to adjacent cells.
The one-to-one volume exchange metric---a true metric which is more efficient to compute in practice---is obtained by adding the further assumption that the labels are distinct for each cell in a cell structure regardless of adjacency.

We now formally define cell structures, transformations between them called cellular fillings, and the volume exchange metric and dissimilarity. 
A \emph{cell structure} in $\Omega\subset \R^{n}$ is a countable partition $\cellstrucP$ of $\Omega$ so that each cell $E \in \cellstrucP$ has positive Lebesgue measure, and the collection has finite interfacial surface energy (see Definition~\ref{defn:Caccioppoli} for a precise definition). We assume that $\Omega$ is open bounded open and connected, and we consider two cell structures to be equal if they differ up to a rearrangement of a set of measure zero.
We do not make assumptions about the topology of cells, but call them ``cells'' because of the proposed applications.

We introduce two different ways to ``fill'' the space between two $n$-dimensional cell structures $\cellstrucP, \cellstrucQ$ defined on the ends of a cylinder $\Omega\times \paren{0,1}$. 
Either notion defines an $(n+1)$-dimensional cell structure $\cellstrucW$ of $\Omega\times \paren{0,1}$.
To avoid pathologies, we introduce the notions of essential interior and essential boundary from geometric measure theory. 
With the notion of the \textit{\textbf{density of $E\subset \mathbb{R}^n$ at} $x$, 
\[
\Theta(E, x) \coloneq \lim_{r\to 0+}\frac{\mathscr{L}^n\paren{E\cap B(x, r)}}{\mathscr{L}^n\paren{B(x, r)}},
\]
}
we can define the \textbf{\textit{essential interior $E^{(1)}$ of}} $E$ consists of the set of points where the density of $E$ is $1$, and the \textbf{\textit{essential boundary $\partial^{\ast} E$ of}} $E$ consists of the points where the density of $E$ is not in $\{0, 1\}$.
In particular, if we define adjacency of cells via their essential boundaries, the adjacency is unchanged by rearranging a set of Lebesgue measure zero (for instance, adding a path connecting one cell to another).  

% The intersection  $\cellstrucW\cap \paren{\Omega\times t}$ is a cell structure on $\Omega\times t$ such that for each cell $E\in \cellstrucW\cap \paren{\Omega\times t}$ there exists a cell $E'\in \cellstrucW$ so that $E\subset \partial^{\ast}E'\cap \paren{\Omega\times t}$ and $\mathscr{L}^{n}\paren{\partial^{\ast}E'\cap \paren{\Omega\times t}\setminus E}=0.$ That is, we take all non-empty sets of the form  $\partial^{\ast}E\cap \paren{\Omega\times t}$ as $E$ ranges over the cells of $\cellstrucW$ and remove potential overlaps on their boundaries. 
%On the other hand,
%the type 2 intersection $\cellstrucW\cap_2 \paren{\Omega\times t}$ is the collection of cells obtained by splitting the cells of the type $1$ intersection into connected components. 
A \emph{one-to-one cellular filling} between cell structures $\cellstrucP$ and $\cellstrucQ$ is an $(n+1)$-dimensional cell structure $\cellstrucW$ in $\Omega\times \paren{0,1}$ so that for each cell $W \in \cellstrucW$ with $\mathscr{H}^n(\partial^{\ast} W\cap (\Omega\times 0)) > 0$, there exists a cell $E \in \cellstrucP$ such that 
\[
\mathscr{H}^n((E\times 0)\; \triangle\; \partial^{\ast} W \cap (\Omega\times 0)) = 0,
\]
and similarly for a cell $W \in \mathcal{W}$ with $\mathscr{H}^n(\partial^{\ast} W\cap (\Omega\times 1)) > 0$. 
% $\cellstrucW\cap_1 \Omega\times\set{0}=\cellstrucP$ and $\cellstrucW\cap_1 \Omega\times\set{1}=\cellstrucQ.$

Call two cells of a cell structure $\cellstrucP$ adjacent if the intersection of their essential boundaries has positive $(n-1)$-dimensional Hausdorff measure. 
A \textit{\textbf{coarsening}} of a cell structure $\cellstrucP$ is a cell structure $\cellstrucP'$ so that each $E' \in \cellstrucP'$ is a union of non-adjacent cells of $\cellstrucP.$   
A \emph{many-to-many cellular filling} between cell structures $\cellstrucP$ and $\cellstrucQ$ in $\Omega$ is a one-to-one cellular filling between coarsenings of $\cellstrucP$ and $\cellstrucQ.$ See Figure~\ref{fig:volume_flux_manymany}. 

Denote the collections of such fillings by $\Fill^{1}\paren{\cellstrucP,\cellstrucQ}$ and $\Fill^{2}\paren{\cellstrucP,\cellstrucQ},$ respectively. The restriction on one-to-one fillings may seem artificial as they do not allow a single connected cell to be transformed into two disjoint connected cells via a ``pair of pants,'' but as we will see below the resulting metric has nicer properties and is easier to compute.

\begin{figure}[t]
\center
\includegraphics[height=0.30\textwidth]{%
			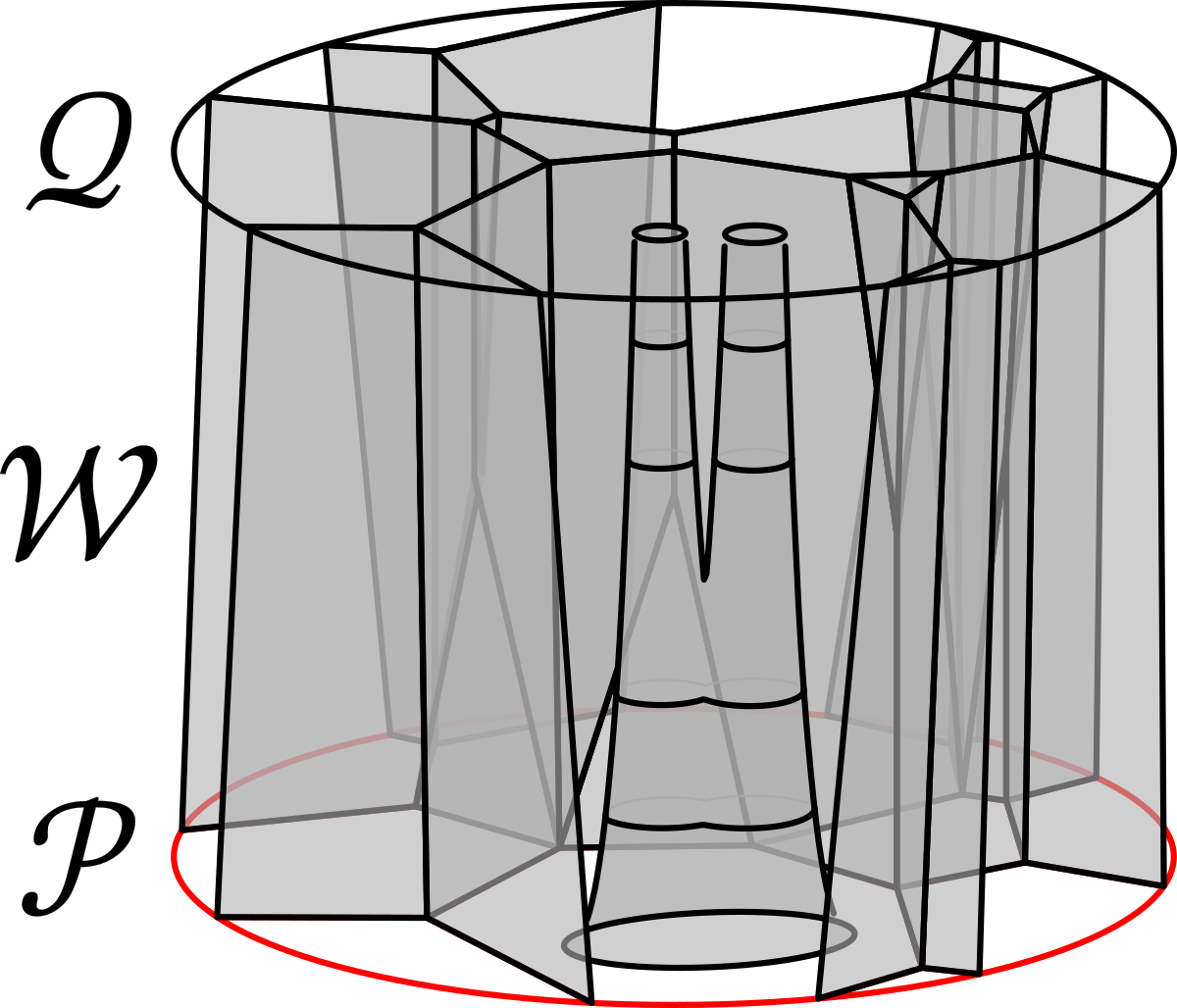} \quad
\includegraphics[height=0.30\textwidth]{%
			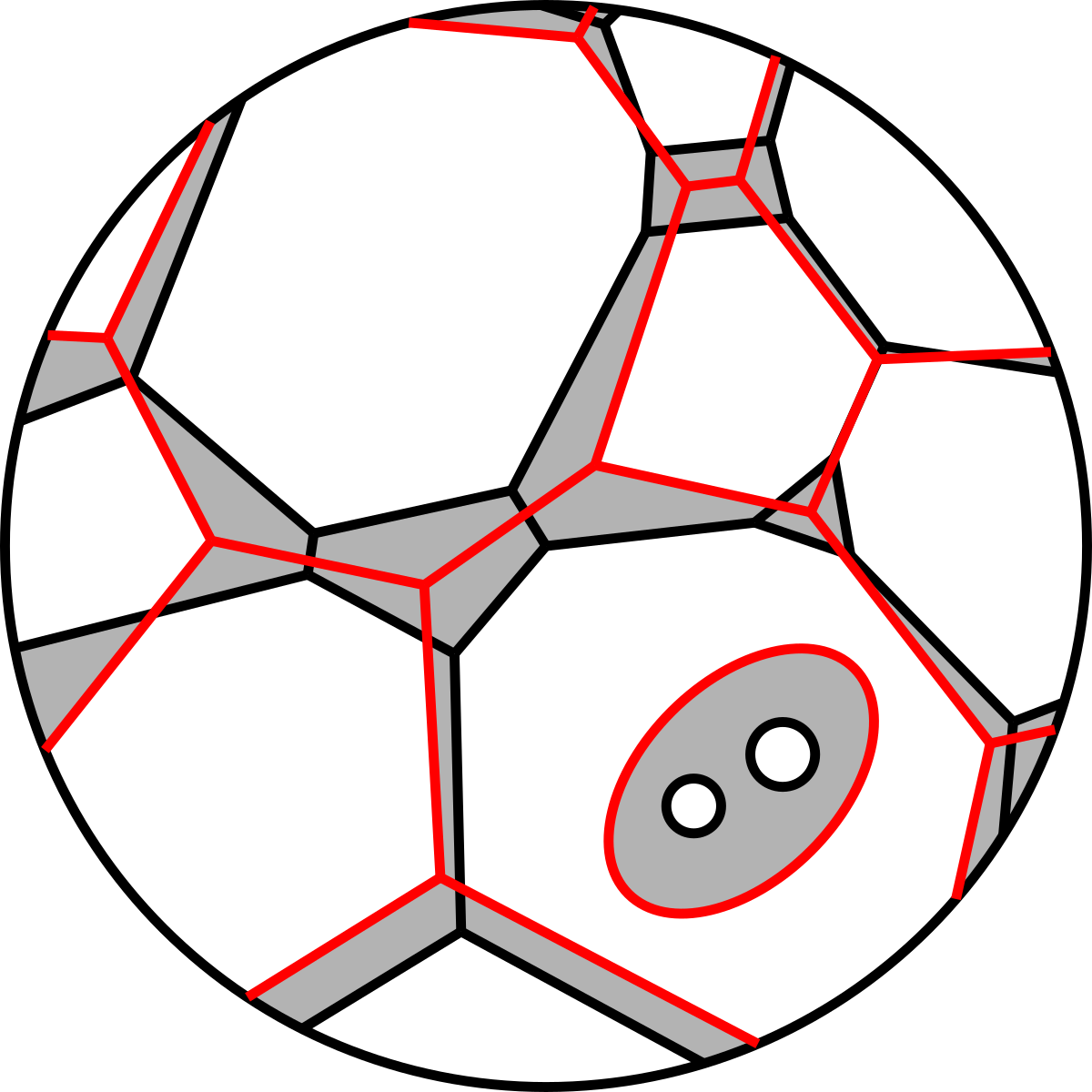}
\caption{\label{fig:volume_flux_manymany}
(Left) A many-to-many cellular filling $\cellstrucW$ between $\cellstrucP$ and $\cellstrucQ$ that is not a one-to-one cellular filling. Note that two cells are merged via a pair of pants. 
(Right) The projection into the plane %$\R^2 \times 0$
shows the volume exchange dissimilarity between $\cellstrucP$ and $\cellstrucQ$ as the area in grey.}
\end{figure}

Let $\cellstrucW$ be an $(n+1)$-dimensional cell structure in $\Omega\times\paren{0,1},$ let $\partial^{\circ} \cellstrucW$ the union of the essential boundaries of the elements of the partition $\cellstrucW,$ in the open cylinder, i.e., 
\[
\partial^{\circ} \cellstrucW = \bigcup_{W\in \cellstrucW} \partial^{\ast} W \cap (\Omega\times(0,1))
\]
and let $\pi: \Omega\times\paren{0,1} \to \Omega\times 0$ be the projection. 
We define the \textit{\textbf{volume exchange}} of a filling $\cellstrucW$ by
\begin{equation}\label{def:projected-area}
\VolEx\brac{\cellstrucW} = \int_{\Omega\times 0} \mathscr{H}^0(\partial^\circ \cellstrucW  \cap \pi^{-1}(x))\,d\mathscr{L}^n(x) = \int_{\partial^{\circ} \cellstrucW} |\langle \verticalvec, {\unitnormal}(x)\rangle|\,d\mathscr{H}^n(x),
\end{equation}
where $\unitnormal(x)$ is the unit normal vector to the approximate tangent space $\mathrm{Tan}^n(\partial^{\ast} \cellstrucW, x)$ (which coincides with the usual tangent space where $\partial^{\ast} \cellstrucW$ is smooth), $\verticalvec= (0, \dots, 0, 1),$ $\mathscr{L}^n$ is the Lebesgue measure, and $\mathscr{H}^n$ is the Hausdorff measure. 
The equality of the two expressions above follows from the area formula (see \cite[Remark 2.94]{ambrosio2000functions} or \cite[3.2.22]{federer2014geometric}).

\begin{definition}\label{defn:continuous}
Let $\cellstrucP,\cellstrucQ$ be cell structures in $\Omega.$ The \textbf{one-to-one volume exchange metric} is
$$\metric_{\mathrm{V}}^1\paren{\cellstrucP,\cellstrucQ}=\min_{\cellstrucW\in \Fill^1\paren{\cellstrucP,\cellstrucQ}} \VolEx\brac{\cellstrucW}$$
and the \textbf{many-to-many volume exchange dissimilarity is}
$$\metric_{\mathrm{V}}^2\paren{\cellstrucP,\cellstrucQ}=\min_{\cellstrucW\in \Fill^2\paren{\cellstrucP,\cellstrucQ}} \VolEx\brac{\cellstrucW}$$
\end{definition}

The existence of these minima will be shown in Proposition~\ref{prop:min} below as a corollary of our main theorem.

We now define discrete analogues of the two volume exchange metrics.
\begin{definition}\label{def:partial-matching}
A \textbf{partial matching} between two sets $\cellstrucP$ and $\cellstrucQ$ is a subset $\partialmatching$ of $\paren{\cellstrucP\cup \{\ast\}} \times \paren{\cellstrucQ \cup \{\ast\}}$ such that each cell of $\cellstrucP$ is matched to exactly one element of $\paren{\cellstrucQ \cup \{\ast\}}$ and each cell of $\cellstrucQ$ is matched to exactly one element of $\paren{\cellstrucP\cup \{\ast\}}.$

We denote the collection of partial matchings between $\cellstrucP$ and $\cellstrucQ$ as $\PM\paren{\cellstrucP, \cellstrucQ}$.
\end{definition}
See Figure~\ref{fig:discreteVE}. Note that $\PM\paren{\cellstrucP, \cellstrucQ} = \PM\paren{\cellstrucQ, \cellstrucP}$. 

\begin{definition}\label{def:independent-matching}
%Call two cells of a cell structure $\cellstrucP$ adjacent if the intersection of their essential boundaries has positive $(n-1)$-dimensional Hausdorff measure. A \textbf{coarsening} of a cell structure $\cellstrucP$ is a cell structure $\cellstrucP'$ so that each $E' \in \cellstrucP'$ is a union of non-adjacent cells of $\cellstrucP.$   
An \textbf{independent matching} between $\cellstrucP$ and $\cellstrucQ$ is a partial matching between coarsenings of $\cellstrucP$ and $\cellstrucQ.$
% is a subset $R$ of $(V_1\cup \{\ast\}) \times (V_2 \cup \{\ast\})$ such that  
% \begin{enumerate}
% \item[(a)] the canonical projections $R \to V_1$ and $R \to V_2$ are onto, and 
% \item[(b)] $R \cap (V_1\times V_2)$ is independent in the sense that if $\paren{v_1,w},\paren{v_2,w}\in R \cap (V_1 \times V_2)$ then $v_1$ and $v_2$ are not connected by an edge of $G_1,$ and if $\paren{v,w_1},\paren{v,w_2}\in R \cap (V_1 \times V_2)$ then $w_1$ and $w_2$ are not connected by an edge of $G_2.$  
% \end{enumerate}
\end{definition}

For cell structures $\cellstrucP, \cellstrucQ\subset \Omega$ denote by $\PM\paren{\cellstrucP,\cellstrucQ}$ the collection of partial matchings between the cells of $\cellstrucP$ and $\cellstrucQ$ and by $\IM\paren{\cellstrucP,\cellstrucQ}$ the collection of independent matchings of the cells of $\cellstrucP$ and $\cellstrucQ.$% when viewed as vertices of the adjacency graph on the cells (where two cells are adjacent if the intersection of their essential boundaries has positive $(n-1)$-dimensional Hausdorff measure). For $\partialmatching\in\PM\paren{\cellstrucP,\cellstrucQ}$ and $\paren{E,F}\in\partialmatching 

\begin{definition}\label{defn:discrete}
Let $\cellstrucP,\cellstrucQ$ be cell structures in $\Omega$.
The \textbf{one-to-one discrete volume exchange metric} between $\cellstrucP$ and $\cellstrucQ$ is defined as 
\begin{equation}\label{def:discrete-volume-exchange-metric}
\metric^1_{\mathrm{DV}}(\cellstrucP, \cellstrucQ) \coloneq \frac{1}{2}\min_{\partialmatching \in \PM(\cellstrucP, \cellstrucQ)}\sum_{(E, F) \in \partialmatching} \mathscr{L}^n\paren{E \triangle F},
\end{equation}
where $E\triangle F \coloneq (E\setminus F) \cup (F\setminus E)$ denotes the \emph{symmetric difference} between sets $E$ and $F$ and we treat $\ast$ as the empty set.  
Similarly, the \textbf{many-to-many discrete volume exchange dissimilarity} is 
\begin{equation}\label{def:discrete-volume-exchange-diss}
\metric^2_{\mathrm{DV}}\paren{\cellstrucP, \cellstrucQ} \coloneq \frac{1}{2}\min_{\partialmatching \in \IM(\cellstrucP, \cellstrucQ)}\sum_{(E', F') \in \partialmatching} \mathscr{L}^n\paren{E' \triangle F'}.
\end{equation}
\end{definition}
We show that the minima are attained when $\cellstrucP$ and $\cellstrucQ$ have countably many cells in Proposition~\ref{prop:min}. 
As we explain below, $\metric^1_{\mathrm{DV}}\paren{\cellstrucP,\cellstrucQ}$  can be computed efficiently by converting to an equivalent linear assignment problem.

\begin{figure}[t]
\center
\includegraphics[height=0.30\textwidth]{%
			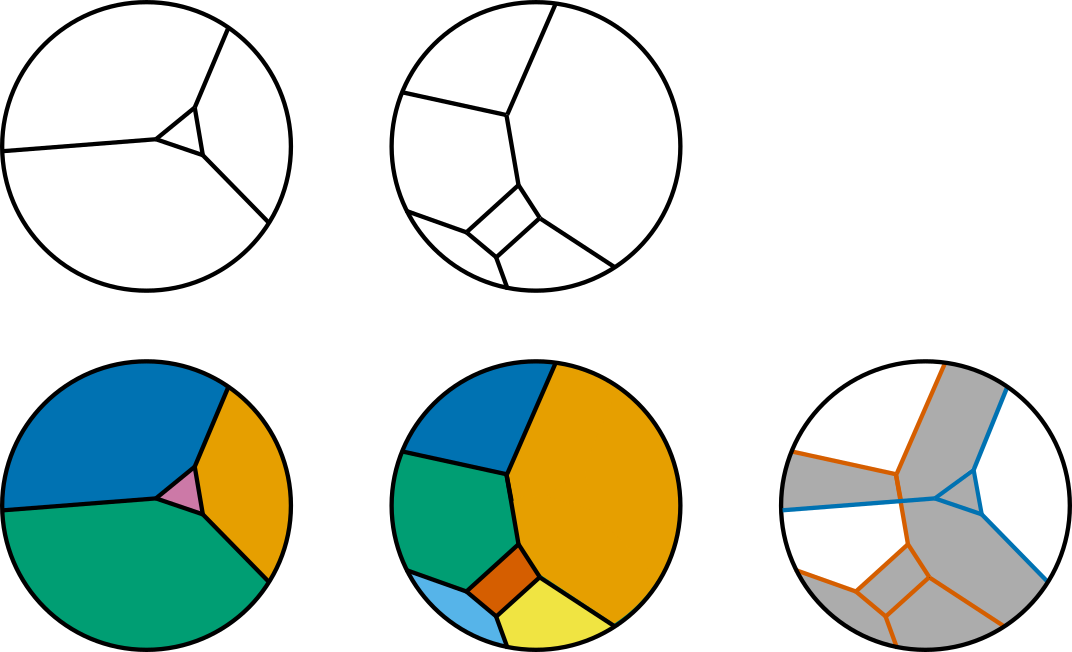} 
\caption{\label{fig:discreteVE}
The optimal partial matching between the cell structures in the top row is illustrated by the colors in the bottom row. That is, the dark blue, orange, and green cells are matched with each other, whereas the pink, light blue, yellow, and red cells are matched to $\ast.$ The cost of the matching is the area of the shaded region on the right. The optimal independent matching would also pair the yellow cell with the green cell, reducing the cost by the area of the former.}
\end{figure}

Our main theorem is the equality of the continuous and discrete volume exchange metrics.

\begin{theorem}\label{thm:main}
Let $\Omega$ be a bounded, connected, open subset of $\R^n$ and let $\cellstrucP,\cellstrucQ$ be cell structures in $\Omega.$ Then
$$\metric_{\mathrm{V}}^1\paren{\cellstrucP,\cellstrucQ}=\metric^1_{\mathrm{DV}}\paren{\cellstrucP, \cellstrucQ}$$
and
$$\metric_{\mathrm{V}}^2\paren{\cellstrucP,\cellstrucQ}=\metric^2_{\mathrm{DV}}\paren{\cellstrucP, \cellstrucQ}\,.$$
\end{theorem}

The proof of Theorem~\ref{thm:main} relies heavily upon the tools of geometric measure theory, a field that arose from the study of Plateau's problem on minimal surfaces and its generalizations. 
Towards that end, we convert the statement into one concerning a distance between piecewise-constant functions, where a piecewise-constant function is one whose non-empty level sets form the cells of a cell structure. 
Note that the assignment of labels in a metric space $\paren{Y,\delta}$ to the cells of a cell structure $\cellstrucP$ induces to a piecewise-constant function $p:\Omega\to Y$, which is well-defined up to a set of Lebesgue measure zero. 
Given a partial matching $\partialmatching$ between the cells of $\cellstrucP$ and those of $\cellstrucQ,$ we may choose $Y$ to be a discrete metric space and label the cells of $\cellstrucP$ and $\cellstrucQ$ so that matched cells have the same labels. 
In Lemma~\ref{lem:l1_sym_diff} we show
$$\frac{1}{2}\sum_{(E, F) \in \partialmatching} \mathscr{L}^n\paren{E \triangle F}=\int_\Omega \delta\paren{p,q}\,d\mathscr{L}^n\,.$$
On the other hand, given two piecewise-constant functions  $p\colon\Omega\to Y$ and $q\colon\Omega\to Y$, a \textit{\textbf{piecewise-constant filling}} between $p$ and $q$ is a piecewise-constant function  $w\colon\Omega\times\paren{0,1}\to Y$ so that the boundary trace (the ``approximate limits'' of $w$ on the boundary) satisfy $w\mid_{\Omega\times 0}=p$ and $w\mid_{\Omega\times 1}=q,$ taken up to equivalence in the $L^1$ norm. 
There is a correspondence between $Y$-labeled cellular fillings between $\cellstrucP$ and $\cellstrucQ$ and piecewise-constant fillings between functions $p$ and $q$ which are constant on the grains of $\cellstrucP$ and $\cellstrucQ$ respectively. 
We define the energy of $w$ as follows. 
We may write $\Omega\times\paren{0,1}=C\cup S\cup F$ where $w$ is asymptotically almost everywhere constant in small neighborhoods of points in $C,$ where $w$ takes two values $w^+\paren{x}$ and $w^-\paren{x}$ asymptotically almost everywhere in  small neighborhoods of $x\in S,$ and $\mathscr{H}^{n}\paren{F}=0.$ Furthermore,  $S$ is $\mathscr{H}^{n}$-rectifiable with a well-defined approximate unit normal $\unitnormal_w\colon S\to \mathbb{S}^{n}$. 
We define the \textbf{\textit{projected surface energy}} of $w$ as 
\begin{equation}\label{def:surface-energy}
\VolEx_\delta[w] = \int_{S} \delta\paren{w^+, w^-}|\langle {\verticalvec}, \unitnormal_w\rangle|\,d\mathscr{H}^n
\end{equation}
where $\mathbf{v} = (0, \dots, 0, 1)$ is the unit vector in the ``vertical'' direction.  

%Recall that a metric space is proper if closed balls are compact.
% \begin{theorem}\label{thm:piecewise}
% Suppose that $\paren{Y,\delta}$ is a separable, proper metric space. 
% Let $g:\Omega\to Y$ and $h:\Omega\to Y$ be piecewise constant, and denote by $\PC\paren{g,h}$ the collection of piecewise-constant fillings between them. Then 
% $$\min_{u\in \PC\paren{g,h}} \VolEx_\delta[u] = \|\delta\paren{g,h}\|_{{L^1}(\Omega)}$$
% and the minimum exists.
% \end{theorem}

\begin{theorem}\label{thm:l1_energy_for_metric_BV}
Let $\Omega$ be a bounded, connected, open subset of $\mathbb{R}^{n}$ and let $(Y, \delta)$ be a complete, separable metric space. 
For any two piecewise constant functions $p, q\in \PC(\Omega; Y)$ on $\Omega$, 
\[
%\min_{u \in \mathsf{F}(p, q)}\VolEx{[\delta(u^+, u^-); S_u]} = \int_\Omega \delta(p, q)\,d\mathscr{L}^{n-1}
\min_{w \in \mathsf{F}(p, q)}\VolEx_\delta[w]  = \int_\Omega \delta\paren{p, q}\,d\mathscr{L}^{n}
\]
and the minimum exists, where $\mathsf{F}(p, q)$  denotes the collection of piecewise-constant fillings between $p$ and $q.$ 
\end{theorem}

%For polycrystalline materials, one can take $Y$ to be the double quotient of $SO(3)$ by the point symmetry group of the crystalline lattice and $\delta$ to be a misorientation energy function.
%The preceding theorem implies that the minimum energetic cost to transform one cell structure into another is given by the $L^1$ norm of the misorientation energy.

%Later, we define the cost of a piecewise constant filling between piecewise constant functions in a way consistent with Equation~\ref{eq:def:volume-exchange}. Then, to prove Theorem~\ref{thm:main}, it will suffice to show that  the minimum cost of a filling between $g_M$ and $g_M$ equals the $L^1$ distance between those functions. In fact, we prove a stronger result that allows for the choice of different surface energy functions.

\subsection{Outline}

We begin in Section~\ref{sec:prelim} by reviewing concepts from metric geometry and geometric measure theory including the properties of scalar-valued functions of bounded variation (BV) and sets of finite perimeter. 
Section~\ref{sec:metric-valued-BV-functions} develops the theory of metric space valued BV functions, with the beginning and Section~\ref{subsec:energywelldefined} reviewing the concept of an approximate limits~\cite{ambrosio1990metric} and defining the terms appearing in Equations~\ref{def:projected-area} and ~\ref{def:surface-energy}. 
A key technical result is Theorem~\ref{thm:general-bv-statement} in Section~\ref{sec:horizontal}, which states that if $\paren{Y,\delta}$ is a complete, separable metric space, $\Omega$ is a bounded, open, connected subset of $\R^n,$ and $p,q:\Omega\to Y$ are BV functions then
\begin{equation*}
\min_{\substack{w \in \BV(\Omega\times  I; Y)\\
w|_{\Omega\times 0} = p\text{ a.e.}\\
w|_{\Omega\times 1} = q\text{ a.e.}}}\int_\Omega |Dw_x|(0,1)\,d\mathscr{L}^n(x) = \int_{\Omega}\delta(p, q)\,d\mathscr{L}^n.
\end{equation*}
In the special case where $w$ is a piecewise-constant filling between $p$ and $q,$ the infimand on the left hand side agrees with the projected surface area defined in Equation~\ref{def:surface-energy}. 
As such, this is a key step towards proving Theorems~\ref{thm:l1_energy_for_metric_BV} and~\ref{thm:main}.
Theorem~\ref{thm:general-bv-statement} is proven by first showing that the projected surface area of $w$ is bounded below by the integral of the total variation of one-dimensional slices $w_x\paren{t}=w\paren{x,t}$ which is in turn bounded by the $L^1$ distance between $p$ and $q.$ 
Then, this lower bound is achieved by the filling $w_0:\Omega\times I\to Y$ defined by
$$w_0\paren{x,t}=\begin{cases} p\paren{x} & t\leq 1/2 \\ q\paren{x} & t>1/2\end{cases}.$$ Note that if $p,q$ are piecewise constant then so is $w_0.$ 

Three types of piecewise constant fillings between piecewise constant functions are defined in Section~\ref{sec:volume}. 
These lead to versions of the one-to-one volume exchange metric and many-to-many volume exchange dissimilarity in this setting, which are shown to agree with the notions in the introduction in Proposition~\ref{prop:VEFunction}. 
This section also includes a proof of Theorem~\ref{thm:l1_energy_for_metric_BV}. 

Section~\ref{sec:volume_discrete} begins with a proof of the main result of the paper, Theorem~\ref{thm:main}. 
Next, the metric properties of the one-to-one volume exchange metric and many-to-many volume exchange dissimilarity are explored in Section~\ref{subsec:metric}.
Sections~\ref{subsec:dual} and~\ref{subsec:compute11} explain how a dual formulation of the discrete one-to-one volume exchange metric can be computed efficiently as a linear assignment problem and Section~\ref{subsec:computemm} outlines an algorithm to compute the many-to-many volume exchange dissimilarity.
Finally, Section~\ref{sec:comp_exp} describes computational experiments applying these dissimilarities to a dataset of synthetic microstructures.

\section{Mathematical Preliminaries}\label{sec:prelim}
% To define and prove theorems about the discrete and continuous volume exchange, we use concepts from metric geometry, and geometric measure theory; in particular, the theory of functions of bounded variation and sets of finite perimeter. 
In this section, we will review this material and provide references. 
Most of the material in this section can be found in Ambrosio, Fusco, and Pallara's monograph~\cite{ambrosio2000functions} on functions of bounded variation, or Maggi's book~\cite{maggi2012sets} on sets of finite perimeter. 

% \subsection{Metric geometry}
% In this section, we will discuss the notions of \emph{isometric metric spaces}, \emph{isometry groups}, their \emph{orbits}, and the natural isomorphism between subgroups of the isometry group of certain spaces of functions into a metric space $Y$, and the isometry group of $Y$. 
We will use these properties to show that the one-to-one discrete volume exchange is in fact a metric. 
Whenever $\delta$ satisfies symmetry, non-negativity, and $\delta\paren{y, y} = 0$ for all $y \in Y$, our function $\delta$ is called a \textit{\textbf{dissimilarity}}.

A \textit{\textbf{metric space} is a set $Y$ together with a \textbf{distance} (\textbf{metric}) $\delta$, a symmetric real-valued function defined on $Y\times Y$ such that $\delta\paren{y_1, y_2} = 0$ if and only if $y_1 = y_2$, and $\delta(y_1, y_2) \leq \delta\paren{y_1, y_3} + \delta\paren{y_3, y_2}$ for all $y_1, y_2, y_3 \in Y$}. 
A function $\varphi\colon Y \to Z$ from the metric space $Y$ into another metric space $(Z, \metric_Z)$ is called \textit{\textbf{distance preserving}} if 
\begin{equation}\label{def:distance-preserving}
\metric_Z(\varphi(y), \varphi(y')) = \delta(y, y') \quad \textit{ for all } y, y' \in Y.
\end{equation}
If in addition, $\varphi$ is \textit{surjective (and hence bijective)} we call $\varphi$ an \textit{\textbf{isometry}}, and say that $Y$ and $Z$ are \textit{\textbf{isometric}}.
The collection of all isometries from $Y$ to $Z$ is denoted by $\Iso(Y; Z)$; when $Z = Y$ we write $\Iso(Y)$ as opposed to $\Iso(Y; Y)$. 
% Because function composition defines a group operation on the family of isometries from $Y$ to itself, we call $\Iso(Y)$ the \textit{\textbf{isometry group of}} $Y$.

% A crucial fact we will be using for the proof of the equality between the discrete volume exchange metric and the continuous volume exchange metric, is the following.
% For any subgroup $G$ of $\Iso(Y)$ we can define a pseudometric $\metric_{Y/G}$ on the \textit{\textbf{orbits of}} $G$, \textit{the quotient set $Y/G = \{Gy : y \in Y\}$ consisting of equivalence classes of the form $Gy \coloneq \{\varphi(y) : \varphi \in G\}$}. 
% This pseudometric $\metric_{Y/G}$ is defined by
% \[
% \metric_{Y/G}(Gy, Gy') \coloneq \inf_{\phi,\varphi \in G}\delta (\varphi(y), \phi(y')) \quad \text{ for all } y, y' \in Y. 
% \]

% Suppose $F$ is a family containing functions from a bounded, connected, open set $\Omega\subset \mathbb{R}^n$ into $Y$, and suppose that $F$ is equipped with a metric $d_F$. 
% If $\varphi\circ f \in F$ and $\metric_F\circ (\varphi\times \varphi) = \metric_F$ for all $\varphi \in \Iso(Y)$, then the isometry group of $Y$ is naturally isomorphic to a subgroup of the isometry group of $F$. 
% This isomorphism $\Phi \colon \Iso(Y) \hookrightarrow \Iso(F)$ is defined as the mapping
% \begin{equation}
% \varphi \mapsto \Phi_\varphi, \quad \text{ where } \quad \Phi_\varphi(f) = \varphi\circ f \quad \text{ for all } f \in F.
% \end{equation}

% We let $\Iso(Y)$ denote the subgroup in $F$.

% \subsection{Scalar-valued \texorpdfstring{$\BV$}{BV}-functions}
Let us go through some terminology and notations from measure theory and scalar-valued functions of bounded variation ($\BV$-functions) that we will liberally use throughout the paper.
We refer the reader to \cite{ambrosio2000functions} for a thorough treatment of these topics. 

We let $\mathfrak{B}(\mathbb{R}^n)$ denote the \textit{\textbf{Borel $\sigma$-algebra} of $\mathbb{R}^n$}, the smallest $\sigma$-algebra on $\mathbb{R}^n$ containing all of the open sets. 
A \textit{\textbf{Borel measure}} is a $\sigma$-additive set function $\mu\colon \mathfrak{B}(\mathbb{R}^n)\to [0, +\infty]$ such that $\mu(\emptyset) = 0$.
If in addition, $\mu(\mathbb{R}^n) < \infty$, then $\mu$ is called a \textit{\textbf{finite Borel measure}}.
If a positive Borel measure $\mu$ is finite on compact sets, it is called a \textit{\textbf{Radon measure}}. 
Similarly, an \textit{\textbf{$\sigma$-additive set function}} $\mu\colon \mathfrak{B}(\mathbb{R}^n) \to \mathbb{R}^m$ (with $\mu(\emptyset) = 0$) is called a \textit{\textbf{vector measure}} on $\mathbb{R}^n$. 
We let $\mathcal{M}(\mathbb{R}^n; \mathbb{R}^m)$ denote the collection of all vector ($\mathbb{R}^m$-valued) measures in $\mathbb{R}^n$.
The \textit{\textbf{total variation}} $|\mu|$ of $\mu$ is defined by $|\mu|(B) \coloneq \sup \sum_{k=1}^\infty |\mu(B_k)|$, where the supremum is taken over all partitions $\{B_k\}$ of $B$ such that each $B_k$ is a Borel set.
Importantly, for any $\mu \in \mathcal{M}(\mathbb{R}^n; \mathbb{R}^m)$, $\mu(\emptyset) = 0$ and its total variation $|\mu|$ is a finite Borel measure on $\mathbb{R}^n$.

We will often use of the \textit{\textbf{$n$-dimensional Lebesgue measure}} $\mathscr{L}^n$ in $\mathbb{R}^n$, the \textit{\textbf{$k$-dimensional Hausdorff measure}} $\mathscr{H}^k$, and frequently, the Hausdorff measure $\mathscr{H}^k\allardell \Sigma$ restricted to ``well-behaved'' subsets $\Sigma$ of $\mathbb{R}^n$, defined as follows.
For any measure $\mu$ and Borel set $E$ in $\mathbb{R}^n$, we define a Borel measure $\mu\allardell E(A) \coloneq \mu(E\cap A)$, for any Borel set $A\subset \mathbb{R}^n$. 

Let $\Omega$ be a bounded, connected, open subset of $\mathbb{R}^n$. 
We let $L^1(\Omega)$ denote the collection of all $\mathscr{L}^n$-measurable, real-valued functions $f$ on $\Omega$ with finite \textit{\textbf{$L^1$-norm}}:
$\|f\|_{L^1(\Omega)} \coloneq \int_\Omega |f|\,d\mathscr{L}^n < +\infty$.
We let $\dashint_A f\,d\mu$ denote the average $\frac{1}{\mu(A)}\int_A f\,d\mu$.
We turn $L^1(\Omega)$ into a metric space by identifying two functions together when the $L^1$-norm of their difference is zero; equivalently, when the two functions differ on a set of $\mathscr{L}^n$-measure zero.
\textit{If for each $j = 1, \dots, n$, there exists a signed \emph{finite} Radon measure $D_jf$ in $\Omega$ satisfying
\begin{equation}\label{def:BV-function}
\int_\Omega f\,\partial_j \phi\,d\mathscr{L}^n = -\int_\Omega \phi\,dD_j f \quad \text{ for all } \phi \in C_c^\infty(\Omega),
\end{equation}}
then $f$ is called a \textbf{\textit{$\BV$-function}} on $\Omega$, the vector-valued measure $Df = (D_1f, \dots, D_nf)$ the \textbf{\textit{weak-derivative of}} $f$, and $D_jf$, the weak-derivative of $f$ in the $j$-th (or $\boldsymbol{e}_j$-th) direction. 
We denote the set of all $\BV$-functions in $\Omega$ by $\BV(\Omega)$.

A subset $S$ of $\mathbb{R}^{n+1}$ is called \textit{\textbf{countably $\mathscr{H}^n$-rectifiable}} if $S$ is contained in the union of (at most) countably many images of Lipschitz functions from $\mathbb{R}^n$ to $\mathbb{R}^{n + 1}$, up to an $\mathscr{H}^n$-negligible set. 
For $\mathscr{H}^n$-a.e.~$x \in S$ there exists an $n$-dimensional (vector) subspace $V\subseteq \mathbb{R}^{n+1}$ called the \textit{\textbf{approximate tangent space}} of $S$ at $x$, the $n$-dimensional subspace $V$ of $\mathbb{R}^{n+1}$ such that 
\begin{equation}
\lim_{r \to 0} \int_{S_{x, r}} \phi\,d\mathscr{H}^n = \int_V \phi\,d\mathscr{H}^n\quad \text{ for all } \phi \in C_c(\mathbb{R}^{n+1}; \mathbb{R}),
\end{equation}
where $S_{x, r} = (S - x)/r$ for any $x \in S$ and $r > 0$, and denoted by $\mathrm{Tan}^n(S, x)$. 

If $S$ is a countably $\mathscr{H}^{n}$-rectifiable subset of $\mathbb{R}^{n+1}$, we say that $S$ is \textit{\textbf{oriented}} by a Borel map $\unitnormal\colon S \to \mathbb{S}^{n}$ if $\mathrm{Tan}^n(S, x) = \unitnormal^\perp(x)$ for $\mathscr{H}^{n}$-a.e.~$x\in S$. 

\subsection{Sets of finite perimeter}

For any $\mathscr{L}^n$-measurable subset $E$ of $\mathbb{R}^n$ the \textit{\textbf{perimeter of $E$ in $\Omega$} is denoted by $P\paren{E; \Omega}$ and defined as 
\[
P\paren{E; \Omega} \coloneq \sup\Big\{\int_E \mathrm{div}{\phi}\,dx : \varphi \in C_c^1(\Omega; \mathbb{R}^n),\ \|\phi\|_\infty \leq 1\Big\}.
\]}

\begin{remark}
The class of sets of finite perimeter in $\Omega$ includes all sets $E$ with $C^1$ boundary inside $\Omega$ such that $\mathscr{H}^{n-1}(\Omega\cap \partial E) < \infty$. 
For such sets, $\mathscr{H}^{n-1}(\Omega \cap \partial E) = P(E, \Omega)$.
\end{remark}  

We will now review a few notions of measure theoretic ways of capturing the boundary, interior, and exterior for a set of finite perimeter. 
The \textit{\textbf{reduced boundary}} of a set of finite perimeter $E$ in $\mathbb{R}^n$ is \textit{the collection $\partial_{\ast} E$ of all points $x \in \R^n$ such that the limit 
\[
\unitnormal_E(x)\coloneq \lim_{r\to 0+} \frac{D\chi_E(B(x, r))}{|D\chi_E|(B(x, r))}
\]
exists in $\mathbb{R}^n$ and satisfies $|\unitnormal_E(x)| = 1$}. 
For each $t \in [0, 1]$, we let $E^{(t)}$ denote the collection of all points $x \in \mathbb{R}^n$ of density $t$, i.e., 
\begin{equation}\label{def:density}
E^{(t)} \coloneq \Big\{x \in \mathbb{R}^n : \lim_{r \to 0+} \frac{\mathscr{L}^n(E\cap B(x, r))}{\mathscr{L}^n(B(x, r))} = t\Big\}.
\end{equation}
We call $E^{(1)}$ the \textit{\textbf{measure theoretic interior}}, and $E^{(0)}$ the \textit{\textbf{measure theoretic exterior}}, and the set $\partial^\ast E \coloneq \mathbb{R}^n \setminus (E^{(0)}\cup E^{(1)})$ the \textit{\textbf{essential boundary}} of $E$.
Federer showed (see~\cite[Theorem~3.61]{ambrosio2000functions}) that 
\[
\partial_{\ast} E \subset E^{(1/2)} \subset \partial^\ast E \subset N\cup \partial_{\ast} E
\]
for some $\mathscr{H}^{n-1}$-negligible $N\subset \mathbb{R}^n$.

\begin{definition}\label{defn:Caccioppoli}
A \textit{\textbf{Caccioppoli partition} of $\Omega$ indexed by $I\subset \mathbb{N}$ is a partition $\mathcal{E} = \{E_i\}_{i \in I}$ of $\Omega$ such that $\sum_{i \in I} P(E_i, \Omega) < \infty$}. We assume that $\abs{E_i}>0$ for all $i\in I$ and consider two partitions to be identical up to re-ordering of the indices or the redistribution of sets of measure zero.
\end{definition}

% Let us now state a part of \cite[Proposition~5.9]{ambrosio2000functions}, which will be useful to us for defining our weakest type of filling. 
% \begin{theorem}
% If  $u\colon \Omega \subset \mathbb{R}^n \to \mathbb{R}^m$ is a piecewise constant function, then $Du$ is concentrated on the approximate jump set $J_u$ and $\mathscr{H}^{n-1}(J_u) < \infty$. 
% Moreover, there exists a set $N \subset \mathbb{R}^n$ with $\mathscr{H}^{n-1}(N) = 0$ such that 
% \begin{equation}
% \bigcup_{p \neq q} \partial^{\ast} \{u = p\} \cap \partial^{\ast} \{u = q\}\cap \Omega \quad\subset\quad J_u \quad\subset\quad N \cup \bigcup_{p \neq q} \partial^{\ast} \{u = p\} \cap \partial^{\ast} \{u = q\} \cap \Omega,
% \end{equation}
% and
% \[
% (u^+(x), u^-(x), \unitnormal_u(x)) = (p, q, \unitnormal_{\{u = p\}}(x))
% \]
% for any $x \in \partial^{\ast} \{u = p\} \cap \partial^{\ast} \{u = q\} \cap \Omega$. 
% \end{theorem}

\section{Metric Space Valued BV Functions}\label{sec:metric-valued-BV-functions}

Throughout this section (unless otherwise stated) we let $\Omega$ be a bounded, connected, open set in $\mathbb{R}^n$, we let $I = (0, 1)$, and we let $(Y, \delta)$ be a complete, separable metric space. 
Recall from the introduction that the surface energy of a piecewise-constant filling $w$ from $\Omega\times\paren{0,1}$ into $Y$ is
$$\ProjE_\delta[w] = \int_{S} \delta\paren{w^+,w^-}|\langle {\verticalvec}, \unitnormal\rangle|\,d\mathscr{H}^n$$
where $S$ is a set consisting of points $x_0$ so that $w$ limits to two values $w^+\paren{x},w^-\paren{x}$ in sufficiently small neighborhoods of $x.$ 
A major goal of this section is to show in Theorem~\ref{thm:approximate-discontinuity-rectifiability} that this expression is well-defined for a class of \emph{metric space valued BV functions}. 
This strengthens an earlier result by Ambrosio~\cite{ambrosio1990metric} which required the metric space $Y$ to be separable and proper. 
Later, we will apply this theorem to a countable discrete metric space which is not proper. 

Before we begin Section~\ref{subsec:verticalrestriction}, we review notions of approximate limit, one-sided trace, boundary trace, and define our notion of a metric space valued BV function, first introduced by Ambrosio in \cite{ambrosio1990metric}. 
Afterwards, we show a version of the Lebesgue differentiation theorem for functions in $L^1_{\mathrm{loc}}(\Omega; Y)$, \textit{the class of all measurable functions $f\colon \Omega \to Y$ such that every point of $\Omega$ has an open neighborhood $U\subset \Omega$ satisfying 
\[
\int_U \delta\paren{f, y_0}\,d\mathscr{L}^n < +\infty
\]
for some $y_0 \in Y$.} 
The class of functions for which $U$ can be taken to be $\Omega$ is denoted by $L^1(\Omega; Y).$ 
We include a proof of this fact as an introduction to the arguments we frequently use in \S\ref{sec:metric-valued-BV-functions} and \S\ref{sec:volume}.
In Section~\ref{subsec:verticalrestriction}, we show that the boundary trace of a metric space valued BV function $f$ on $\Omega\times\paren{0,1}$ agrees with those of the vertical restriction $f_x=f\paren{x,\cdot}$ for almost every $x\in \Omega.$ 
This is applied in Section~\ref{sec:horizontal} to prove the second major result of this section, Theorem~\ref{thm:general-bv-statement}. 
As described in the outline, this is a major step towards proving the main theorems of the paper. 
Finally, Section~\ref{subsec:energywelldefined} includes a proof of Theorem~\ref{thm:approximate-discontinuity-rectifiability} which was discussed in the preceding paragraph. 

% The main result towards that end is a boundary trace theorem which states that $\Omega\times\paren{0,1}$ can be decomposed as $C\cup J\cup S$ where the approximate limit of $w$ is well-defined on $C,$ $\mathscr{H}^n\paren{S}=0,$ and for each $x\in S$ the \emph{boundary traces} $w^+\paren{x}, w^-\paren{x}\in Y$ and \emph{approximate normal vector} $\unitnormal(j)\in S^{n}$ are defined so that 
% $$\lim_{r\to 0+} \dashint_{\Omega\cap B_{\unitnormal}^+(x_0, r)} \delta(w, a_0)\,d\mathscr{L}^{n+1} = 0 \qquad \text{ and } \qquad \lim_{r\to 0+} \dashint_{\Omega \cap B_{\unitnormal}^-(x_0, r)} \delta(w, b_0)\,d\mathscr{L}^{n+1} = 0$$ 
%  where $B_{\unitnormal}^{+}(x_0, r) \coloneq B(x_0, r) \cap H_{\unitnormal}^+$ and $B_{\unitnormal}^{-}(x_0, r) = B(x_0, r)\cap H_{\unitnormal}^-$ are the two (open) halves of the ball $B(x_0, r)$ defined by its intersection with the open half-spaces $H_{\unitnormal}^+ \coloneq \{x : \langle x - x_0, \unitnormal\rangle > 0\}$ and $H_{\unitnormal}^- \coloneq \{x  : \langle x - x_0, \unitnormal\rangle < 0\}$. 

We mostly use the same definitions as in Ambrosio's seminal work on metric space valued BV functions~\cite{ambrosio1990metric}, but slightly differ on the notion of an \emph{approximate limit}. 
We say that a Borel function $w\colon \Omega\to Y$ has an \textit{\textbf{approximate limit}} at $x_0 \in \Omega$ if there exists a $y_0 \in Y$ such that
\begin{equation}\label{def:L1-approximate-limit}
\lim_{r\to0+}\dashint_{\Omega \cap B(x_0, r)}\delta (w, y_0)\,d\mathscr{L}^n = 0.
\end{equation}
If the approximate limit $y_0$ exists, it is unique, and will be denoted by $\aplim_{x\to x_0}w(x)$.
The Borel set $S_w$ of points in $\Omega$ where $w$ does not satisfy condition~(\ref{def:L1-approximate-limit}) is called the \textit{\textbf{approximate discontinuity set}}.
Whenever the limit $(\ref{def:L1-approximate-limit})$ holds for a point $x_0 \in \partial \Omega$, we call $y_0$ the \textit{\textbf{boundary trace of}} $w$ at $x_0$.

If the boundary trace of $w$ exists for $\mathscr{H}^{n-1}$-a.e.~$x_0 \in \Sigma\subset \partial \Omega$, we let $w|_\Sigma$ denote the function $w|_\Sigma(x_0) = \aplim_{x\to x_0}w(x)$.

A variation to the notion of an approximate limit is that of a one-sided trace, and a boundary trace.  
For a fixed unit vector $\unitnormal \in \mathbb{S}^{n-1}$ we say that $a_0 \in Y$ is the \textit{\textbf{positive one-sided trace of}} $w$ and $b_0 \in Y$ is the \textit{\textbf{negative one-sided trace of}} $w$ if 
\begin{equation*}\label{def:one-sided-traces}
\lim_{r\to 0+} \dashint_{\Omega\cap B_{\unitnormal}^+(x_0, r)} \delta(w, a_0)\,d\mathscr{L}^n = 0 \qquad \text{ and } \qquad \lim_{r\to 0+} \dashint_{\Omega \cap B_{\unitnormal}^-(x_0, r)} \delta(w, b_0)\,d\mathscr{L}^n = 0
\end{equation*}
where $B_{\unitnormal}^{+}(x_0, r) \coloneq B(x_0, r) \cap H_{\unitnormal}^+$ and $B_{\unitnormal}^{-}(x_0, r) = B(x_0, r)\cap H_{\unitnormal}^-$ are the two (open) halves of the ball $B(x_0, r)$ defined by its intersection with the open half-spaces $H_{\unitnormal}^+ \coloneq \{x : \langle x - x_0, {\unitnormal}\rangle > 0\}$ and $H_{\unitnormal}^- \coloneq \{x  : \langle x - x_0, {\unitnormal}\rangle < 0\}$. 

Whenever a one-sided trace of $w$ exists at $x_0$, we denote it by $w^{\pm}(x_0, \unitnormal)$.

A Borel function $w\colon \Omega\to Y$ such that $\delta(w, y_0) \in L^1(\Omega; \mathbb{R})$ for some $y_0 \in Y$ is said to be a \textit{\textbf{$Y$-valued function of bounded variation}} (or a \textit{\textbf{metric BV function}}) if there exists a finite, positive measure $\sigma$ on $\Omega$ such that for any $1$-Lipschitz real-valued function $\varphi\colon Y \to \mathbb{R}$, 
\begin{equation}\label{def:metric-space-valued-bv-functions}
|D(\varphi\circ w)|(B) \leq \sigma(B) \quad \text{ for all } B \in \mathfrak{B}(\Omega).
\end{equation}
The least such measure $\sigma$ is called the \textbf{\textit{total variation}} of $w$ and is denoted $\abs{Dw}.$ 
We denote the class of all $Y$-valued $\BV$-function by $\BV(\Omega; Y)$. 

Using the separability of $Y$ and the Lebesgue differentiation theorem for real-valued functions, we now show that the Lebesgue differentiation theorem holds for any function in $L^1_{\mathrm{loc}}(\Omega; Y)$. 
\begin{prop}\label{prop:lebesgue-differentiation-thm-for-metric-space-valued-functions}
Let $\Omega \subset \mathbb{R}^n$ be open and let $(Y, \delta)$ be a separable metric space.
If $f \in L^1_{\mathrm{loc}}(\Omega; Y)$ then 
\[
\lim_{r\to 0+}\dashint_{\Omega \cap B(x_0, r)}\delta(f, f(x_0))\,d\mathscr{L}^n = 0 \quad \text{ for } \mathscr{L}^n\text{-a.e.}~x_0 \in \Omega.
\]
\end{prop}
\begin{proof}
Consider a countable dense subset $\{y_i\}$ of $Y$. 
Since every map $y\mapsto \delta(y, y_i)$ is 1-Lipschitz, the map $x\mapsto \delta(f(x), y_i)$ is in $L^1_{\mathrm{loc}}(\Omega)$ for any $i \in \mathbb{N}$ and hence the Lebesgue differentiation theorem applies to each of these maps. 
Furthermore, because $\{y_i\}$ is countable, $\mathscr{L}^n$-a.e.~$x_0\in \Omega$ 
\begin{equation}\label{eq:lebesgue-diff}
\lim_{r\to0+}\dashint_{\Omega \cap B(x_0, r)} |\delta(f, y_i) - \delta(f(x_0), y_i)|\,d\mathscr{L}^n = 0 \quad \text{ for all } i \in \mathbb{N}.
\end{equation}
For any such $x_0$, the triangle inequality gives us
\[
\delta(f, f(x_0)) \leq |\delta(f, y_i) - \delta(f(x_0), y_i)| + 2\delta(f(x_0), y_i) \quad \text{ for any } i \in \mathbb{N}. 
\]
For any $\varepsilon > 0$ we may choose a $y_i$ such that $\delta(f(x_0), y_i) < \varepsilon$ and therefore 
\[
\limsup_{r\to0+}\dashint_{\Omega \cap B(x_0, r)}\delta(f, f(x_0))\,d\mathscr{L}^n \leq 2\varepsilon. 
\]
Letting $\varepsilon \to 0$ proves the proposition.
\end{proof}

\subsection{Boundary Traces of Vertical Restrictions}
\label{subsec:verticalrestriction}
% In Section~\ref{sec:volume} we introduce the notion of a \emph{filling} between piecewise constant (PC) functions defined on the ends of a cylinder $\Omega\times I$.
% Loosely speaking, a filling will be defined as any $Y$-valued piecewise constant function on the cylinder, whose boundary trace agrees with the two PC functions defined on the ends of the cylinder. 
% Within the proof of the existence of the boundary trace for any $w \in \BV(\Omega\times I; Y)$, we prove the existence of the boundary trace for almost every one of its vertical restrictions.
In this subsection, we investigate properties of vertical restrictions of metric $\BV$ functions on $\Omega\times I$. 
A \textit{\textbf{vertical restriction}} of a function (not necessarily $\BV$) $w\colon \Omega\times I \to Y$ \textbf{\textit{at}} $x \in \Omega$ is the function
\[
w_x\colon I\to Y; \quad w_x(t) \coloneq w(x, t). 
\]
We start by proving two independent lemmas for $\BV$-curves: the endpoint (boundary) traces of BV-curves exist and their distance is bounded above by its total variation, and show in Lemma~\ref{lem:restriction-is-almost-always-bv} that almost every 
restriction of a metric $\BV$ function is a $\BV$-curve. 
Afterwards we show the agreement between the boundary trace of a metric $\BV$ function on the cylinder and the endpoint trace of almost every one of its vertical restrictions. 
We wait until \S\ref{subsec:energywelldefined} to prove the existence of boundary traces of metric $\BV$ functions. 

The \textit{\textbf{variation}} of a function $\gamma\colon I \to Y$, within $J\subset I$, is defined by the formula
\begin{equation}\label{def:variation}
\variation\paren{\gamma; J} = \sup\Big\{\sum_{i = 1}^l \delta(\gamma(t_i), \gamma(t_{i+1})) : t_1 \leq \cdots \leq t_{l+1} \text{ in } J \Big\},
\end{equation}
and its \textit{\textbf{essential variation}}, defined by
\begin{equation}\label{def:essential_variation}
\essvariation\paren{\gamma; J} \coloneq \inf_{\hat{\gamma} = \gamma\ (\mathscr{L}\allardell J)\text{-a.e.}} \variation\paren{\hat{\gamma}; J}. 
\end{equation}
A function \textit{$\gamma\colon I\to Y$ is in $\BV(I; Y)$ if and only if $\essvariation\paren{\gamma; I} < +\infty$, and moreover, $|D\gamma|(I) = \essvariation\paren{\gamma; I}$}~\cite{ambrosio1990metric}.
We call any such function a \textit{\textbf{$\BV$-curve}} in $Y$.
% \begin{remark}
% The reader may recognize that our formula (\ref{def:variation}) for variation is also used to define the length of a curve in a metric space. 
% Because the definition of a \emph{curve} requires the function defining it to be continuous, we instead use term variation. 
% \end{remark}

\begin{lemma}\label{lem:BV-curve_boundary_trace}
If $\gamma$ is a $\BV$-curve in a complete metric space $(Y, \delta)$, then there exist unique $y_0$ and $y_1$ in $Y$ such that 
\begin{equation}\label{def:BV-curve_boundary_trace}
\lim_{r\to 0+}\frac{1}{r}\int_{(0, r)}\delta(\gamma(t), y_0)\,dt = 0 \quad \text{ and } \quad \lim_{r\to 0}\frac{1}{r}\int_{(1-r, 1)}\delta(\gamma(t), y_1)\,dt = 0.
\end{equation}
Moreover, the endpoint traces satisfy
\begin{equation}\label{ineq:endpoint-trace-bound}
\delta(y_0, y_1) \leq |D\gamma|(I).
\end{equation}
\end{lemma}
\begin{proof}
Let $\hat{\gamma}$ be a function from $I$ to $Y$ with $\variation\paren{\hat{\gamma}; I} < +\infty$ such that $\hat{\gamma} = \gamma$ $\mathscr{L}$-a.e.~in $I$. 
Assume, for the moment, its one-sided limits 
\begin{equation}\label{def:one-sided-limits}
\lim_{t\to 0+}\hat{\gamma}(t) \quad \text{ and } \quad \lim_{t\to 1-}\hat{\gamma}(t),
\end{equation}
exist, and that we denote them by $y_0$ and $y_1$, respectively.
Then since one-sided traces agree with one-sided limits whenever the one-sided limits exist, the one-sided traces of $\hat{\gamma}$ are $y_0$ and $y_1$.
Now, because $\hat{\gamma}$ and $\gamma$ agree $\mathscr{L}$-almost everywhere in $I$, the one-sided traces (\ref{def:BV-curve_boundary_trace}) of $\gamma$ equals $y_0$ and $y_1$.  
The uniqueness of the points $y_0$ and $y_1$ satisfying (\ref{def:BV-curve_boundary_trace}) follows from the triangle inequality. 

Our only task is then to show the one-sided limits (\ref{def:one-sided-limits}) exist.
First, because the function $t\mapsto \variation\paren{\hat{\gamma}; (0,t]}$ defined on $I$ is bounded and monotonic, its one sided limits 
\[
\lim_{t\to 0+}\variation\paren{\hat{\gamma}; (0, t]} \quad \text{ and } \quad \lim_{t\to 1-}\variation\paren{\hat{\gamma}; (0, t]}
\]
exist.
Thus, since Cauchy sequences converge in $Y$, 
\begin{equation}\label{eq:bound_for_0}
\delta(\hat{\gamma}(r), \hat{\gamma}(t)) \leq \variation\paren{\hat{\gamma}; (0, t]} - \variation\paren{\hat{\gamma}; (0, r]}\quad \text{ for }  \quad 0 < r < t < 1,
\end{equation}
implies that the limits $y_0 \coloneq \lim_{t \to 0+}\hat{\gamma}(t)$ and $y_1 \coloneq \lim_{t \to 1-}\hat{\gamma}(t)$ exist.

We conclude our proof by showing inequality (\ref{ineq:endpoint-trace-bound}).
First, inequality (\ref{eq:bound_for_0}) and continuity of $\delta$ together imply $\delta(y_0, y_1) \leq \variation\paren{\hat{\gamma}; I}$.
Then, because the limits $y_0$ and $y_1$ are unique and in particular, do not depend on the representative for $\gamma$, taking the infimum over all representatives of $\gamma$ gives 
\[
\delta(y_0, y_1) \leq \inf_{\hat{\gamma} = \gamma \text{ a.e.}}\variation\paren{\hat{\gamma}; I} = \essvariation\paren{\gamma; I} = |D\gamma|(I).
\]
\end{proof}

\textbf{NOTATION.} To stay consistent between the notation we are using for the boundary trace of a $\BV$-functions whose domain is not simply an interval, we will denote the points $y_0$ and $y_1$ in $Y$ which satisfy the limit equations (\ref{def:BV-curve_boundary_trace}), by $\gamma|_{0}$ and $\gamma|_{1}$, respectively.

\begin{lemma}\label{lem:restriction-is-almost-always-bv}
Let $\Omega$ be a bounded, connected, open subset of $\mathbb{R}^n$, let $I = (0, 1)$, and let $(Y, \delta)$ be a complete separable metric space. 
If $w \in \BV(\Omega\times I; Y)$ then for $\mathscr{L}^n$-a.e.~$x\in \Omega$, the vertical restriction $w_x \in \BV(I; Y)$.
\end{lemma} 
\begin{proof}
By definition of $w\in \BV(\Omega\times I; Y)$ we can choose a finite Radon measure $\mu$ on $\Omega\times I$ such that
\[
|D(\varphi \circ w)|\leq \mu \quad \text{ for any } \varphi \in \mathrm{Lip}_1(Y; \mathbb{R}).
\]
In the following we show that there is a single Radon measure with the same property on $w_x$ for $\mathscr{L}^n$-a.e.~$x\in \Omega$. 

First we choose a countable dense subset $\{y_i\}$ of $Y$ and define 
\[
f_i(x, t)\coloneq \delta(w(x, t), y_i) \quad \text{ for all } (x, t) \in \Omega\times I.
\]
Because $f_i \in \BV(\Omega\times I; \mathbb{R})$, the theorem for 1-dimensional restrictions of scalar-valued $\BV$ functions implies that for $\mathscr{L}^n$-a.e.~$x \in \Omega$ the one dimensional restriction $(f_i)_x$ belongs to $\BV(I; \mathbb{R})$. 
Since $\{f_i\}$ are countable, there exists an $\mathscr{L}^n$-negligible set $N$ such that for any $i \in \mathbb{N}$, the one-dimensional restriction $(f_i)_x$ belongs to $\BV(I)$ for all $x \in \Omega\setminus N$. 

Now for such an $x \in \Omega\setminus N$, consider $\bigvee_{i} |D(f_i)_x|$,  defined as the \textit{least measure dominating all the measures in} $\{|D(f_i)_x| : i \in \mathbb{N}\}$. 
By \cite[(1.7), (1.8)]{ambrosio1990metric}, 
\[
\int_\Omega \bigvee_i |D(f_i)_x|\,d\mathscr{L}^n(x) = \bigvee_i \int_\Omega |D(f_i)_x|\,d\mathscr{L}^n(x) = \bigvee_i |D_{n+1}f_i| \leq \bigvee_i|Df_i| \leq \mu,
\]
and in particular, 
\[
\int_\Omega \Big(\bigvee_i |D(f_i)_x|\Big)(I)\,d\mathscr{L}^n(x) \leq \mu(\Omega\times I) < +\infty.
\]
Thus there exists an $\mathscr{L}^n$-negligible set $N'$ such that  $\vee_i |D(f_i)_x|$ is a finite Radon measure on $I$ for every $x\in \Omega\setminus N'$, which in addition, bounds $|D(f_i)_x|$  from above for any $i \in \mathbb{N}$ and $x \in \Omega\setminus (N\cup N')$.
Fix $x \in \Omega \setminus (N \cup N')$. 
For the remainder of the proof, we will show $|D(\varphi\circ w_x)|\leq \vee_i |D(f_i)_x|$ for any $\varphi \in \mathrm{Lip}_1(Y; \mathbb{R})$.

Any $(f_i)_x$ has a representative that is continuous outside the set of jump points of the scalar-valued function \cite[Theorem 3.28 (b)]{ambrosio2000functions}, and moreover, the set of jump points is at most countable since $|D(f_i)_x|(I) < +\infty$. 
Therefore we may choose a subset $E_x$ of $I$, with $\mathscr{L}(I\setminus E_x) = 0$, on which any $(f_i)_x$ agrees with a good representative that is continuous on $E_x$.
Let $s < t$ in $E_x$.
Since $|(f_i)_x(t) - (f_i)_x(s)| \leq |D(f_i)_x|(s, t) \leq \Big(\bigvee_i |D(f_i)_x|\Big)(s, t)$, we get
\begin{equation}\label{ineq:distance-bound}
\delta\paren{w_x(t), w_x(s)} = \sup_{i}|(f_i)_x(t) - (f_i)_x(s)| \leq \Big(\bigvee_i |D(f_i)_x|\Big)(s, t).
\end{equation}
Let $\varphi \in \mathrm{Lip}_1(Y)$. 
To show $w_x \in \BV(I; Y)$, we first notice that $\varphi\circ w_x \in L^1(I)$ follows from the bound $|\varphi\circ w_x| \leq |\varphi(y_1)| + (f_1)_x$. 
To conclude, it suffices to show that $|D(\varphi\circ w_x)|(J) \leq \bigvee_i |D(f_i)_x|(J)$ for any open interval $J\subset I$. 
Towards that end, choose any finite sequence $t_0 < \cdots < t_m$ in $E_x\cap J$. 
Inequality (\ref{ineq:distance-bound}) together with the fact that $\varphi$ is 1-Lipschitz, gives us 
\[
\sum_k |\varphi(w_x(t_k)) - \varphi(w_x(t_{k-1}))| \leq \sum_k \big(\bigvee_i |D(f_i)_x|\big)(t_{k-1}, t_k) \leq \big(\bigvee_i|D(f_i)_x|\big)(J).
\]
Taking the supremum of the left-hand side over all partitions $t_0 < \cdots < t_m$ in $E_x \cap J$ gives us $|D(\varphi\circ w_x)|(J) \leq \bigvee_i |D(f_i)_x|(J)$, therefore concluding the proof.
\end{proof}

\begin{lemma}\label{lem:boundary-trace-endpoint-trace-agreement}
Let $\Omega$ be a bounded, connected, open subset of $\mathbb{R}^{n}$ with Lipschitz boundary, and let $(Y, \delta)$ be a complete separable metric space.
If the boundary trace $w|_{\partial (\Omega\times I)}(x, t)$ of $w \in \BV(\Omega\times I; Y)$ exists for a.e.~$(x, t)\in \Omega\times\{0, 1\}$, then
\begin{equation}\label{cond:traces_of_slices}
w_x|_{\partial (0,1)}(t) = w|_{\partial \cyl}(x, t) \quad \text{ for a.e.}~(x, t)\in \Omega\times\{0, 1\}. 
\end{equation}
\end{lemma}
\begin{proof}
In the following, we prove (\ref{cond:traces_of_slices}) for $t = 0$. 
The case of $t = 1$ follows from a similar argument. 
Suppose $p(x) \coloneq w|_{\partial (\Omega\times I)}(x, 0)$ exist for a.e.~$x \in \Omega$.
We start our proof the same as in our proof of the existence of boundary traces for metric space valued BV functions. 
In particular, we choose a countable dense subset $\{y_i\}_{i \in \mathbb{N}}$ of $Y$ and for each $i$, define
\[
v_i(x) \coloneq \delta(w(x), y_i) \quad \text{ for all } x \in \cyl.
\] 
(\emph{Step 1}) In the first step, we prove 
\[
(v_i)_x|_{\partial I}(0) = \delta(p(x), y_i) \quad \text{ for a.e.}~x \in \Omega.
\]
Suppose the boundary trace of $w$ exists at the point $(x, 0)$, and is equal to $p(x)$.
The reverse triangle inequality implies $|\delta(w, y_i) - \delta(p(x), y_i)| \leq \delta(w, p(x))$, and hence $v_i|_{\partial (\Omega\times I)}(x, 0) = \delta(p(x), y_i)$.

To show that the endpoint (boundary trace) of the vertical restriction of $v_i$ is also $\delta(p(x), y_i)$ we extend $v_i$ to all of $\mathbb{R}^{n+1}$ by defining 
\[
\bar{v}_i(x,t) \coloneq
\begin{cases}
v_i(x, t) & \text{ if } (x, t) \in \Omega\times  I\\
0 & \text{ otherwise}
\end{cases}.
\]
Notice that because the one-sided approximate limit $(\bar{v}_i)_{\verticalvec}^+(x, 0)$ in the positive vertical direction $\verticalvec$ agrees with the boundary trace $v_i|_{\partial(\Omega\times I)}(x, 0)$, 
\[
(\bar{v}_i)^+_{\verticalvec}(x, 0) = \delta(p(x), y_i) \quad \text{ and } \quad (\bar{v}_i)^-_{\verticalvec}(x, 0) = 0.
\]

By Remark 3.109 (also see Theorem 3.108(a)) of AFP's monograph \cite{ambrosio2000functions}, for $\mathscr{L}^{n}$-a.e.~$x \in \Omega$, 
\begin{equation}\label{thm:jump_set_of_restriction_commutative}
\{x\}\times J_{(\bar{v}_i)_x} = J_{\bar{v}_i}\cap\{(x, t) : t \in \mathbb{R}\}.
\end{equation}

Let $x$ be a point in $\Omega$ satisfying (\ref{thm:jump_set_of_restriction_commutative}) along with the one-sided identities of Remark 3.109.
There are two cases to consider: either $\delta(p(x_0), y_i) = 0$, or $\delta(p(x_0), y_i) > 0$. 

Assuming the former, then $\bar{v}_i$ is approximately continuous at $(x, 0)$, and therefore, (\ref{thm:jump_set_of_restriction_commutative}) implies $(\bar{v}_i)_{x}$ is also approximately continuous at $t = 0$.
Because the one-sided traces of $\bar{v}_i$ and $(\bar{v}_i)_{x}$ are $0$ in the negative $\verticalvec$ direction, this would require their one-sided traces in the positive $\verticalvec$ direction to also be $0$, and thus, the endpoint trace of $(\bar{v}_i)_{x}$ at $t = 0$ equals $\delta(p(x), y_i)$. 

In the case that $\delta(p(x), y_i) > 0$, then $(x, 0)$ is a jump point of $\bar{v}_i$, and therefore, $t = 0$ is a jump point of its vertical restriction $(\bar{v}_i)_{x}$.
By definition, the one-sided trace of $(\bar{v}_i)_x$ exists at $0$ and is equal to the end-point trace of $(v_i)_{x}$ at $0$, so by \cite[Remark 3.109]{ambrosio2000functions}, 
\[
(v_i)_x|_{\partial  I}(0) = \lim_{r\to0+}\dashint_{(0, r)} (\bar{v}_i)_x\,d\mathscr{L} = (\bar{v}_i)^+_{\verticalvec}(x, 0).
\]

(\emph{Step 2})
Fix an $x \in \Omega$ such that Step 1 holds for every $i$, $w_x \in \BV(I; Y)$, and its endpoint trace exists at $0$. 
Set $y^\ast \coloneq w_x|_{\partial I}(0)$.
Lemma~\ref{lem:restriction-is-almost-always-bv} and countability of $y_i$ implies such conditions hold $\mathscr{L}^n$-a.e.~ in $\Omega$.
The reverse triangle inequality implies
\[
\lim_{r\to0+}\dashint_{(0, r)}|\delta(w_x(t), y_i) - \delta(y^\ast, y_i)|\,dt \leq \lim_{r\to0+}\dashint_{(0, r)} \delta(w_x(t), y^\ast)\,dt = 0 \quad \text{ for any } i \in \mathbb{N}.
\]
Hence, $(v_i)_x|_{\partial I}(0) = \delta(y^\ast, y_i)$.
However, we already showed that the endpoint trace of $(v_i)_x$ is equal to $\delta(p(x), y_i)$. 

This implies that $\delta(y^\ast, y_i) = \delta(p(x), y_i)$ for all $i \in \mathbb{N}$.
Because $\{y_i\}$ is dense $Y$, $y^\ast = p(x)$. 
\end{proof}

%The lemma showing the agreement between endpoint traces of the vertical restrictions and the boundary trace, together with inequality (\ref{ineq:endpoint-trace-bound}).

\subsection{Minimizing the Horizontal Total Variation}
\label{sec:horizontal}
In this section, we prove the following theorem.
\begin{theorem}\label{thm:general-bv-statement}
Let $Y$ be a complete, separable metric space and let $p, q \in \BV(\Omega; Y)$ where $\Omega\subset \mathbb{R}^n$ is bounded, connected, and open. 
Then 
\begin{equation}\label{eq:generalfilling}
\min_{\substack{w \in \BV(\Omega\times  I; Y)\\
w|_{\Omega\times0} = p\text{ a.e.}\\
w|_{\Omega\times1} = q\text{ a.e.}}}\int_\Omega |Dw_x|(0,1)\,d\mathscr{L}^n(x) = \int_{\Omega}\delta(p, q)\,d\mathscr{L}^n
\end{equation}
and the minimum is achieved.
\end{theorem}

The quantity $\int_\Omega |Dw_x| (I)$  is also denoted by $|\langle Dw_x, \boldsymbol{e}_{n+1}\rangle|$, in particular, because it agrees with the total variation of the weak derivative of $w$ in the direction $\boldsymbol{e}_{n+1}$ whenever $w$ is a vector-valued $\BV$ function.
In the case of a piecewise-constant filling between $p$ and $q$ (defined in the next section), $|\langle Dw_x, \boldsymbol{e}_{n+1}\rangle|(\Omega\times I)$ will agree with the surface energy of the $L^1$-approximate discontinuity set of $w$. 
Our volume exchange metric will be defined, in part, by minimizing this surface energy and Theorem~\ref{thm:l1_energy_for_metric_BV} will be a corollary of the previous result.

We separately prove that the left hand side of Equation~\ref{eq:generalfilling} is greater than or equal to the right hand side in Proposition~\ref{prop:direcitonal-derivative-dominates-L1} and the reverse inequality in  Lemma~\ref{lem:step-function}. This suffices to show the preceding theorem.

\begin{prop}\label{prop:direcitonal-derivative-dominates-L1}
Let $\Omega$ be a bounded, connected, open subset of $\mathbb{R}^{n}$, and let $(Y, \delta)$ be a complete separable metric space.
If $w \in \BV(\Omega\times  I; Y)$ has boundary traces $p$ and $q$ at $\Omega\times 0$ and $\Omega\times 1$, respectively, then  
\[
\int_{\Omega} \delta(p, q)\,d\mathscr{L}^{n} \leq \int_\Omega |Dw_x|(I)\,d\mathscr{L}^n(x).
\]
\end{prop}
\begin{proof}
By Lemma~\ref{lem:restriction-is-almost-always-bv} $w_x \in \BV(\Omega; Y)$ for $\mathscr{L}^{n}$-a.e.~$x \in \Omega$.
Thus, by Lemma~\ref{lem:boundary-trace-endpoint-trace-agreement} and the upper bound (\ref{ineq:endpoint-trace-bound}) 
\[
\int_{\Omega} \delta(p, q)\,d\mathscr{L}^{n} = \int_{\Omega} \delta(w_x|_{0}, w_x|_{1})\,d\mathscr{L}^n(x) \leq \int_\Omega |Dw_x|(I)\,d\mathscr{L}^n(x). 
\]
\end{proof}

We now show that directional weak derivative of the step-function $w$ (defined below in Equation (\ref{def:step-function})) that jumps from $p$ to $q$ at some point between $0$ and $1$ is exactly the $L^1$-norm of $\delta(p, q)$. 

\begin{lemma}\label{lem:step-function}
Let $(Y, \delta)$ be a complete separable metric space, let $\Omega$ be a bounded, connected, open subset of $\mathbb{R}^n$, and let $p, q \in \BV(\Omega; Y)$.
The step-function $w^\ast\colon\Omega\times I \to Y$ defined by the equation
\begin{equation}\label{def:step-function}
w^\ast(x, t) \coloneq \begin{cases}
p(x) & \text{if } \quad (x, t) \in \Omega\times (0, 1/2]\\
q(x) & \text{if } \quad (x, t) \in \Omega\times (1/2, 1),
\end{cases}
\end{equation}
which jumps from $p$ to $q$ at $t = 1/2$, has the following properties.
\begin{enumerate}
\item[(a)] $w^\ast \in \BV(\Omega\times  I; Y)$;
\item[(b)] the boundary trace of $w^\ast$ exists and is given by 
\[
w^\ast|_{\Omega\times 0}(x, 0) = p(x) \quad \text{ and } \quad  w^\ast|_{\Omega\times 1}(x,1) = q(x) \quad \text{ for } \mathscr{L}^n\text{-a.e.}~x \in \Omega;
\]
\item[(c)] $\int_\Omega |Dw_x|(I)\,d\mathscr{L}^n(x) = \|\delta(p, q)\|_{L^1(\Omega)}$.
\end{enumerate}
\end{lemma}
\begin{proof}
(\emph{a}) We will show $w$ is in $\BV(\Omega\times I; Y)$ by definition.
First notice that the triangle inequality implies $w^\ast \in L^1(\Omega\times I; Y)$.
Let $\varphi\colon Y\to \mathbb{R}$ be an arbitrary 1-Lipschitz function, and let us denote the post-compositions of $w^\ast$, $p$ and $q$ with $\varphi$ by $v$, $v_0$ and $v_1$, respectively.

Let us now compute the weak derivatives of $v$. 
Let $\phi \in C_c^\infty(\Omega\times I; \mathbb{R})$ and first fix $j \in \{1, \dots, n\}$.
Because the map $x\mapsto \phi(x, t)$ is in $C_c^\infty(\Omega; \mathbb{R})$ for any $t \in  I$, we get 
\begin{equation}
    \int_\Omega v_i(x)\partial_j\phi(x, t)\,d\mathscr{L}^{n}(x) = -\int_\Omega\phi(x, t)\,dD_jv_i(x) \quad \text{ for } i = 0, 1 \text{ and } t \in  I.
\end{equation}
Together with Fubini's theorem and the definition of $w$, this implies  
\[
    \int_{\Omega\times(0,1)} v\,\partial_j \phi\,d\mathscr{L}^{n+1} = -\int_{(0, 1/2)}\left(\int_\Omega\phi\,D_jv_0\right)d\mathscr{L} - \int_{(1/2, 1)}\left(\int_\Omega\phi\,D_jv_1\right)d\mathscr{L},
\]
and thus
\begin{equation}\label{eq:weak-derivative-to-n}
    D_jv = D_jv_0\,\mathscr{L}\allardell (0, 1/2) + D_jv_1\,\mathscr{L}\allardell(1/2, 1).
\end{equation}
For $j = n+1$, because the map $t\mapsto \phi(x, t)$ is smooth and compactly supported on $ I$ (for any $x \in \Omega$), the definition of $v$ implies $\int_{(0,1)} v(x, t)\,\partial_{n+1}\phi(x, t)\,dt = (v_0(x) - v_1(x))\,\phi(x, 1/2)$.
and thus 
\begin{equation}
    D_{n+1}v = (v_1 - v_0)\mathscr{H}^n\allardell \Omega\times \{1/2\}.
\end{equation}

Now by definition, there exists finite measures $\sigma_0$ and $\sigma_1$ on $\Omega$ such that $|D(\varphi\circ p)| \leq \sigma_0$ and $|D(\varphi\circ q)| \leq \sigma_0$ for any 1-Lipschitz function $\varphi\colon Y\to \mathbb{R}$. 
Thus,
\[
|Dv| \leq \sigma_0\otimes \mathscr{L}^1\allardell (0, 1/2) + \sigma_1 \otimes\mathscr{L}^1\allardell (1/2, 1) + \delta(p, q)\,\mathscr{H}^n\allardell (\Omega \times \{1/2\}).
\]
Since the measure bounding $|Dv|$ has total a total mass of $\frac{1}{2}\sigma_0(\Omega) + \frac{1}{2}\sigma_1(\Omega) + \int_\Omega \delta(p, q) < +\infty$, and since $\varphi$ was an arbitrary, we have shown $w$ is in $\BV(\Omega; Y)$.

(\emph{b}) 
We now show that the boundary trace of $w^\ast$ exists at the bottom $\Omega\times 0$ of the cylinder $\Omega\times  I$; a similar argument can be made for the top $\Omega\times 1$ of the cylinder. 
By Proposition \ref{prop:lebesgue-differentiation-thm-for-metric-space-valued-functions} the $L^1$-approximate limit of $p$ exists and is equal to $p(x_0)$ for $\mathscr{L}^n$-a.e.~$x_0 \in \Omega$. 
Therefore, for any such $x_0 \in \Omega$, since $B((x_0, t), r) \subset B(x_0, r)\times [0, r)$ and 
\[
\int_{B(x_0, r)\times [0, r)} \delta(w, p(x_0))\,d\mathscr{L}^{n+1} = r\int_{B(x_0, r)}\delta(p, p(x_0))\,d\mathscr{L}^n ,
\]
we get that for small enough $r > 0$, 
\begin{align*}
\dashint_{\Omega\times(0,1)\cap B((x_0, 0), r)}\delta(w^\ast, p(x_0))\,d\mathscr{L}^{n+1} &\leq \frac{r}{\frac{1}{2}\omega_{n+1}r^{n+1}}\int_{B(x_0, r)}\delta(p, p(x_0))\,d\mathscr{L}^{n}\\
&= \frac{2\omega_{n}}{\omega_{n+1}}\dashint_{B(x_0, r)}\delta(p, p(x_0))\,d\mathscr{L}^{n} \longrightarrow 0
\end{align*}
as $r\to 0+$. 
Hence by the definition, the boundary trace of $w^\ast$ exists for $\mathscr{L}^{n}$-a.e.~$x_0\in \Omega$ and is equal to $p(x_0)$. 

(\emph{c}) 
We now prove the final part of the Lemma.
For each $x \in \Omega$, the vertical slice $w^\ast_x$ is a step function which jumps from $p(x)$ to $q(x)$ at $t = 1/2$. 
Because $w^\ast_x$ minimizes the variation over any representative of $w^\ast_x$, and because $\essvariation\paren{w^\ast_x, (0,1)} = |Dw^\ast_x|(I)$, we get $\delta\paren{p, q} = \variation\paren{w^\ast_x; (0,1)} = |Dw^\ast_x| (I)$. 
Hence
\[
\int_\Omega \delta(p(x), q(x))\,d\mathscr{L}^{n} = \int_\Omega |Dw^\ast_x|(I)\,d\mathscr{L}^{n}.
\]
\end{proof}

\subsection{Rectifiability of approximate discontinuity set and existence of boundary traces}
\label{subsec:energywelldefined}
\begin{theorem}\label{thm:approximate-discontinuity-rectifiability}
Suppose $(Y, \delta)$ is a complete and separable metric space, and suppose $w \in \BV(\Omega; Y)$. 
Then $S_w$ is countably $\mathscr{H}^{n-1}$-rectifiable, and there exist Borel maps 
\[
\unitnormal\colon S_w\to \mathbb{S}^{n-1}, \quad w^+, w^-\colon S_w \to Y
\]
such that $\unitnormal$ orients $S_w$ and, for $\mathscr{H}^{n-1}$-almost every $x_0 \in S_w$, 
\[
\lim_{r\to 0}\dashint_{B_{\unitnormal}^\pm(x_0, r)}\delta\paren{w, w^\pm(x_0)}\,d\mathscr{L}^n = 0 \quad \text{ and } w^+(x_0) \neq w^-(x_0).
\]
\end{theorem}
\begin{proof}
Let $\{y_i\}$ be a countable dense subset of $Y$, and suppose $y_1$ is chosen such that $\delta\paren{w, y_1} \in L^1(\Omega)$ -- which we may assume exists by the definition of the class $\BV(\Omega; Y)$. 
In addition, define
\begin{equation}\label{def:min-distance}
f_i(x)\coloneq \delta\paren{w(x), y_i} \quad \text{ and } \quad v_N\paren{x} \coloneq \min_{i \leq N} f_i(x)
\end{equation}
for each $N \in \mathbb{N}$.

(\textit{Step 1}) 
The distance function $y\mapsto \min_{i \leq N}\delta\paren{y, y_i}$ is 1-Lipschitz.
Hence, by the definition of the class $\BV(\Omega; Y)$, $v_N \in \BV(\Omega)$ for any $N$, and there exists a finite Radon measure $\mu$ on $\Omega$ such that $|D v_N|\leq \mu$ for all $N$. 

Moreover, because $0 \leq v_{N+1} \leq v_N \leq \delta\paren{w, y_1}$ and $v_N \to 0$ as $N \to \infty$, the dominated convergence theorem gives us that 
\[
\lim_{N\to 0}\|v_N\|_{L^1(\Omega; \mathbb{R})} = 0.
\]
Now, by the Radon-Nikodym theorem, there exists an $L^1$-summable function $g_N\colon \Omega\to \mathbb{R}^n$ such that $Dv_N = g_N\mu$. 
Notice that we also have $|g_N(x)| \leq 1$ for $\mu$-a.e.~$x \in \Omega$. 
Together with the definition of the weak-derivative of $v_N$, this implies that for any $\phi \in C_c^1(\Omega; \mathbb{R}^n)$ we have 
\begin{equation}\label{eq:convergenceweak}
\int_\Omega \phi\cdot g_N\,d\mu = -\int_\Omega v_N\,\mathrm{div}\,\phi\,d\mathscr{L}^n \longrightarrow 0 \text{ as } N\to \infty.
\end{equation}
Now, because $\mu$ is a finite Radon measure on $\Omega$, $C_c^1(\Omega; \mathbb{R}^n)$ is dense in $L^1(\mu; \mathbb{R}^n)$ and $L^\infty(\mu; \mathbb{R}^n)\subset L^1(\mu; \mathbb{R}^n)$, for any $\varepsilon > 0$, and $\psi \in L^\infty(\mu; \mathbb{R}^n)$, we can pick $\phi \in C_c^1(\mu; \mathbb{R}^n)$ such that $$\|\psi - \phi\|_{L^1(\mu; \mathbb{R}^n)} < \varepsilon.$$ Writing $\int_\Omega \psi\cdot g_N\,d\mu = \int_\Omega \phi\cdot g_N\,d\mu + \int_\Omega (\psi - \phi)\cdot g_N\,d\mu$, the uniform boundedness of $g_N$, and (\ref{eq:convergenceweak}), this implies $g_N$ converges \textit{weakly} to $0$ in $L^1(\Omega)$.
Strong convergence of $v_N$ to $0$ in $L^1(\Omega; \mathbb{R})$ together with weak convergence of $g_N$ to $0$ in $L^1(\mu; \mathbb{R}^n)$ implies the ordered pair $(v_N, g_N)$ converges weakly to $(0, 0)$ in the Banach space $L^1(\Omega) \times L^1(\mu; \mathbb{R}^n)$.

By Mazur's lemma, for $m \in \mathbb{N}$, there exist finite convex combinations 
\[
(u_m, h_m) = \sum_{i = m}^{K_m}\lambda_i^m (v_i, g_i)  \quad \text{ where } \quad \|u_m\|_{L^1(\Omega)} + \|h_m\|_{L^1(\mu; \mathbb{R}^n)} \to 0.
\]
Since $D u_m = h_m \mu$, we get that 
\begin{equation}\label{def:bv-scalar-bound}
\|u_m\|_{\BV(\Omega)} \to 0.
\end{equation}
This provides us with a subsequence $\{u_{m_k}\}$ such that $\sum_{k = 1}^\infty \|u_{m_k}\|_{\BV(\Omega)} < +\infty$, which we note, implies 
\[
u \coloneq \sum_{k = 1}^\infty u_{m_k} \in \BV(\Omega).
\]

Let $M \in \mathbb{N}$ and $N_M \coloneq \max\{K_{m_1}, \dots, K_{m_M}\}$.
For any $N \geq N_M$ we have 
\begin{equation}\label{ineq:bv-domination}
M\,v_N \leq \sum_{k = 1}^M v_{K_{m_k}} \leq \sum_{k = 1}^M \sum_{i = m_k}^{K_{m_k}} \lambda_i^{m_k} v_i = \sum_{k = 1}^M u_{m_k} \leq u.
\end{equation}

(\textit{Step 2})
Define 
\[
\Gamma \coloneq S_u \cup \bigcup_{i = 1}^\infty S_{f_i}.
\]
Since $u\in \BV(\Omega)$ and $f_i\in \BV_{\mathrm{loc}}(\Omega)$, the scalar $\BV$ rectifiability theorem~\cite[Theorem 3.78]{ambrosio2000functions} implies that $\Gamma$ is countably $\mathscr H^{n-1}$-rectifiable.
Let $\unitnormal$ be a Borel map orienting $\Gamma$. 
Applying the interior trace theorem to $u$ and each $f_i$ gives us an $\mathscr{H}^{n-1}$-negligible set $\mathcal{N}\subset \Gamma$ such that the one sided traces $u^\pm_{\unitnormal}(x)$ of $u$ and $(f_i)^\pm_{\unitnormal}(x)$ (for any $i$) exist at every $x \in \Gamma\setminus \mathcal{N}$, and are Borel.

Let $x_0 \in \Omega\setminus \mathcal{N}$.  
For the following argument, we set $B_r = B(x_0, r)$ in the case $x_0 \notin \Gamma$, and set $B_r$ to be either half-ball $B_{\unitnormal(x_0)}^\pm(x_0, r)$ if $x_0\in \Gamma$.
Additionally, we denote the trace (corresponding to $B_r$) of $u$ by $\tau$, and the trace of any $f_i$ by $\tau_i$. 
Now, because 
\[
\Big|v_N(x) - \min_{i \leq N}\tau_i\Big| \leq \sum_{i = 1}^N|f_i(x) - \tau_i| \quad \text{ for any } x \in \Omega
\]
we get 
\[
\lim_{r\to 0+}\dashint_{B_r}|v_N - \min_{i \leq N}\tau_i|\,d\mathscr{L}^n = 0.
\]
Thus taking the limit of the average value of both sides of the inequality (\ref{ineq:bv-domination}) over $B_r$ gives us 
\[
0 \leq \min_{i \leq N} \tau_i \leq \frac{\tau}{M} \quad \text{ for } N \geq N_M \text{ and } M \in \mathbb{N},
\]
and hence $\inf_i \tau_i = 0$.
This provides us with a sequence $\{i_j\}$ such that $\tau_{i_j} < 2^{-j}$ for each $j \in \mathbb{N}$, and because 
\[
\delta(y_i, y_j) \leq \tau_i + \tau_j \quad \text{ for any } i,j \in \mathbb{N},
\]
we get that the subsequence $\{y_{i_j}\}$ is Cauchy.
In particular, since $Y$ is complete, this implies the existence of a point $y_0 \in Y$ such that $y_{i_j}\to y_0$.

Thus
\[
\lim_{r\to0+} \dashint_{B_r} \delta\paren{w, y_0}\,d\mathscr{L}^n\; \leq\; \tau_{i_j} + \delta\paren{y_{i_j}, y_0}\; \longrightarrow\; 0 \qquad \text{ as } j \to \infty. 
\]

Therefore if $x_0 \notin \Gamma$, i.e., the approximate limit of $u$ and of all the $f_i$ exist at $x_0$, then the approximate limit of $w$ exists at $x_0$, and in particular, $S_w$ is contained in the rectifiable set $\Gamma$.
Therefore, $S_w$ is $\mathscr{H}^{n-1}$ rectifiable.
Moreover, we have shown that the one-sided traces $w^+(x_0)$ and $w^-(x_0)$ exist at any $x_0 \in \Gamma\setminus \mathcal{N}$. 
Enlarging $\mathcal{N}$ if necessary, assume it is Borel and set $G:=S_w\setminus\mathcal{N}$.
By the triangle inequality and uniqueness of scalar traces, 
\[
\delta\paren{w^\pm(x), y_i} = (f_i)_{\unitnormal}^\pm(x) \quad \text{ for } x \in S_w\setminus\mathcal{N},\; i \in \mathbb{N}.
\]
Since the metric balls around the countable dense set $\{y_i\}$ generate the Borel $\sigma$-algebra of $Y$, the Borel measurability of the scalar traces implies $w^\pm$ are Borel on $S_w\setminus \mathcal{N}$. 
Extend both maps by $y_1$ on $S_w \cap \mathcal{N}$ and restrict $\unitnormal$ to $S_w$. 

Finally, the traces are distinct on $S_w\setminus \mathcal{N}$. 
Indeed, if they coincided, then combining the two half-ball limits would give a full-ball average limit, contradicting membership in $S_w$. 
\end{proof}

The following theorem shows that metric space valued $\BV$-functions on bounded open sets have boundary traces defined almost everywhere if the boundary is Lipschitz. 
This theorem motivates our notion of a filling (defined in the next section) between two metric space valued $\BV$-functions.

Recall that a subset $U\subset \mathbb{R}^n$ has \textit{\textbf{Lipschitz boundary} if for every point $x \in \partial U$ there exists an open ball $B_r(x)$ and a hyperplane $P$ through $x$, such that $B_r(x)\cap \partial U$ is the graph of some Lipschitz function over $P$}.

\begin{theorem}[Existence of the boundary trace]\label{thm:boundary-traces}
Let $\Omega\subset \mathbb{R}^{n}$ be a bounded, open, connected set with Lipschitz boundary $\partial \Omega$ and let $(Y, \delta)$ be a complete separable metric space. 
If $w \in \BV(\Omega; Y)$ then for $\mathscr{H}^{n-1}$-a.e.~$x_0 \in \partial \Omega$ there exists a $w|_{\partial \Omega}(x_0) \in Y$ such that 
\[
\lim_{r\to 0+} \dashint_{\Omega\cap B(x_0, r)} \delta(w, w|_{\partial \Omega}(x_0))\,d\mathscr{L}^{n} = 0.
\]
\end{theorem}
\begin{proof}
Define $f_i$ and $v_N$ as in the preceding theorem. 
By Step 1 of its proof, there  exists a nonnegative $u \in \BV(\Omega)$ such that, for every $M \in \mathbb{N}$, we have $M v_N \leq u$ for $\mathscr{L}^n$-almost everywhere.
for all sufficiently large $N$. 

By the scalar boundary trace theorem \cite[Theorem 3.87]{ambrosio2000functions}, $u$ and all of the $f_i$ have finite boundary one-sided approximate limits outside of a common $\mathscr{H}^{n-1}$-negligible subset of $\partial \Omega$. 

At each remaining boundary point $x_0$, apply the reconstruction argument from Step 2 of the proof of Theorem~\ref{thm:approximate-discontinuity-rectifiability} with $B_r \coloneq \Omega \cap B(x_0, r)$. 
This gives a point $w|_{\partial \Omega}(x_0) \in Y$ satisfying 
\[
\lim_{r\to 0+} \dashint_{\Omega\cap B(x_0, r)} \delta\paren{w, w|_{\partial\Omega}(x_0)}\,d\mathscr{L}^n = 0.
\]
This concludes the proof.
\end{proof}

\section{The Volume Exchange Metric and Dissimilarity}\label{sec:volume}
In this section, we define three notions of piecewise constant (PC) fillings between two piecewise constant functions $p,q:\Omega\to Y.$ 
After defining standard piecewise constant (PC) fillings between $p$ and $q$, we will define two types of \emph{weak} PC fillings. 
Our weakest type of PC filling (type-2) will only need to preserve distances between values corresponding to adjacent level-sets of our trace boundary condition, whereas type-1 PC fillings will need to preserve distances between all level sets.
As explained below, these will correspond to the two notions of cellular fillings used to define the one-to-one volume exchange metric and many-to-many volume exchange dissimilarity in the introduction. 

We say that a function $w\colon \Omega \to Y$ is \textit{\textbf{piecewise constant in $\Omega$} if there exists a Caccioppoli partition $\{E_\alpha\}_{\alpha \in A}$ of $\Omega$ and a function $y \colon A\to Y$ such that for each $\alpha \in A$, $w(x) = y_\alpha$ for $\mathscr{L}^n$-a.e.~$x \in E_\alpha$, and 
\[
\sum_{\alpha\neq \beta} \delta\paren{y_\alpha, y_\beta} \mathscr{H}^{n-1}(\Omega\cap \partial^{\ast} E_\alpha \cap \partial^{\ast} E_\beta) < +\infty.
\]}
% Note that because we do not require $y$ to be injective, the piecewise constant function $w$ may be represented in many ways (even for a fixed partition and function $y$).
% In particular, for the quotient set $A_\sim$ of any equivalence relation $\sim$ on $A$ such that $y_\alpha = y_\beta$ for any $\alpha \sim \beta$, the function $\tilde{y}\colon A_\sim \to Y$ defined by $\tilde{y}_{[\alpha]} = y_\alpha$ provides us with another representation of $w$:
% \[
% w|_{E{\brac{\alpha}}}=\tilde{y}_{\brac{\alpha}} \forall \brac{\alpha}\in A_{\sim}\quad \text{ where } \quad E_{[\alpha]} \coloneq \bigcup_{\beta \in [\alpha]} E_\beta.
% \]
% The family of all such quotient sets of $A$ is partially ordered under refinements; $A_\sim$ is a refinement of $A_\approx$ (denoted by $A_{\sim} \leq A_\approx$) if $\alpha\sim \beta$ implies $\alpha \approx \beta$.
% Moreover, the collection of level sets of $y$ is the maximal element, and the collection of singletons of $A$ is minimal for any function $A\to Y$. 
% Note that this induces a partial ordering on the representations of $w.$

\begin{definition}[Piecewise constant filling]
Given $p, q \in \PC(\Omega; Y)$, a \textit{\textbf{piecewise constant (PC) filling between}} $p$ and $q$ is a piecewise constant function $w\colon \Omega\times I \to Y$ such that 
\begin{equation}\label{def:type-0-filling}
w|_{\Omega\times 0}(x, 0) = p(x) \quad \textit{ and } \quad w|_{\Omega\times 1}(x, 1) = q(x) \quad \textit{ for $\mathscr{L}^n$-a.e.}~x \in \Omega.
\end{equation}
We denote the family of all fillings between $p$ and $q$ by $\mathsf{F}(p, q)$.
\end{definition}

The following notion will be used to define our weakest type of PC filling. 
A function $\kappa\colon \im{p}\to \im{p}$ is called a \textit{\textbf{coarsening map for}} $p$ if
\[
\delta(\kappa(y), \kappa(y')) = \delta(y, y') \text{ for all } y, y' \in Y  \text{ such that } 
\mathscr{H}^{n-1}(\partial^{\ast} \{p = y\} \cap \partial^{\ast} \{p = y'\}) > 0.
\]
With any coarsening map $\kappa$ of $p$, one can define a new piecewise constant function $\kappa \circ p$ which we call a coarsening of $p.$ Denote the collection of coarsenings of $p$ by $\mathcal{C}\paren{p}.$ 
A level-set of $\kappa\circ p$ is the union of non-adjacent level-sets of $p$. 

\begin{definition}[Weak piecewise constant filling]\label{def:weak_filling}
Let $p, q\in \PC(\Omega; Y)$ and $w \in \PC(\Omega\times I; Y)$.
If there exists an isometry $\psi \colon Y\to Y$ such that $w$ is a filling between $p$ and $\psi \circ q$, then $w$ is called a \textit{\textbf{weak type-1 PC filling between}} $p$ and $q$.
If there exists coarsening maps $\kappa_0$ and $\kappa_1$ for $p$ and $q$ such that $w$ is a weak type-1 filling between $\kappa_0\circ p$ and $\kappa_1\circ q$, then $w$ is called a \textbf{weak type-2 PC filling between} $p$ and $q$. 

For $t \in \{1, 2\}$, we let $\mathsf{F}^{(t)}(p, q)$ denote the set of of all weak type-$t$ piecewise constant fillings between $p$ and $q$.
\end{definition}

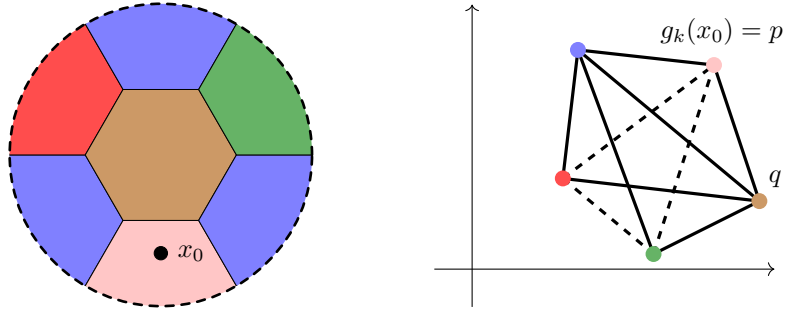
\begin{figure}[ht]
\begin{minipage}[t]{0.33\textwidth}
\begin{center}
\begin{tikzpicture}[
    line join=round,
    line cap=round
]

%------------------------------------------------
% Parameters
%------------------------------------------------
\def\hexradius{1}  % Side length of each hexagon
\def\diskradius{2}  % Radius of the disk
\def\latticerange{7}  % Number of lattice layers to draw

%------------------------------------------------
% Assign colors to individual cells.
%
% Syntax:
%     \sethexcolor{q}{r}{color}
%
% Cells not listed here are white.
%------------------------------------------------
\newcommand{\sethexcolor}[3]{%
    \expandafter\def\csname hexcolor@#1@#2\endcsname{#3}%
}

\sethexcolor{0}{0}{brown!80}
\sethexcolor{1}{0}{green!50!black!60}
\sethexcolor{0}{1}{blue!50}
\sethexcolor{-1}{1}{red!70}
\sethexcolor{-1}{0}{blue!50}
\sethexcolor{0}{-1}{pink!90}
\sethexcolor{1}{-1}{blue!50}

% Return the assigned color, or white if none was assigned.
\newcommand{\hexcolor}[2]{%
    \ifcsname hexcolor@#1@#2\endcsname
        \csname hexcolor@#1@#2\endcsname
    \else
        white%
    \fi
}

%------------------------------------------------
% Draw one hexagonal cell with axial coordinates
% (q,r).
%
% Its Euclidean center is
%
% x = (3/2)sq,
% y = sqrt(3)s(r+q/2),
%
% where s is the hexagon side length.
%------------------------------------------------
\newcommand{\drawhexcell}[2]{%
    \pgfmathsetmacro{\centerx}{1.5*\hexradius*(#1)}
    \pgfmathsetmacro{\centery}{
        sqrt(3)*\hexradius*((#2)+(#1)/2)
    }
    \edef\thiscolor{\hexcolor{#1}{#2}}

    \path[
        draw=black,
        line width=0.35pt,
        fill=\thiscolor
    ]
        ({\centerx+\hexradius*cos(0)},
         {\centery+\hexradius*sin(0)})
        \foreach \angle in {60,120,...,300}{
            --
            ({\centerx+\hexradius*cos(\angle)},
             {\centery+\hexradius*sin(\angle)})
        }
        -- cycle;

    % Uncomment to display the lattice coordinates.
    %
    % \node[font=\tiny] at (\centerx,\centery)
    %     {$(#1,#2)$};
}

%------------------------------------------------
% Clip the entire lattice to the disk
%------------------------------------------------
\begin{scope}
    \clip (0,0) circle[radius=\diskradius];

    \foreach \q in {-\latticerange,...,\latticerange}{
        \foreach \r in {-\latticerange,...,\latticerange}{
            \drawhexcell{\q}{\r}
        }
    }
\end{scope}

% Draw the boundary of the disk on top.
\draw[
    black,
    line width=1pt,
    dashed
]
    (0,0) circle[radius=\diskradius];
\draw[fill=black](0, -1.3) circle[radius=2.5pt];
\draw(.1,-1.3) node[right] {$x_0$};
\end{tikzpicture}
\end{center}
\end{minipage}%
\hspace{10mm}
\begin{minipage}[t]{0.33\textwidth}
\begin{tikzpicture}[scale=1]
    % Axes
    \draw[->] (-0.5,0) -- (4,0);
    \draw[->] (0,-0.5) -- (0,3.5);
    % Point data: name/x/y/color
    \def\pointdata{
        A/1.2/1.2/red!70,
        B/1.4/2.9/blue!50,
        C/2.4/0.2/green!50!black!60,
        D/3.2/2.7/pink!90,
        E/3.8/.9/brown!80
    }
    % Create named coordinates
    \foreach \name/\x/\y/\col in \pointdata {
        \coordinate (\name) at (\x,\y);
    }
    % Draw all edges between distinct pairs
    \draw[black,line width=1.2pt] (A) -- (B);
    \draw[black,line width=1.2pt] (B) -- (D);
    \draw[black,line width=1.2pt] (E) -- (B);
    \draw[black,line width=1.2pt] (E) -- (A);
    \draw[black,line width=1.2pt] (E) -- (D);
    \draw[black,line width=1.2pt] (E) -- (C);
    \draw[black,line width=1.2pt] (C) -- (B);
    \draw[black,line width=1.2pt,dashed] (A) -- (D);
    \draw[black,line width=1.2pt,dashed] (A) -- (C);
    \draw[black,line width=1.2pt,dashed] (C) -- (D);
    % Draw the points on top
    \foreach \name/\x/\y/\col in \pointdata {
        \fill[\col] (\name) circle[radius=3pt];
    }
    \draw(3.3,2.85) node[above] {$g_k(x_0) = p$};
    \draw(3.8,1.2) node[right] {$q$};
\end{tikzpicture}
\end{minipage}
\caption{The disk $\Omega$ on the left is the domain of a piecewise constant function $p$ whose image consists of the $5$ points on the right. The collection of cells of a single color depicts the level set of $p$ at the point with the same color. The isometry $\psi$ used within a weak type 1 filling must preserve the lengths of all line-segments. However within a weak type 2 filling, $\psi\circ\kappa$ only needs to preserve the lengths of the solid line-segments because the cells corresponding to the dashed lines are not adjacent.}
\end{figure}

Suppose that $Y$ is a complete, separable metric space and $\Omega$ is a bounded, connected open subset of $\mathbb{R}^n$ with Lipschitz boundary. 
Whenever $w \in \BV(\Omega\times I; Y)$ Theorem~\ref{thm:approximate-discontinuity-rectifiability} tells us that there exists a Borel map $\unitnormal_w\colon S_w\to \mathbb{S}^n$ where $\unitnormal_w(x)$ is normal to the approximate tangent space $\mathrm{Tan}^n(\mathscr{H}^n\allardell S_w)$ of $S_w$.
The tangential Jacobian $J^{S_w}\pi$ of the orthogonal projection $\pi\colon \mathbb{R}^n\times \mathbb{R} \to \mathbb{R}^n$ is given by $|\langle \verticalvec, \unitnormal_w(x)\rangle|$ for $\mathscr{H}^{n}$ a.e.~$x\in S_w$. 
By the generalized area formula (see \cite[3.2.22]{federer2014geometric}),
\begin{equation}\label{def:projected-g-area}
\VolEx_\delta\brac{w} = \int_{\Omega\times 0} \left(\int_{S_w \cap \pi^{-1}(x)} \delta(w^+, w^-)\,d\mathscr{H}^0\right)\,d\mathscr{L}^n(x) = \int_{S_w} \delta(w^+, w^-)|\langle {\verticalvec}, \unitnormal\rangle|\,d\mathscr{H}^{n}.
\end{equation}
By minimizing $\VolEx_\delta$ over our three different types of fillings, we are able to equip $\PC\paren{\Omega; Y}$ with the \textit{\textbf{projected surface energy}} $\metric_{\mathrm{V}}$, the \textit{\textbf{volume exchange pseudometric}} $\metric^{1}_{\mathrm{V}}$, and the \textit{\textbf{volume exchange dissimilarity}} $\metric^{2}_{\mathrm{V}}$.
They are defined, respectively, by 
\begin{equation}\label{def:volume-exchange-metrics}
\metric_{\mathrm{V}}\paren{p, q} \coloneq \inf_{w \in \mathsf{F}(p, q)}\VolEx_\delta\brac{w}, \quad  
\metric^{1}_{\mathrm{V}}(p, q) \coloneq \inf_{w \in \mathsf{F}^{(1)}(p, q)} \VolEx_{\delta}\brac{w}, \quad 
\metric^{2}_{\mathrm{V}}(p, q) = \inf_{w \in \mathsf{F}^{(2)}(p, q)} \mathscr{E}_{\delta}\,[w].
\end{equation}
for all $p, q \in \PC(\Omega; Y)$. 

\begin{remark}
Directly from the definitions we can see that for any two piecewise constant functions $p, q \colon \Omega \to Y$,
\[
\mathsf{F}(p, q) \subset \mathsf{F}^{(1)}(p, q) \subset \Fill^{(2)}(p, q),
\] 
and hence, $\metric_{\mathrm{V}} \geq \metric_{\mathrm{V}}^1 \geq \metric_{\mathrm{V}}^2.$
\end{remark}

Before we prove Theorem~\ref{thm:l1_energy_for_metric_BV}, let us state an important corollary of it, which says that these three ``metrics'' are given to us by minimizing the $L^1$-norm of $\delta\paren{f, g}$, where $f$ and $g$ are obtained by post-composing $p$ and $q$ with the functions from $Y$ to $Y$ we used to define our fillings. 
\begin{corollary}\label{cor:iso}
Let $\paren{Y,\delta}$ be a complete, separable metric space. If $p, q\colon \Omega \to Y$ are piecewise constant on a bounded, connected, open domain $\Omega\subset\R^n$ then
\begin{align*}
\metric_{\mathrm{V}}\paren{p, q} &= \int_\Omega \delta\paren{p, q}\,d\mathscr{L}^n\\
  \metric^{1}_{\mathrm{V}}(p, q)&= \inf_{\psi_1, \psi_2 \in \Iso(Y)}\int_\Omega\delta(\psi_1\circ p, \psi_2\circ q)\,d\mathscr{L}^n\\
  \metric^{2}_{\mathrm{V}}(p, q) &= \inf_{\psi_1, \psi_2 \in \Iso(Y)}\inf_{\substack{p'\in\mathcal{C}\paren{p},\\q'\in\mathcal{C}\paren{q}}} \int_\Omega\delta(\psi_1\circ p',\psi_2\circ q')\,d\mathscr{L}^n.
\end{align*}
\end{corollary}

\begin{proof}[Proof of Theorem~\ref{thm:l1_energy_for_metric_BV}]
We will first show that for any $w \in \PC(\Omega\times I; Y)$,
\begin{equation}\label{eq:main-equality}
\int_\Omega|Dw_x|(I)\,d\mathscr{L}^{n}(x) = \int_{S_w} \delta\paren{w^+, w^-}\,|\langle \verticalvec, \unitnormal_w\rangle|\,d\mathscr{H}^{n}.
\end{equation}

Theorem~\ref{thm:approximate-discontinuity-rectifiability}, tells us that the approximate discontinuity set $S_w$ of $w$ is countably $\mathscr{H}^n$-rectifiable, and has an orienting Borel map $\unitnormal_w\colon S_w \to \mathbb{S}^{n}$ such that the one sided approximate limits of $w$ exist $\mathscr{H}^{n}$-a.e.~in $S_w$. 
Moreover, it defines Borel maps $w^+, w^-\colon S_w \to Y$. 
Therefore, by the generalized area formula,
\[
\int_{S_w} \delta(w^+, w^-)|\langle\verticalvec, \unitnormal_w\rangle|\,d\mathscr{H}^{n} = \int_\Omega \sum_{\{t : (x, t) \in S_w\}}\delta(w^+(x, t), w^-(x, t))\,d\mathscr{L}^{n}(x).
\]
It remains to show 
\[
|Dw_x|(I) = \sum_{\{t : (x, t) \in S_w\}}\delta(w^+(x, t), w^-(x, t)) \quad \text{ for $\mathscr{L}^n$-a.e.}~x\in \Omega.
\]
Let $\{y_i\}$ be a countable dense subset of $Y$ and define $f_i \coloneq \delta(w, y_i)$ for each $i$.
Since $w \in \PC(\Omega; Y)$, each $f_i$ is a piecewise-constant scalar-valued $\BV$ function. 
The reverse triangle inequality and the existence of one-sided traces $w^\pm$ for $w$ from Theorem~\ref{thm:approximate-discontinuity-rectifiability} implies 
\[
f_i^\pm(x) = \delta(w^\pm, y_i) \quad \text{ for $\mathscr{L}^n$-a.e.}~x\in \Omega.
\]
By \cite[Theorem 3.108, Remark 3.109]{ambrosio2000functions} for almost every $x \in \Omega$,
\[
|D (f_i)_x| = \sum_{\{t : (x, t) \in S_w\}} \Big|\delta\paren{w^+(x, t), y_i) - \delta\paren{w^- (x, t), y_i}}\Big|\,\boldsymbol{\delta}_t \quad \text{ for all } i,
\]
where $\boldsymbol{\delta}_t$ denotes the Dirac delta measure supported on $t$. 

By Lemma~\ref{lem:restriction-is-almost-always-bv}, $(\mathscr{L}^n\allardell \Omega)$-almost every vertical restriction $w_x$ of $w$ is a $\BV$-curve. 
Additionally, for any such $x$, 
\[
|D w_x| = \bigvee_i |D(f_i)_x|.
\]
Indeed, in the last part of the proof of Lemma~\ref{lem:restriction-is-almost-always-bv}, we show that $|Dw_x| \leq \vee_i |D(f_i)_x|$. 
Furthermore, since $\varphi_i(y)\coloneq \delta\paren{y, y_i}$ is 1-Lipschitz and $(f_i)_x = \varphi_i\circ w_x$, the definition of $|Dw_x|$ gives us $|D(f_i)_x| \leq |D w_x|$. 

Thus 
\[
|D w_x|(I) = \Big(\bigvee_i |D(f_i)_x|\Big)(I) = \sum_{\{t : (x, t) \in S_w\}} \sup_i \Big|\delta\paren{w^+(x, t), y_i) - \delta\paren{w^- (x, t), y_i}}\Big|;
\]
and because $\delta(y, y') = \sup_i |\delta(y, y_i) - \delta(y', y_i)|$ for any $y, y' \in Y$, we get that
\[
|Dw_x|(I) = \sum_{\{t : (x, t) \in S_w\}} \delta\paren{w^+(x, t), w^-(x, t)}.
\]
Therefore, 
\[
\int_{S_w} \delta\paren{w^+, w^-}\,|\langle \verticalvec, \unitnormal_w\rangle|\,d\mathscr{H}^{n} = \int_\Omega|Dw_x|(I)\,d\mathscr{L}^{n}(x).
\]

The step function $w^\ast$ (defined in Lemma~\ref{lem:step-function}) corresponding to the piecewise constant trace boundaries $p$ and $q$, is piecewise constant, and hence belongs to $\mathsf{F}(p, q)$. 
The proof of Lemma~\ref{lem:step-function} shows that any step function obtains the minimum in the statement of Theorem~\ref{thm:general-bv-statement}. 
Therefore, restricting the minimization to $\mathsf{F}(p, q)$ leaves the minimum unchanged, and our conclusion follows from equality (\ref{eq:main-equality}).
\end{proof}

Next, we relate these definitions to those given in the introduction. 
\begin{prop}\label{prop:VEFunction}
Let $Y$ be the discrete metric space on $\N$ and let $\cellstrucP=\set{E_\alpha'}_{\alpha\in A}$ and $\cellstrucQ=\set{E_\beta''}_{\beta\in B}$  be cell structures on $\Omega$ where $A,B\subset \N$ satisfy $\abs{\N\setminus A},\abs{\N\setminus B}=\infty.$ If $p,q:\Omega\to Y$ are defined by setting $p\paren{E_\alpha'}=\alpha$ and $q\paren{E_\beta''}=\beta$ for all $\alpha\in A$ and $\beta\in B$ then
$$\metric^{1}_{\mathrm{V}}(\cellstrucP, \cellstrucQ) =\metric^{1}_{\mathrm{V}}(p, q) $$
and
$$\metric^{2}_{\mathrm{V}}(\cellstrucP, \cellstrucQ) =\metric^{2}_{\mathrm{V}}(p, q).$$
\end{prop}\label{prop:minfill}
\begin{proof}
It suffices to show that for $i\in\set{1,2}$ and every $\cellstrucW\in \Fill^{i}\paren{\cellstrucP,\cellstrucQ}$ there is a $w\in \mathsf{F}^i\paren{p,q}$ with 
$$\VolEx\paren{\cellstrucW}=\VolEx_{\delta}\brac{w}$$
and vice versa.

We begin with the case $i=1.$ 
First, let $\cellstrucW\in \Fill^{1}\paren{\cellstrucP,\cellstrucQ}$ and denote its cells by $\set{E_{\gamma}}_{\gamma\in C}.$  For convenience, we treat $\cellstrucP$ and $\cellstrucQ$ as partitions of $\Omega\times\set{0}$ and $\Omega\times\set{1},$ respectively. By definition, the cells of $\cellstrucP$ are --- up $\mathscr{H}^n$-negligible sets --- the $\mathscr{H}^{n}$-non-negligible intersections $\partial^{\ast} E\cap   \Omega\times\set{0}$ for cells $E\in\cellstrucW,$ and analogously for $\cellstrucQ.$ Say that a cell $E$ of $\cellstrucW$ corresponds to a cell $E'$ of $\cellstrucP$ or $\cellstrucQ$ if $\mathscr{H}^{n}\paren{\partial^{\ast}E\cap E'}>0.$ We may reindex the cells of $\cellstrucW$ if necessary so that $E_{\alpha}$ corresponds to $E_{\alpha}'$ for all $\alpha \in A,$ where we are implicitly using that $\abs{\N\setminus A}=\infty.$  Choose a permutation $\phi:Y\to Y$ so that $E_{\phi^{-1}\paren{\beta}}$ is the cell of $\cellstrucW$ corresponding to $E_{\beta}''$ for all $\beta \in B.$  

Define $w:\Omega\times I\to Y$ by $w\paren{E_\gamma}=\gamma.$ 
By construction, $w\mid_{\Omega\times 0}=p$ and $w\mid_{\Omega\times 1}=\phi^{-1}\circ q$,  That is, $w$ defines a weak type-1 PC filling between $p$ and $q.$ Comparing Equations~\ref{def:projected-area} and~\ref{def:surface-energy}, we have that
$$\VolEx\paren{\cellstrucW}=\VolEx_{\delta}\brac{w}.$$

On the other hand, let $w\in \mathsf{F}^1\paren{p,q}$ and let $\cellstrucW$ be the partition of $\Omega\times I$ into level sets of $w.$ By construction, $\cellstrucW\in \Fill^{1}\paren{\cellstrucP,\cellstrucQ}$ and  
$$\VolEx\paren{\cellstrucW}=\VolEx_{\delta}\brac{w}.$$

Now, let $i=2$ and $\cellstrucW\in \Fill^{2}\paren{\cellstrucP,\cellstrucQ}$ so $\cellstrucW$ is a type-1 cellular filling between coarsenings of $\cellstrucP'$ of $\cellstrucP$ and $\cellstrucQ'$ of $\cellstrucQ.$ Index the cells of $\cellstrucP'$ and $\cellstrucQ'$ by index sets $A'\subset A$ and $B'\subset B$ so the index of a cell equals the indices of one of the constituent cells from the respective finer partition.  Let $p':\Omega\to Y,q':\Omega\to Y$ be as in the statement but for $\cellstrucP'$ and $\cellstrucQ'$ rather than $\cellstrucP$ and $\cellstrucQ.$ Then $p'$ and $q'$ are coarsenings of $p$ and $q,$ respectively. Applying the same argument as in the case $i=1$ produces a permutation $\phi:Y\to Y$ and a PC filling $w:\Omega\times(0,1)\to Y$ between $p'$ and $\phi\circ q'$ so that 
$$\VolEx\paren{\cellstrucW}=\VolEx_{\delta}\brac{w}.$$
Since $p'$ and $q'$ are coarsenings of $p$ and $q,$ $w$ is the desired weak type-2 PC filling between $p$ and $q.$ 
The argument for the other direction is identical to that for $i=1.$ 
%and denote its cells by $E_1,\ldots,E_k.$ Suppose that the non-negligible components of $E_{j}\cap \Omega\times\set{0}$ and $E_{j}\cap \Omega\times\set{1}$  are $C_j^1\ldots,C_j^{m_j}$ and $D_j^1\ldots,D_j^{m_l},$ respectively (where we allow $m_j=0$ and/or $m_l=0$). Now, define $p':\Omega\to Y$ by setting $p'\paren{C_j^l}=p\paren{C_j^1}$ where we allow $j$ to range from $1$ to $k$ and $m$ to range from $1$ to $m_j.$ Define $q'$ analogously by setting $q'\paren{D_j^l}=q\paren{D_j^1}$ for each $j$ and $l.$ Then $p'$ and $q'$ are coarsenings of $p$ and $q.$ Repeating the construction from the first part of the argument, we obtain a weak type-1 PC filling $w$ between $p'$ and $q'$ --- which is then a weak type-2 filling between $p$ and $q$ --- so that
%$$\VolEx\paren{\cellstrucW}=\VolEx_{\delta}\paren{w}.$$
\end{proof}

\section{The Discrete Volume Exchange Metric and Dissimilarity}
\label{sec:volume_discrete}
% \subsection{A Dual Formulation}

Here, we define a discrete analog of the continuous volume exchange dissimilarities, prove their equality, and describe algorithms to compute them. Theorem~\ref{thm:main} reduce these a priori continuous problems to discrete optimizations, and Theorem~\ref{thm:dual-equality} converts the former into a linear assignment problem. We begin by reviewing the discrete analogue of the one-to-one volume exchange metric.

For convenience, take $I, J\subset \N$  to be index sets. 
Let $M = \{U_i\}_{i \in I}$ and $N = \{V_j\}_{j \in J}$ be Caccioppoli partitions of an ambient space $\Omega \subset \mathbb{R}^n$, and $\mathrm{PM}(I,J)$ be the collection of all partial matchings between them. 
By Definition~\ref{def:partial-matching} each index in $I\cup J$ is seen exactly once in the partial matching. 
For the following definition, let $U_\ast, V_\ast = \emptyset $. 

As in the introduction, recall that the \textit{discrete one-to-one volume exchange metric} between cell structures $\cellstrucP$ and $\cellstrucQ$ is defined as
\[
\metric_{\mathrm{DV}}^{1}(\cellstrucP, \cellstrucQ) \coloneq \frac{1}{2}\min_{\partialmatching \in \PM(\cellstrucP, \cellstrucQ)}\sum_{(E, F) \in \partialmatching} \mathscr{L}^n(E \triangle F)
\]
where we will justify that the minimum exists below.
% where $E\triangle F \coloneq (A\setminus B) \cup (B\setminus A)$ denotes the \textbf{symmetric difference} between sets $A$ and $B$. 

\begin{lemma}\label{lem:l1_sym_diff}
If $(Y, \delta)$ is a countable discrete metric space and $p, q \in \PC(\Omega; Y)$, then 
\[
\int_\Omega \delta(p, q)\,d\mathscr{L}^n = \frac{1}{2}\sum_{i \in Y}\mathscr{L}^n(\{p = i\}\triangle \{q = i\}).
\]
\end{lemma}
\begin{proof} Our piecewise constant functions $p$ and $q$ provide us with two partitions $\cellstrucP = \{E_i\}$ and $\cellstrucQ = \{F_i\}$ of $\Omega$, given by the level sets $E_i \coloneq p^{-1}(i)$ and $F_i \coloneq q^{-1}(i)$ over all $i \in Y$. By definition, the nonempty elements of $\cellstrucP$ and $\cellstrucQ$ form cell structures on $\Omega.$ 
With these partitions, observe that 
\[
    \sum_{i \in Y} \int_{E_i} \delta(p, q)\,d\mathscr{L}^n = \int_\Omega \delta(p, q)\,d\mathscr{L}^n = \sum_{i \in Y} \int_{F_i} \delta(p, q)\,d\mathscr{L}^n.
\]
Since
\begin{align*}
\int_\Omega \delta(p, q)\,d\mathscr{L}^n = \sum_{i,j \in Y} \delta(i, j)\mathscr{L}^n\paren{E_i\cap F_j}
\end{align*}
it follows that
\begin{align*}
\int_{E_i} \delta(p, q)\,d\mathscr{L}^n = \sum_{j \in q(E_i)} \delta(i, j)\mathscr{L}^n(E_i\cap F_j).
\end{align*} 
However, if $j=i$ for $j \in q(E_i)$ then $\delta\paren{i,j} = 0$. Thus, 
\[
    \sum_{j \in q(E_i)} \delta(i, j)\mathscr{L}^n(E_i\cap F_j) = \mathscr{L}^n(E_i\setminus F_i).
\]
So, 
\begin{align*}
    \int_\Omega \delta(p, q)\,d\mathscr{L}^n &= \frac{1}{2} \paren{\sum_{i \in Y} \int_{E_i} \delta(p, q)\,d\mathscr{L}^n + \sum_{j \in Y} \int_{F_j} \delta(p, q)\,d\mathscr{L}^n} \\
    &= \frac{1}{2} \paren{\sum_{i \in Y} \mathscr{L}^n(E_i\setminus F_i) + \sum_{j \in Y} \mathscr{L}^n(F_j\setminus E_j)} \\
    &= \frac{1}{2} \sum_{i \in Y} \paren{\mathscr{L}^n(E_i\setminus F_i) + \mathscr{L}^n(F_i\setminus E_i)} \\
    &= \frac{1}{2}\sum_{i \in Y}\mathscr{L}^n\paren{E_i\triangle F_i}.
\end{align*}
\end{proof}
\begin{proof}[Proof of Theorem~\ref{thm:main}]

Let $\cellstrucP=\set{E_\alpha'}_{\alpha\in A}$ and $\cellstrucQ=\set{E_\beta''}_{\beta\in B}$  be cell structures on $\Omega$ and denote by $\paren{Y,\delta}$ the discrete metric space on $\N,$ where we have chosen the index sets $A,B\subset \N$ so that $\abs{\N\setminus A}=\infty$ and $\abs{\N\setminus B}=\infty.$  Define piecewise constant functions $p,q:\Omega\to Y$ by setting $p\paren{E_\alpha'}=\alpha$ and $q\paren{E_\beta''}=\beta$ for $\alpha\in A,\beta\in B.$ By Proposition~\ref{prop:VEFunction} it suffices to show that 
$$\metric^{i}_{\mathrm{V}}(p, q)=\metric_{DV}^i\paren{\cellstrucP,\cellstrucQ}$$
for $i\in\set{1,2}.$ %As such, we assume that $k\geq \abs{A}+\abs{B}.$ For notational convenience we suppress the dependence of $Y_k,p_k,$ and $q_k$ on $k.$

% Let $(Y', \delta)$ be a countable, discrete metric space and have $\cellstrucP=\{E_\alpha\}_{\alpha \in A}$, $\cellstrucQ=\{F_\beta\}_{\beta \in B}$ be cell structures whose index sets $A,B$ are finite subsets of $Y'$. Recall that for each partition, the functions $p,q\colon \Omega \to Y'$ are piecewise constant on $\Omega$ if they are of the form \[p = \sum_{\alpha \in A}y_\alpha \chi_{E_{\alpha}} \quad \text{and} \quad q = \sum_{\beta \in B}y_\beta \chi_{F_{\beta}} \] where $\{y_{\alpha}\}_{\alpha \in A} , \{y_{\beta}\}_{\beta \in B} \subset Y'$. In the proof, we will apply Theorem~\ref{thm:l1_energy_for_metric_BV} to the functions $p,q.$ That theorem requires that the codomain of the functions and the piecewise constant fillings between them be proper. $Y'$ is not proper. This is not a problem; since a cellular filling is required to have finitely many cells it suffices to show that for any finite $Y\subset Y'$ containing $\im p\cup \im q$ with sufficiently many elements we have that
% $$\metric_{V}^i\paren{p_Y,q_Y}=\metric_{DV}^i\paren{\cellstrucP,\cellstrucQ}$$
% for $i\in\set{1,2}$ when $p_Y,q_Y$ are obtained by restricted the domain of the functions to $Y.$ As such, we assume that $Y$ is any such subset of $Y'$ with at least $\abs{A}+\abs{B}$ elements. For notational convenience we refer to $p_Y$ and $q_Y$ as $p$ and $q$ below. 

{\it Proof that $\metric_{\mathrm{DV}}^1=\metric_{\mathrm{V}}^1$}. Let $\partialmatching \in \PM(\cellstrucP,\cellstrucQ)$ %We aim to construct an isometry $\varphi_\partialmatching \in \mathrm{Iso}(Y)$ corresponding to $\partialmatching$. Notably, $\varphi_{\partialmatching}\colon Y \to Y$ is an isometry of $(Y, \delta)$ if and only if $\varphi_{\partialmatching}$ is a bijection. 
and let $f_A \colon A \to Y$ and $f_B \colon B \to Y$ be injective functions with the following properties:

\begin{itemize}
    \item[(i)] $f_A(\alpha)=f_B(\beta)$ if and only if $(E_\alpha',E_\beta'') \in \partialmatching$;
    \item[(ii)] $f_A(\alpha)\neq f_B(\beta)$ for any $\beta \in B$ if $(E_\alpha',\ast) \in \partialmatching$;
    \item[(iii)] $f_B(\beta)\neq f_A(\alpha)$ for any $\alpha \in A$ if $(\ast, E_\beta'') \in \partialmatching$.
    \item(iv)] $\abs{\N\setminus f_A\paren{A}},\abs{\N\setminus f_B\paren{B}}=\infty.$
\end{itemize}
 % Set $X = \mathrm{im}(f_A)\cup \mathrm{im}(f_B).$
 % We construct an isometry $\varphi$ using a bijective function $\psi \colon Y\setminus X \to Y\setminus X$ such that $\tilde{\varphi}\restriction_{X} = \varphi$ and $\tilde{\varphi}\restriction_{Y \setminus X} = \psi$.
Extend $f_A$ and $f_B$ to bijective functions from $\N$ to $\N.$ By construction, $f_A\circ p, f_B\circ q \in \mathrm{PC}(\Omega;Y)$. We have that
 \begin{align*}
\metric_{V}^1\paren{p,q}&=\metric_V^1\paren{f_A\circ p, f_B\circ q}\\
    &= \inf_{\varphi \in \mathrm{Iso}(Y)} \int_{\Omega} \delta(f_A\circ p, \varphi \circ ( f_B\circ q)) \,d\mathscr{L}^n &&\text{by Corollary~\ref{cor:iso}}\\
     &\leq  \int_{\Omega} \delta(f_A\circ p, \mathrm{Id}_Y \circ ( f_B\circ q)) \,d\mathscr{L}^n \\
    &= \frac{1}{2} \sum_{i\in Y} \mathscr{L}^n(\set{f_A\circ p = i}\triangle\set{ f_B\circ q = i})  &&\text{by Lemma~\ref{lem:l1_sym_diff}}\\
     &= \frac{1}{2}\sum_{(E,F) \in \partialmatching}\mathscr{L}^n(E\triangle F).
\end{align*}
Since this holds for all $\partialmatching\in\PM\paren{\mathcal{P},\mathcal{Q}},$ $\metric_{\mathrm{V}}^{1} \leq \metric_{\mathrm{DV}}^1$.

 %By application of Lemma \ref{lem:l1_sym_diff} and Equation \ref{eq:one-sided-iso}, it follows that $\tilde{\varphi}$ is candidate minimizer of $\metric_V^{(1)}(f_A\circ p,f_B\circ q)$ as
To show the other inequality, let $\varphi\in \mathrm{Iso}(Y).$% and let $A=\mathrm{im}(p)$ and $B=\mathrm{im}(\varphi \circ q)$. The functions partition the domain into cell structures \[\Omega = \bigsqcup_{\alpha \in A} E_{\alpha} = \bigsqcup_{\beta \in B} E_{\beta}'\] where $E_{\alpha} = p^{-1}(\alpha)$ and $E_{\beta}'' = ({\varphi \circ q})^{-1}(\beta)$.
%By Lemma \ref{lem:l1_sym_diff}
%\[\begin{aligned}\int_{\Omega}\delta(p,\varphi\circ q)\,d\mathscr{L}^n &= \frac{1}{2} \sum_{i \in Y}\mathscr{L}^n(\set{p=i}\triangle\set{(\varphi\circ q)=i}).
%\end{aligned}\] 
We construct a partial matching $\partialmatching_{\phi}\in \PM\paren{\cellstrucP,\cellstrucQ}$ as follows: 

\begin{itemize}
    \item[(i)] if $\mathscr{L}^n(\set{p=i})$ and $\mathscr{L}^n(\set{(\varphi\circ q)=i})$ are nonzero, append $(E_i',E_{\phi^{-1}\paren{i}}'') \in \partialmatching_{\varphi}$;
    \item[(ii)] if $\mathscr{L}^n(\set{p=i})$ is nonzero and $\mathscr{L}^n(\set{(\varphi\circ q)=i})=0$, append $(E_i',\ast) \in \partialmatching_{\varphi}$;
    \item[(iii)] append $(\ast, E_{\phi^{-1}\paren{i}}'') \in \partialmatching_{\varphi}$ if the opposite of (ii) is true.
\end{itemize}

Thus
\[\begin{aligned}
    \metric_{\mathrm{V}}^1(p,q) &=\inf_{\varphi\in \Iso\paren{Y}} \int_{\Omega}\delta(p,\varphi\circ q)\,d\mathscr{L}^n  &&\text{by Corollary~\ref{cor:iso}}\\
    &= \frac{1}{2} \inf_{\varphi\in \Iso\paren{Y}}\sum_{i \in Y}\mathscr{L}^n(\set{p=i}\triangle\set{\varphi\circ q=i}) &&\text{by Lemma~\ref{lem:l1_sym_diff}} \\
    &= \frac{1}{2} \inf_{\varphi\in \Iso\paren{Y}}\sum_{(E',E'') \in \partialmatching_{\varphi}}\mathscr{L}^n\paren{E' \triangle E''} \\
    &\geq \metric_{\mathrm{DV}}^1\paren{\cellstrucP,\cellstrucQ}.
\end{aligned}\] 
This concludes our proof that $\metric_{\mathrm{V}}^{1} = \metric_{\mathrm{DV}}^1$.

\textit{Proof that $\metric_{\mathrm{V}}^{2} = \metric_{\mathrm{DV}}^2$.} 
Denote by $\mathcal{C}\paren{\cellstrucP}$ and $\mathcal{C}\paren{p}$ the collections of coarsenings of a cell structure $\cellstrucP$ and a piecewise constant function $p,$ respectively. Also, let $\mathcal{S}\paren{p}$ be the cell structure defined by the level sets of a piecewise constant function $p.$ By the previous case,
$$ \metric^{2}_V(p,q)=\inf_{p'\in \mathcal{C}\paren{p},q'\in\mathcal{C}\paren{q}}\metric^{1}_V(p',q')=\inf_{p'\in \mathcal{C}\paren{p},q'\in\mathcal{C}\paren{q}}\metric_{\mathrm{DV}}^1\paren{\mathcal{S}\paren{p'},\mathcal{S}\paren{q'}}\,.$$
On the other, by definition,
$$\metric_{\mathrm{DV}}^{2}(\cellstrucP, \cellstrucQ)=\inf_{\cellstrucP'\in\mathcal{C}\paren{\cellstrucP},\cellstrucQ'\in\mathcal{C}\paren{\cellstrucQ}}\metric_{\mathrm{DV}}^1\paren{\cellstrucP',\cellstrucQ'}.$$ 
 Thus it suffices to show that if $\cellstrucP' \in \mathcal{C}(\cellstrucP)$ then there exists a $p' \in \mathcal{C}(p)$ such that $\mathcal{S}(p')$ corresponds to $\cellstrucP '$ and vice versa. We begin with a coarsened cell structure $\cellstrucP '$. Let $p \in \PC(\Omega; Y)$ be the index map defined at the beginning of the proof. We construct a coarsening $p'\in\mathcal{C}\paren{p}$ as follows. Each cell $F\in\cellstrucP'$ is union of cells of $\cellstrucP$: $F'=\sqcup_{\alpha\in A\paren{F}}E_{\alpha}'.$ Choose $\alpha\in A\paren{F}$ and set $p'\paren{F}=\alpha.$ Then $p'$ is a coarsening of $p$ whose level sets are the cells of $\cellstrucP'.$

% For each $F_i' \in \cellstrucP'$ we construct a coarsening $p' \in \mathcal{C}(p)$ as follows: \[ p'(F'_i) = \begin{cases}
%     i_1 & \text{if } E_{i}' = E_i \text{ for some $E_i \in \cellstrucP$ }\\
%     i_1   & \text{when } E_i'=E_{i_1}\sqcup E_{i_2} \sqcup \cdots \sqcup E_{i_n}
% \end{cases}.\]
%That is, $p'=p$ when the input is a single component, otherwise it returns the index of an arbitrary cell in the union. By construction $\mathcal{S}(p') =\cellstrucP '$ only differ up to a set of measure zero.  

Conversely, let $p' \in \mathcal{C}(p)$. By definition, the inverse image of its level set is the inverse image of a union of non-adjacent level-sets. As such, $p'$ partitions the domain into
\[\Omega = \bigsqcup_{\alpha \in A}\Omega_{\alpha}^{p'}\] where $\Omega_{\alpha}^{p'} = (p')^{-1}(\alpha) = \bigcup_{i\in I} p^{-1}(\set{\alpha_i})$ for some $I \subset A$. Since each cell $\Omega_{\alpha}^{p'}$ is a union of non-adjacent cells of $\cellstrucP$, it follows that $\mathcal{S}(p') \in \mathcal{C}(\cellstrucP).$
\end{proof}

\begin{prop}\label{prop:min}
Let $\Omega$ be a open, connected, bounded subset of $\R^n$ and let $\cellstrucP$ and $\cellstrucQ$ be cell structures of $\Omega.$  
The minima in Definitions~\ref{defn:continuous} and Definition~\ref{defn:discrete} are realized.
\end{prop}

\begin{proof}
By Theorem~\ref{thm:main} and Proposition~\ref{prop:minfill} it suffices to show that the minima in Definition~\ref{defn:discrete} are realized.
We begin by considering $\metric_{DV}^1.$ Suppose we partition $\Omega$ into $\cellstrucP = \{E_i\}_{i \in A}$ and $\cellstrucQ = \{F_j\}_{j \in B}$, where $A$ and $B$ are both countably infinite subsets of $\mathbb{N}.$ The argument is similar when one of $A$ or $B$ is finite. 
Pick a minimizing sequence $\{\partialmatching_k\}_{k \in \mathbb{N}} \subset \PM(\cellstrucP, \cellstrucQ)$ of 
\[
L(\partialmatching) \coloneq \sum_{(E, F) \in \partialmatching}\mathscr{L}^n(E\triangle F)
\]
over all $\partialmatching \in \PM(\cellstrucP, \cellstrucQ)$.
Before we construct a candidate $\partialmatching^\ast \in \PM\paren{\cellstrucP, \cellstrucQ}$ and show it obtains the infimum, we first note that for any $\partialmatching \in \PM\paren{\cellstrucP, \cellstrucQ}$,
\begin{equation}\label{eq:finite_approx}
|L(\partialmatching) - L_N(\partialmatching)| \leq \sum_{i > N} \mathscr{L}^n(E_i) + \sum_{j > N} \mathscr{L}^n(F_j) \longrightarrow 0 \quad \text{ as } N\to \infty,
\end{equation}
where 
\[
L_N(\partialmatching) \coloneq \sum_{\substack{(E_i, F_j) \in \partialmatching\\ i,j \leq N}} \mathscr{L}^n(E_i\triangle F_j) + \sum_{\substack{i \leq N\\ \not\exists j\leq N: \paren{E_i,F_j}\in  \partialmatching}} \mathscr{L}^n(E_i) + \sum_{\substack{j \leq N\\  \not\exists i\leq N: \paren{E_i,F_j}\in  \partialmatching}} \mathscr{L}^n(F_j).
\]

For any $(i,j,k) \in A\times B\times \mathbb{N}$, we define
\[
c_{ijk} = 
\begin{cases}
1 & \text{ if } (i,j) \in \partialmatching_k\\
0 & \text{ otherwise}
\end{cases}
\]
to encode whether $E_i$ and $F_j$ were matched at step $k$ of the minimizing sequence.  
Enumerate $A\times B$ with $\{(i_m, j_m)\}_{m = 1}^\infty$.
Since $c_{i_1, j_1, k} \in \{0, 1\}$, we can pick an infinite subset $N_1$ of $\mathbb{N}$ such that $c_{i_1, j_1, k}$ is constant for $k \in N_1$; denote this constant by $c_1$ and pick any such $k_1 \in N_1$.  
Similarly, there is an infinite subset $N_2$ of $N_1$ such that $c_2 \coloneq c_{i_2, j_2, k}$ is constant for $k \in N_2$; pick a $k_2 \in N_2$ such that $k_2 > k_1$. 
We repeat this process for all $m \in \{1, 2, \dots\}$ to obtain $\{c_m\}$ and $k_1 < k_2 < \cdots$.

We define $\partialmatching^\ast$ with the following procedure. 
In the first step, we let $\partialmatching^\ast$ contain all the pairs $(i_m, j_m) \in A\times B$ for which $c_{m} = 1$.
Then, if $i' \in A$ is not present in the first entry of any of the pairs just added, we add $(i', \ast)$ to $\partialmatching^\ast$. 
Similarly, if $j' \in B$ is not present in the second entry of any of the any pair added in the first step, we add $(\ast, j')$ to $\partialmatching^\ast$. 

Let $\varepsilon > 0$. 
By (\ref{eq:finite_approx}) we can choose $N$ large enough so that 
\begin{align*}
|L(\partialmatching^\ast) - L(\partialmatching_{k_m})| &\leq |L(\partialmatching^\ast) - L_N(\partialmatching^\ast)| + |L_N(\partialmatching^\ast) - L_N(\partialmatching_{k_m})| + |L_N(\partialmatching_{k_m}) - L(\partialmatching_{k_m})|\\
&\leq \varepsilon + |L_N(\partialmatching^\ast) - L_N(\partialmatching_{k_m})|.
\end{align*}
Because the first summation used to define $L_N$ is a sum over (at most) $N^2$ pairs in $\partialmatching$, for large enough $m$, the three families we sum over for  $L_N(\partialmatching^\ast)$ are the same three families we use for $L_N(\partialmatching_{k_m})$. 
Letting $\varepsilon \to 0$ concludes this portion of the proof.

% We now consider the $\metric_{DV}^2.$ Partition $\Omega$ into $\cellstrucP = \{E_i\}_{i \in \N}$ and $\cellstrucQ = \{F_j\}_{j \in \N}.$ We establish notation related to an independent matching $\partialmatching\in \IM\paren{\cellstrucP,\cellstrucQ}.$ Recall that $\partialmatching$ is a partial matching between coarsenings of $\cellstrucP'$ of $\cellstrucP$ and $\cellstrucQ'$ of $\cellstrucQ.$ For $\paren{i,j}\in\N\times \N,$ let $E$ be the cell of $\cellstrucP'$ so that $E_i\subset E$ and let $F$ be the cell of $\cellstrucQ'$ so that $F_j\subset F.$ We abuse notation and write $\paren{i,j}\in\partialmatching$ if $\paren{E,F}\in\partialmatching.$ Observe that $\partialmatching$ is determined by the pairs of indices that satisfy this relation. As in Theorem~\ref{thm:dual-equality} below, we have that
% $$L\paren{\partialmatching}\coloneq \frac{1}{2}\sum_{(E, F) \in \partialmatching} \mathscr{L}^n(E_i\triangle F_j)=\mathscr{L}^n(\Omega)-\sum_{\paren{i,j}\in\partialmatching}\mathscr{L}^n(E_i\cap F_j)\,.$$
% The remainder of argument continues mutatis mutandis as the one for $\metric_{DV}^1.$ 

We now consider the $\metric_{\mathrm{DV}}^2.$ 
Pick a sequence of coarsenings $\{\cellstrucP^k\}_{k = 1}^\infty$ of $\cellstrucP$ and coarsenings $\{\cellstrucQ^k\}_{k = 1}^\infty$ of $\cellstrucQ$ so that 
\[
\metric_{DV}^1(\cellstrucP^k, \cellstrucQ^k) \longrightarrow \metric_{DV}^2(\cellstrucP, \cellstrucQ) \quad \text{ as } k \to \infty.
\]
We will show that there exist ``limit'' coarsenings of $\{\cellstrucP^k\}$ and $\{\cellstrucQ^k\}$ denoted by $\cellstrucP^\infty$ and $\cellstrucQ^\infty$, respectively, which satisfy
\[
\metric_{DV}^1(\cellstrucP^\infty, \cellstrucQ^\infty) = \metric_{DV}^2(\cellstrucP, \cellstrucQ). 
\]

We first define $\cellstrucP^\infty$. 
Any coarsening $\cellstrucP^k$ defines an equivalence relation $\sim_k$ on $A$, where $i \sim_k j$ if there is a cell $E \in \cellstrucP^k$ so that $(E_i \cup E_j) \subset E$. 
For any $(i, j, k) \in A\times A\times \mathbb{N}$, we define
\[
c^\cellstrucP_{ijk} \coloneq
\begin{cases}
1 & \text{ if } i \sim_k j\\
0 & \text{ otherwise }
\end{cases}.
\]
Enumerating $A\times A$ with $\set{(i_m, j_m)}$ we use a similar diagonalization argument as in the case of $\metric_{DV}^1$ to obtain a sequence $\{k_m\}$ such that  $c_m^\cellstrucP \coloneq c^\cellstrucP_{i_m j_m k_l}$ is a constant function of $l \geq m$.
This allows us to define the equivalence relation $\sim$ where $i_m \sim j_m$ if $c_m^\cellstrucP = 1$, and hence, the coarsening 
\[
\cellstrucP^\infty \coloneq \Big\{E_{[i]} \coloneq \bigcup_{j \in [i]} E_j : [i] \in A/{\sim}\Big\}
\]
of $\cellstrucP$. 
Apply the same diagonal argument for $\cellstrucQ$ along the subsequence already selected for $\cellstrucP$.
Notice that adjacent cells in $\cellstrucP$ (and $\cellstrucP$) cannot belong to the same limiting class, otherwise, they would belong to the same class at a sufficiently large $k$, which is a contradiction. 
By subadditivity and lower semicontinuity of perimeter $P$, $\cellstrucP^\infty$ and $\cellstrucQ^\infty$ are Caccioppoli partitions of $\Omega$. 

Now consider the first $N$ cells $\{E_1, \dots, E_N\}$ of $\cellstrucP$.
Because for all sufficiently large $k\in\set{k_m}$, we have that $i \sim_k j$ if and only if $i \sim j$ (for $i,j \in \{1, \dots, N\}$), we have a cost zero matching between 
\[
\cellstrucP^k|_{1, \dots,N} \coloneq \Big\{E_{[i]} : [i] \in \{1, \dots, N\}/{\sim_k}\Big\}
\]
and $\cellstrucP^\infty|_{1, \dots, N}$ (similarly defined). 
For such $k$, we get
\[
\metric_{DV}^1(\cellstrucP^k, \cellstrucP^\infty) \leq 2\,\sum_{i > N} \mathscr{L}^n(E_i) \longrightarrow 0 \quad \text{ as } N\to \infty.
\]
A similar argument also gives us $\metric_{DV}^1(\cellstrucQ^k, \cellstrucQ^\infty) \longrightarrow 0$ as $N\to \infty$.
Because $\metric_{DV}^1$ satisfies the triangle inequality, 
\[
\metric_{DV}^1(\cellstrucP^\infty, \cellstrucQ^\infty) \leq \metric_{DV}^1(\cellstrucP^\infty, \cellstrucP^k) + \metric_{DV}^1(\cellstrucP^k, \cellstrucQ^k) + \metric_{DV}^1(\cellstrucQ^k, \cellstrucQ^\infty). 
\]
Recalling that $\lim_{k\to \infty}\metric_{DV}^1(\cellstrucP^k, \cellstrucQ^k) = \metric_{DV}^2(\cellstrucP, \cellstrucQ)$, letting $k\to \infty$ in the inequality above gives us that 
\[
\metric_{DV}^1(\cellstrucP^\infty, \cellstrucQ^\infty) \leq \metric_{DV}^2(\cellstrucP, \cellstrucQ).
\]
Now, since $\cellstrucP^\infty$ and $\cellstrucQ^\infty$ are coarsenings of $\cellstrucP$ and $\cellstrucQ$, by definition, $\metric_{DV}^2(\cellstrucP, \cellstrucQ) \leq \metric_{DV}^1(\cellstrucP^\infty, \cellstrucQ^\infty)$. 
By our previous result for the existence result for $\metric_{DV}^1$, there exists an optimal partial matching between $\cellstrucP^\infty$ and $\cellstrucQ^\infty$. 
This concludes the proof.
\end{proof}

\subsection{Metric Properties}\label{subsec:metric}

Recall that if $\metric$ is non-negative, symmetric, and satisfies $\metric\paren{M,M}=0$ for all $M$ then it is called a dissimilarity.
\begin{prop}
    $\metric_{\mathrm{V}}^1$ is a metric and $\metric_{\mathrm{V}}^2$ is a dissimilarity but does not satisfy the triangle inequality. % on $\mathrm{PC}(\Omega; \mathbb{R}^m).$ 
    \label{prop:dissimilarity}
\end{prop}
\begin{proof}

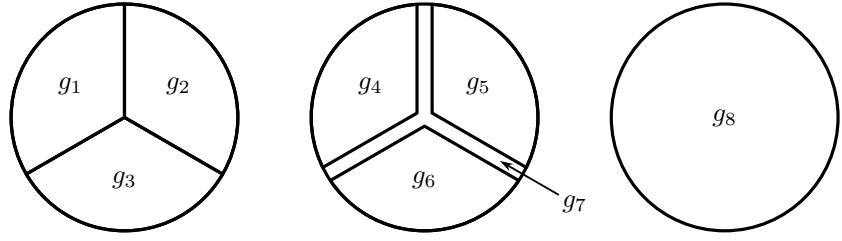
\begin{figure}
    \centering
\begin{tikzpicture}[
    boundary/.style = {line width=1.2pt, line join=round, line cap=round},
    pointer/.style  = {line width=0.7pt, -{Stealth[length=5pt, width=3pt]}},
    % grains : pic
    %   Draws a disk cut into three grains by a Y (branches up, lower left and
    %   lower right). A positive tube width pulls the grains apart so that a
    %   thin Y-shaped fourth grain, the tube, appears between them.
    %   Parameters
    %   ----------
    %   #1 : number (cm)
    %       Width of the inner tube; 0 gives a plain triple junction.
    %   Returns
    %   -------
    %   None (draws the disk centered at the pic position).
    pics/grains/.style = {code = {
        \pgfmathsetmacro{\halfw}{#1/2}
        \pgfmathsetmacro{\rimoff}{asin(\halfw/\diskR)} % tube half-width at rim (deg)
        \foreach \angA/\angB in {90/210, 210/330, 330/450} {
            \pgfmathsetmacro{\bis}{(\angA+\angB)/2}                  % grain bisector
            \pgfmathsetmacro{\rcorner}{\halfw/sin((\angB-\angA)/2)} % inner corner
            \draw[boundary]
                (\bis:\rcorner) -- ({\angA+\rimoff}:\diskR)
                arc[start angle={\angA+\rimoff}, end angle={\angB-\rimoff},
                    radius=\diskR] -- cycle;
        }
        \draw[boundary] (0,0) circle[radius=\diskR];
    }},
]
    % ---- layout (cm) ----
    \def\diskR{1.5}     % disk radius
    \def\tubeW{0.20}    % tube width in B: thin, so its volume is close to 0
    \def\spacing{3.97}  % center-to-center distance (same ratio to size as DVE)
    \coordinate (A) at (0, 0);
    \coordinate (B) at (\spacing, 0);
    \coordinate (C) at ({2*\spacing}, 0);
 
    \pic at (A) {grains=0};
    \pic at (B) {grains=\tubeW};
    \draw[boundary] (C) circle[radius=\diskR];
 
    % ---- cell labels g_1 .. g_8: upper left, upper right, bottom of each Y ----
    \pgfmathsetmacro{\labR}{0.55*\diskR} % centroid distance of a 120-degree sector
    \foreach \disk/\numL/\numR/\numB in {A/1/2/3, B/4/5/6}
        \path (\disk) ++(150:\labR) node {$g_{\numL}$}
              (\disk) ++( 30:\labR) node {$g_{\numR}$}
              (\disk) ++(270:\labR) node {$g_{\numB}$};
    \node at (C) {$g_8$};
 
    % ---- cell g_7, the tube: label outside, arrow along its lower-right arm ----
    \def\arrowOut{0.55} % arrow length outside the rim
    \def\arrowIn{0.35}  % how far the tip reaches into the tube
    \path (B) ++(330:{\diskR+\arrowOut}) coordinate (tail7);
    \draw[pointer] (tail7) node[anchor=150, inner sep=1pt] {$g_7$}
        -- ++(150:{\arrowOut+\arrowIn});
 
    % Optional disk labels: uncomment to put A, B, C under the disks.
    % \foreach \disk in {A, B, C}
    %     \path (\disk) ++(0,-\diskR) node[below=2pt] {$\disk$};
\end{tikzpicture}
    \caption{Failure of the triangle inequality for $\metric_{V}^2.$}
    \label{fig:many-to-many-failure}
\end{figure}

Let $\mathcal{M}$ and $\mathcal{N}$ be cell structures and $j\in \set{1,2}.$ It is obvious from either the discrete or continuous definitions that 
$$\metric_{\mathrm{V}}^j(\mathcal{M},\mathcal{N})=\metric_{\mathrm{V}}^j(\mathcal{N},\mathcal{M})\geq 0\,.$$ 
If $\metric_{\mathrm{V}}^1(M,N)=0$ then the symmetric difference of every pair of cells in an optimal matching has zero volume. That is, $M$ and $N$ differ up to a rearrangement of a set of measure zero and are equal as partitions of $\Omega.$

To show that $\metric_{\mathrm{DV}}^1$ satisfies the triangle inequality let $\mathcal{M}=\set{M_i}_{i\in I},\mathcal{N}=\set{N_j}_{j\in J},$ and $\mathcal{O}=\set{O_k}_{k\in K}$ be cell structures in $\Omega$ and let $R$ and $S$ be optimal partial matchings between the cells of $\mathcal{M}$ and $\mathcal{N}$ and those of $\mathcal{N}$ and $\mathcal{O},$ respectively. Define 
$$\hat{T}\subset \paren{I\cup \set{\ast}}\times \paren{J\cup \set{\ast}} \times \paren{K\cup \set{\ast}}$$
to be the set of triplets of the following three types
\begin{itemize}
    \item $\paren{i,j,k}$ if $\paren{M_i,N_j}\in R,$ $\paren{N_j,O_k}\in S,$ $N_j\neq \ast,$ and $\paren{M_i,O_k}\neq \paren{\ast,\ast},$ 
    \item $\paren{i,\ast,\ast}$ if $\paren{M_i,\ast}\in R,$
    \item and $\paren{\ast,\ast,k}$ if $\paren{\ast,N_k}\in S.$
\end{itemize}
As 
$$T=\set{\paren{M_i,O_k}:\exists j \text{ s.t.} \paren{i,j,k}\in\hat{T}}$$
is a partial matching between $M$ and $O,$ it follows that 
\begin{align*}
\metric_{\mathrm{DV}}^1(\mathcal{M},\mathcal{O}) & \leq \frac{1}{2}\sum_{\paren{i,j,k}\in \hat{T}}|M_i \triangle O_k |\\
&\leq \frac{1}{2}\sum_{\paren{i,j,k}\in \hat{T}}|M_i \triangle N_j |+|N_j \triangle O_k |\\
&=\frac{1}{2}\sum_{\paren{i,j}\in R: j\notin \hat{J}} |M_i \triangle N_j | + \frac{1}{2}\sum_{\paren{j,k}\in S: j\notin \hat{J}} |N_j \triangle O_k |\\
&\leq \metric_{\mathrm{DV}}^1(\mathcal{M},\mathcal{N}) + \metric_{\mathrm{DV}}^1(\mathcal{N},\mathcal{O}),
\end{align*}
where $\hat{J}=\set{j:\paren{\ast,j}\in R, \paren{j,\ast}\in S}.$

\textbf{Failure of triangle inequality for $\metric_{\mathrm{V}}^2$.}
Let $\Omega = \mathbf{D}^2$ be the open disk of area one in two dimensions. Consider the following three cell structures. Let $A=\set{g_1,g_2,g_3}$ be a partition of the disk into equal thirds, $B=\set{g_4, g_5, g_6, g_7}$ be similar to $A$ with an inclusion $g_7$ such that $g_4, g_5, g_6$ are nonadjacent and $0<|g_7|<\frac{1}{3}$, and $C = \set{g_8}$. A visualization of the cell structures can be found in Figure \ref{fig:many-to-many-failure}. Since every cell in $A$ is adjacent to one another, $\metric_{\mathrm{V}}^2(A,C)=2/3$. However, both $\metric_{\mathrm{V}}^2(A,B)=|g_7|$ and $\metric_{\mathrm{V}}^2(B,C)$ since $\set{g_4,g_5,g_6} \mapsto \set{g_8}$. By construction $\metric_{\mathrm{V}}^2(A,C) \nleq \metric_{\mathrm{V}}^2(A,B) + \metric_{\mathrm{V}}^2(B,C)$.

%By the pseudometric properties of $\metric_{\mathrm{V}}^1$, it suffices to check that $\metric_{\mathrm{V}}^1(p,q)=0$ if and only if $p=q$. Let $p,q \in \mathrm{PC}(\Omega; \mathbb{R}^m)$. Suppose that $p=q$ and pick $\mathrm{Id} \in G$ to be the identity map. Then trivially $\|p-\mathrm{Id}(q)\|_{L^1}=0$. Conversely, if $\metric_{\mathrm{V}}^1(p,q)=0$  
\end{proof}

A notion relevant to our computational experiments of \S\ref{sec:comp_exp} is that of a conditionally negative definite metric.
A symmetric function $\psi\colon Z\times Z \to \mathbb{R}$ is called \textit{\textbf{conditionally negative definite} if for any finite subset $\{z_1, \dots, z_m\}$ of $Z$, the matrix $D = (\psi\paren{z_i, z_j})_{i,j = 1, \dots, m}$ is conditionally negative definite, i.e., for any vector $\boldsymbol{\alpha}^T = (\alpha_1, \dots, \alpha_m) \in \mathbb{R}^m$ such that $\sum \alpha_i = 0$, we have that $\boldsymbol{\alpha}^T D\boldsymbol{\alpha} \leq 0$.}
The next proposition shows that even for cellular structures within intervals of $\mathbb{R}$, the volume exchange metric is \emph{not} conditionally negative definite.
\begin{prop}
If $\Omega$ is an open interval in $\mathbb{R}$, then $\metric_{\mathrm{V}}^1$ is not conditionally negative definite.
\label{prop:negative_definite}
\end{prop}

\begin{figure}
\center
		\includegraphics[width=0.75\textwidth]{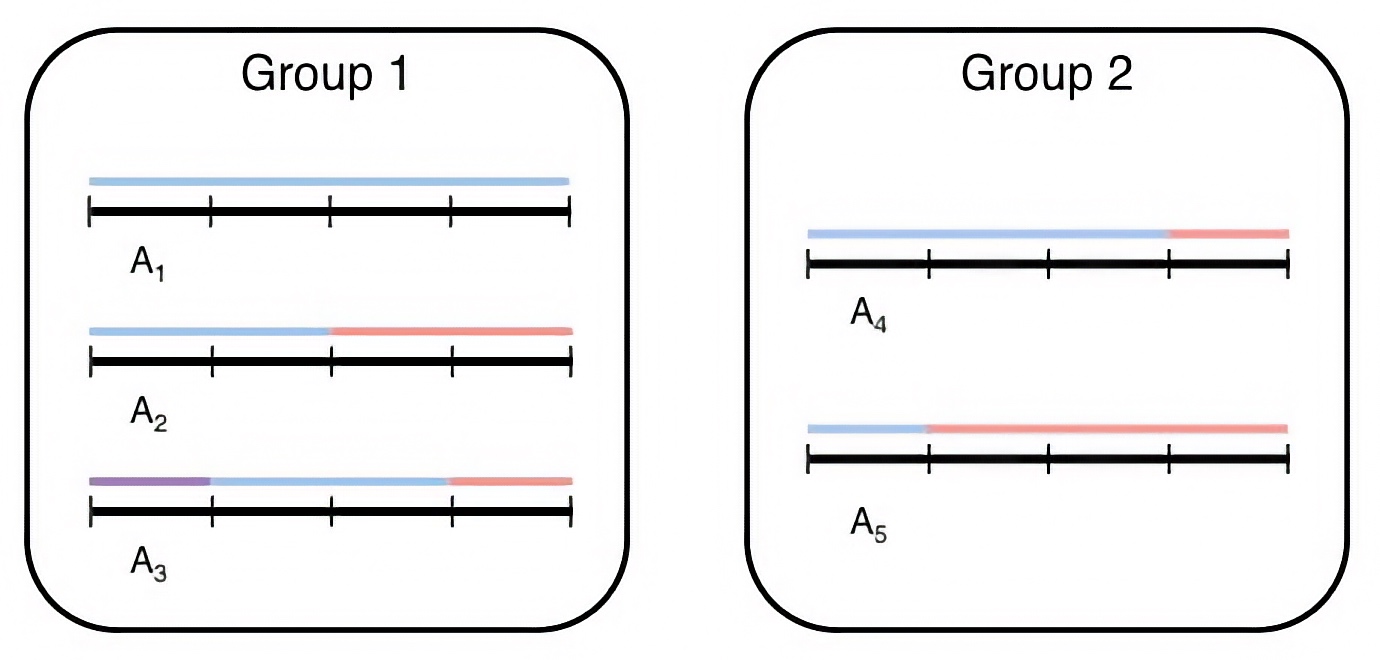} 
\caption{\label{fig:CND} Five cell structures on the interval $\paren{0,4}$.}
\end{figure}

\begin{proof}
The distance matrix 
\[D = \Big(\metric_{\mathrm{V}}^1\paren{A_i, A_j}\Big)_{i,j = 1,\dots, 5} = \begin{pmatrix}
    0 & 2 & 2 & 1 & 1 \\
    2 & 0 & 2 & 1 & 1 \\
    2 & 2 & 0 & 1 & 1 \\
    1 & 1 & 1 & 0 & 2 \\
    1 & 1 & 1 & 2 & 0
\end{pmatrix}\]
for the five cell-structures $A_1, \dots, A_5$ of $\Omega = (0, 4) \subset \mathbb{R}$ depicted in Figure~\ref{fig:CND} is not conditionally negative definite. 
Indeed, for the vector $\boldsymbol{\alpha}^T = (2, 2, 2, -3, -3)$ we get $\boldsymbol{\alpha}^T D \boldsymbol{\alpha} = 12$. 
\end{proof}
A similar example can be constructed when $\Omega$ is the open unit cube in $\mathbb{R}^n$. 
One simply needs to take the cells of each of the five partitions of the open interval in the proof of the previous proposition and take its Cartesian product with $(0, 1)^{n-1}$.

\subsection{A Dual Formulation of $\metric_{DV}^{1}$}\label{subsec:dual}
The following theorem states that the discrete one-to-one volume exchange metric is equal to an equivalent dual formulation:
\begin{equation}\label{def:dual-discrete-volume-exchange-metric}
\metric^{1\ast}_{DV}\paren{\cellstrucP, \cellstrucQ} \coloneq \mathscr{L}^n(\Omega) - \max_{\partialmatching \in \PM(\cellstrucP, \cellstrucQ)} \sum_{(E,F) \in \partialmatching} \mathscr{L}^n(E \cap F).
\end{equation}

\begin{theorem}\label{thm:dual-equality}
Suppose $\cellstrucP$ and $\cellstrucQ$ are cell structures of a common bounded, open, connected domain $\Omega \subset \mathbb{R}^n$.
Then 
\[
\metric_{DV}^1(\cellstrucP, \cellstrucQ) = \metric^{1\ast}_{DV}(\cellstrucP, \cellstrucQ). 
\]
\end{theorem}
\begin{proof}

By the definition of a partial matching $\partialmatching\in\PM\paren{\cellstrucP,\cellstrucQ},$ 
\[
\sum_{(E, F) \in \partialmatching} \mathscr{L}^n(E) = \sum_{E \in \cellstrucP} \mathscr{L}^n(E) = \mathscr{L}^n(\Omega) = \sum_{F \in \cellstrucQ} \mathscr{L}^n(F) = \sum_{(E, F) \in \partialmatching} \mathscr{L}^n(F).
\]

Thus, for any partial matching $\partialmatching \in \PM(\cellstrucP, \cellstrucQ)$, 
\[
\begin{aligned}
\mathscr{L}^n(\Omega) - \sum_{(E,F) \in \partialmatching} \mathscr{L}^n(E \cap F)&= \sum_{(E,F) \in \partialmatching} \frac{1}{2}\big(\mathscr{L}^n(E) + \mathscr{L}^n(F)\big) - \mathscr{L}^n(E \cap F)\\
&= \frac{1}{2} \sum_{(E,F) \in \partialmatching} \paren{\mathscr{L}^n(E) - \mathscr{L}^n(E \cap F) + \mathscr{L}^n(F) - \mathscr{L}^n(E\cap F)}\\
&= \frac{1}{2}\sum_{(E,F) \in \partialmatching} \mathscr{L}^n(E\triangle F).
\end{aligned}
\]
Because $d_{DV}^{1\ast}$ can be formulated as a minimization problem, i.e., 
\[
\mathscr{L}^n(\Omega) - \max_{\partialmatching \in \PM(E, F)} \sum_{(E,F) \in \partialmatching} \mathscr{L}^n(E \cap F)
= \min_{\partialmatching \in \PM(\cellstrucP, \cellstrucQ)}\Big\{\mathscr{L}^n(\Omega) - \sum_{(E,F) \in \partialmatching} \mathscr{L}^n(E \cap F)\Big\}, 
\]
we get that $\metric_{DV}^{1\ast}(\cellstrucP, \cellstrucQ) = \metric^{1}_{DV}(\cellstrucP, \cellstrucQ)$. 
\end{proof}

We note that the proof of Theorem~\ref{thm:dual-equality} also says that $\partialmatching^{\mathrm{opt}}$ maximizes $\sum_{(E, F) \in \partialmatching} \mathscr{L}^n\paren{E \cap F}$ over all $\partialmatching \in \PM\paren{\cellstrucP, \cellstrucQ}$ if and only if it minimizes $\sum_{(E,F) \in \partialmatching} \mathscr{L}^n\paren{E\triangle F}$ over all $\partialmatching \in \PM\paren{\cellstrucP, \cellstrucQ}$.

% \section{Computation of the Discrete Volume Exchange Dissimilariaties}

\subsection{Computing $\metric_{\mathrm{DV}}^1$}\label{subsec:compute11}
In the remainder of the paper, we assume that the cell structures have a finite number of cells. We outline a procedure to find the partial matching $R^{\mathrm{opt}}\in \mathrm{PM}(I,J)$ that minimizes the volume exchange using the formulation in Equation \ref{def:dual-discrete-volume-exchange-metric}. 
Towards that end,  we construct a weighted balanced bipartite graph $B=(U,V,E)$ which encapsulates the set of admissible partial matchings. Let $U_1$ include a vertex  $u_i$ for each cell $U_i\in M,$ and let $U_2$ be a collection of ``dummy'' vertices $u_j^{\ast}$ for each cell $V_j$ of $N.$ Set $U=U_1\cup U_2,$ and define $V, V_1,$ and $V_2$ analogously so that $\abs{U}=\abs{V}.$

%We construct a balanced bipartite graph $B=(U,V,E)$ with vertex and edge sets that are related to $M$ and $N$ as follows. Let $U_1 = \{u_i\}_{i \in I}$ be a set of vertices such that each $u_i$ corresponds to a cell $U_i \in M$ and similarly let $V_1 = \{v_j\}_{j \in J}$ for each $V_j \in N$. To ensure that $|U|=|V|$ we additionally include a set of ``dummy"--nodes $U_2 = \{u_j^{\emptyset}\}_{j \in J}$ and $V_2 = \{v_i^{\emptyset}\}_{i \in I}$ and let $U=U_1\cup U_2$ and $V=V_1\cup V_2$.
 
%The edge set $E$ is generated between a pair of vertices if they satisfy one of the two criteria:
$E$ consists of edges of three types:

\begin{itemize}
    \item[(1)] an edge $e_{i}^{\ast}$ with weight $w_{i}^{\ast}=0$ between each vertex $u_i\in U_1$ and the corresponding dummy vertex $v_{i}^{\ast}\in V_2$;
    \item[(2)] an edge $\hat{e}_{j}^{\ast}$ with weight $\hat{w}_{j}^{\ast}=0$ between each vertex $v_j\in V_1$ and the corresponding dummy vertex $u_{j}^{\ast}\in U_2$;
    \item [(3)]an edge $e_{ij}$ is included between vertices $u_i\in U_1$ and $v_j\in V_1$ when their corresponding cells overlap, with weight $w_{ij} = |U_i \cap V_j|$.
\end{itemize}

%Intuitively, edges generated by (1) encode the cost of including/collapsing a cell from the cellular structure. Furthermore, because $\metric_{DV}^{1^\ast}$ can be formulated as a minimization problem presented in Theorem \ref{thm:dual-equality}, it is sufficient to consider the set of overlapping cells as they provide an upper bound on $\metric_{DV}^{1^\ast}$. Condition (2) encodes these admissible pairings, reducing the size of the search space.

%The computation of $\metric_{DV}^{1^\ast}$ is simplified to computing a linear assignment problem on the cost matrix (i.e., the weighted adjacency matrix of $B$). The time complexity of computing the metric under this construction is $\mathrm{O(max}\{|I|,|J|\}^3)$ and can be computed using classical algorithms such as the 

Then
$$\metric_{\mathrm{DV}}^{1^\ast}\paren{M,N}= \abs{\Omega}-\max_{\mathbf{x}} \sum_{i,j=1}^{|U|} \mathrm{Adj}(B)_{ij}\mathbf{x}_{ij}$$ where $\mathrm{Adj}(B)$ is the weighted adjacency matrix of $B$ (where non-edges have weight zero) and $\mathbf{x}_{ij}$ is a binary decision variable with the constraint $\sum_{j=1}^{|U|}\mathbf{x}_{ij}=1$ for all $1\leq i\leq |U|$, similarly for $j$.
The right hand side is a  linear assignment problem that can be solved using classical methods such as the Hungarian algorithm or the Jonker--Volgenant--Castanon algorithm \cite{edmonds1972theoretical, jonker1987shortest, kuhn1955hungarian}, which have time complexity $\mathrm{O(max}\{|I|,|J|\}^3).$

% one must consider the symmetric difference between it and its neighboring cells $\{V_{j_{1}},\dots,V_{j_k}\} \subseteq N$. Each choice of pairing sequentially alters the available cells for the next. The resulting search on the graph $B$ is equivalent to matching a rook polynomial to the graph, a problem that is generally $\mathrm{NP}-$hard \cite{godsil1981theory, valiant1979complexity}. 

\subsection{Computing $\metric_{DV}^2$}\label{subsec:computemm}
% We begin by reminding the reader the definition of independent matchings of cell structures. 
% \begin{definition}\label{def:independent-matching}
% Let $q=\paren{V_1,E_1}$ and $G_2=\paren{V_2,E_2}$ be graphs. An \textbf{independent matching} between $V_1$ and $V_2$ is a subset $R$ of $(V_1\cup \{\emptyset\}) \times (V_2 \cup \{\emptyset\})$ such that  
% \begin{enumerate}
% \item[(a)] the canonical projections $R \to V_1$ and $R \to V_2$ are onto, and 
% \item[(b)] $R \cap (V_1\times V_2)$ is independent in the sense that if $\paren{v_1,w},\paren{v_2,w}\in R \cap (V_1 \times V_2)$ then $v_1$ and $v_2$ are not connected by an edge of $q,$ and if $\paren{v,w_1},\paren{v,w_2}\in R \cap (V_1 \times V_2)$ then $w_1$ and $w_2$ are not connected by an edge of $G_2.$  
% \end{enumerate}
% We denote the collection of independent matchings between $I$ and $J$ as $\IM(I,J)$.
% \end{definition}

We derive the set of independent matching between a pair of cellular structures from the adjacency graph on their cells. The collection of independent matching between cellular structures $M$ and $N$ is in general much larger than the class of partial matchings between them. 
%Unlike the prior metric, the search space for the optimal minimizer $R^{\mathrm{opt}} \in \IM(M,N)$ is significantly larger. 
In particular, the matching realizing $\metric^{2}_{DV}(M,N)$ minimizes the symmetric difference between \textit{subcollections} of non-adjacent cells; the previous metric only permits matchings between individual cells. We outline a procedure to find that optimal matching $R^{\mathrm{opt}} \in \IM(I,J).$ % \in \text{IM}(M,N)$ that minimizes $\metric^{2}_{DV}(M,N)$. 

Suppose $\sigma \subset I$ and let $M^{\sigma}$ denote a subcollection of cells indexed by $\sigma$. Let $\mathbf{x} \in \set{0,1}^{|I|}$ be a binary integer vector. We say that $\mathbf{x}^{\sigma}$ corresponds to $M^{\sigma}$ if $\mathbf{x}^{\sigma}_i=\chi_{U_i\in M^{\sigma}}.$ Suppose $\mathbf{x}^{\sigma}$ and $ \mathbf{y}^{\tau}$ correspond to subcollections $M^\sigma$ and $N^{\tau}$ and let $\mathbf{Z}^{\sigma\tau} \in \R^{|I|\times|J|}$ denote their outer product. 
Consider the real vectors $\bm{\alpha} \in \R^{|I|}$ and $\bm{\beta} \in \R^{|J|}$ such that $\bm{\alpha}_i = |U_i|$ and $\bm{\beta}_j = |V_j|$ and let $\bm{\Gamma} \in \R^{|I|\times|J|}$ be a real matrix such that $\bm{\Gamma}_{ij} = |U_i \cap V_j|$.  
By construction $|M^\sigma| = \bm{\alpha}^\top \mathbf{x}^{\sigma}$, $|N^\tau| = \bm{\beta}^\top \mathbf{y}^{\tau}$, and \[\mathrm{tr}( \bm{\Gamma}^\top \mathbf{Z}^{\sigma\tau}) = \sum_{\substack{U_i \in M^\sigma \\ V_j \in N^\tau }}|U_i\cap V_j|.\] As such, if $M^\sigma$ and $N^{\tau}$ are subcollections paired by an independent matching $R$, then their symmetric difference equals \[|M^\sigma \triangle N^{\tau}| = \bm{\alpha}^\top \mathbf{x}^{\sigma} + \bm{\beta}^\top \mathbf{y}^{\tau} - 2  \mathrm{tr}( \bm{\Gamma}^\top \mathbf{Z}^{\sigma\tau})\] and   
\begin{equation}
\label{eqn:dv_many_to_many}\metric_{DV}^2(M,N) = \frac{1}{2}\min_{ R \in \IM(M,N)}\sum_{ \paren{\sigma,\tau} \in R} \bm{\alpha}^\top \mathbf{x}^{\sigma} + \bm{\beta}^\top \mathbf{y}^{\tau} - 2  \mathrm{tr}( \bm{\Gamma}^\top \mathbf{Z}^{\sigma\tau}).
\end{equation}

% If $\mathbf{x}^{\sigma}$ and $\mathbf{y}^{\tau}$ correspond to subcollections $M^{\sigma}$ and $N^{\tau}$ respectively and set $\mathbf{z}^{\sigma \tau}= \mathbf{x}^{\sigma} (\mathbf{y}^{\tau})^\top$

$R^{\mathrm{opt}}$ can be determined by finding the optimal minimizing decision vectors $\mathbf{x} \in \{0,1\}^{|I|}$ and $\mathbf{y} \in \{0,1\}^{|J|}$ subject to the constraint that the corresponding subcollections consist of nonadjacent cells. In its current form, the many-to-many volume exchange dissimilarity is presented as a binary integer quadratic program, and the general class of binary integer quadratic programs is $\mathrm{NP}-$complete. To reduce the burden of computing $\metric_{DV}^2$ a \textit{branch--cut} algorithm is used to prune the search space.
This is achieved by (1) approximating the solution using a faster quadratic program and (2) tightening the bounds to be integer valued using a branch-and-bound and cutting planes algorithm.
Additional details regarding the branch--cut algorithm and its implementation can be found in \cite{padberg1991branch}.

In the next section, we only compute  $\metric_{DV}^2$ for small sample windows taken from a larger cell structure.
These windows typically contain $\sim 50$ cells or less, and the algorithm described here yields reasonable runtimes.
%See Section~\ref{sec:comp_exp} below for further detail regarding the runtime of the volume exchange algorithms.

\section{Computational Experiments}
\label{sec:comp_exp}

% Comments: Add more depth about dataset, and Dylan's paper. Add Discussion of figures

With the theoretical description of the volume exchange dissimilarities established, this section shows their effectiveness in computational experiments.
Our approach is a modification of an earlier study~\cite{miley2025microstructure} which proposed a practical framework for the comparison of polycrystalline microstructures.
The key idea was to represent a microstructure as a family of window probability distributions parameterized by the window size. % of windows of local configurations at various scales.
Windows were computationally represented as voxelized two-dimensional cell structures with an associated measure having a value of one on cell boundary voxels and zero on cell interior voxels.
A window distance was defined as the unbalanced Wasserstein distance between the associated measures using the taxicab distance on the underlying voxel space; this is referred to as the grain boundary distance below.
For a given window size, an empirical window distribution was sampled from each of the microstructures being compared.
A microstructure distance was then defined as the Wasserstein distance between the associated empirical window distributions using the grain boundary distance on the underlying window space.
% A different distance called the grain boundary distance on voxelized cell structures in two dimensions is defined  (an unbalanced Wasserstein distance between measures assigning the same weight to each voxel in the boundary), and cell structures in a larger ambient space are compared by viewing samples of windows as empirical probability distributions and taking the Wasserstein distance between them. 
This framework was applied to a dataset of synthetic microstructures, and was shown to be able to distinguish microstructures belonging to different classes and to match sampled windows to the most likely originating microstructures.
% the Wasserstein distance between probability distributions of windows is demonstrated to distinguish between different cell structure classes and the grain boundary distance enables a small sample of windows to be correctly matched to the cell structure they were sampled from.

Our approach here follows that of Ref.\ \cite{miley2025microstructure} but with (i) the grain boundary distance on windows replaced by the volume exchange metric and volume exchange dissimilarity and (ii) the Wasserstein distance on empirical window distributions replaced by the energy distance (defined below).
% The volume exchange distance has several advantages over the grain boundary distance, including being much more computationally efficient.
The use of the energy distance also effectively resolves an issue reported in Ref.\ \cite{miley2025microstructure} of there being a non-zero self-distance due to the stochasticity of the empirical window sampling.
A more comprehensive computational study comparing the effectiveness of the volume exchange dissimilarities and the grain boundary distance is intended to be the subject of a forthcoming companion paper.

We use the same set of three-dimensional microstructures as in Ref.\ \cite{miley2025microstructure}.
Figure \ref{fig:microstructures} shows representative two-dimensional sections of the five microstructures.
EI is considered the base equiaxed structure, with the other four deviating from EI in ways that correspond to realizable physical processes.
The larger grain size of EII could be formed by an isothermal anneal, and the large grains in BII could indicate a process of abnormal grain growth.
The grains of RI are highly anisotropic as would occur during a rolling process, and BI contains small grains that could be inclusions from a precipitation reaction.
% A visual of the dataset we will be working with can can be found in figure (\ref{fig:microstructures}).
% From that figure, you will notice that each microstructure has a distinct morphological pattern.
% EI and EII are equiaxed microstructures with different average grain sizes.
% BI and BII are bimodal microstructures that share the same primary structure as EI but have a secondary structure with a different average grain size.
% RI has the same average grain size as EI but has a clear anisotropic morphology.
These represent five natural and geometrically distinct microstructure classes, and the purpose of our computational experiments is to show that the volume exchange dissimilarities can be used to reliably distinguish microstructures of different classes.
% In the previous sections we have provided the theoretical foundation of the volume exchange dissimilarities along with a comprehensive overview of the algorithms used to compute them. Let us now go over a practical application for the volume exchange dissimilarities to be able to distinguish between different microstructures that have visually distinct morphologies. 

The sampling procedure to construct the empirical window distribution is adapted from Ref.\ \cite{miley2025microstructure}.
Starting with a three-dimensional microstructure in a large cube, two-dimensional microstructures are constructed by intersecting the cube with planes in uniform increments along a single axis, with the increment large enough to minimize the possibility of correlations.
While three-dimensional synthetic cell structures could be represented using samples of three-dimensional windows, real data is most often found in the form of two-dimensional slices.
Square windows of a fixed width chosen to be 20 pixels  are sampled uniformly from the two-dimensional slices, rejecting windows that overlap with one of the previously sampled windows.

%Window sampling procedure:

%\begin{itemize} 
%\item From a 3D cubic volume of a microstructure, take sequential planar level sets along a single axis at uniform increments such that grains in one level set are not part of the next.
%\item Sample windows uniformly from a given level set and resample if a window overlaps with one of the previously sampled windows.
%\end{itemize}

%We prevent overlap in windows to ensure as little correlation between sampled windows. Intuitively, these windows act as the support to an underlying probability distribution on a metric space, we call the {\em window space}, given by a metric between windows. The authors in that \cite{miley2025microstructure} used their own dissimilarity between windows they define as the `grain boundary distance'. They would then compute a dissimilarity between window distributions via the Wasserstein Distance. In our experiment, we swap out their use of the grain boundary distance for our volume exchange dissimilarities.

%\subsection{Experiment}\label{lab:experiment}
After sampling windows from the five classes of synthetic microstructures, we compare the sets of sampled windows using the energy distance \cite{szekely2013energystat} which is commonly used in statistical literature and applications, see \cite{mak2018annals, ziel2021energy}.

%To compare distributions of windows, we chose to adopt a metric that has been well studied in statistical literature  known as the {\bf energy distance} \cite{szekely2013energystat}.

\begin{figure}
    \centering
    \includegraphics[width=\textwidth]{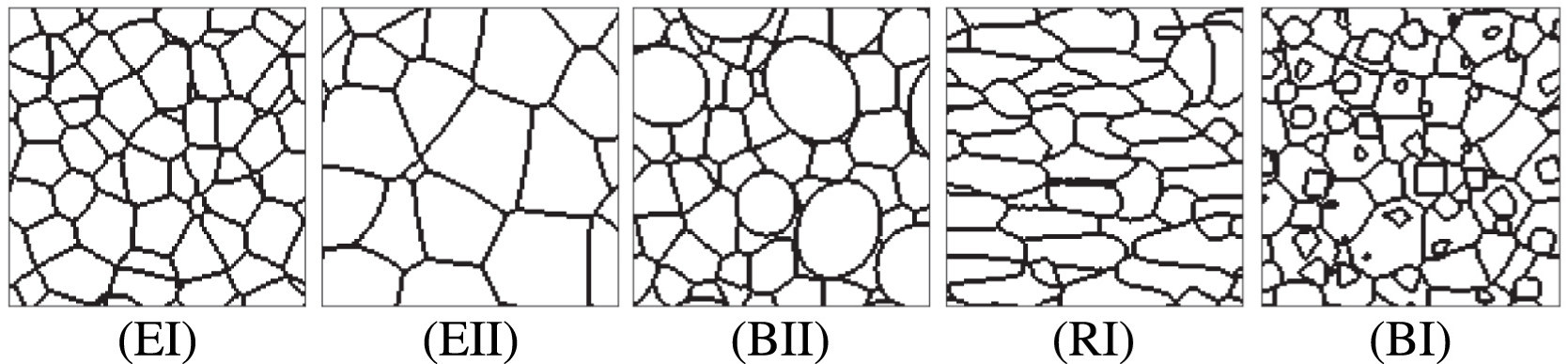}
    \caption{Representative areas ($80\times80$ pixels) of microstructures generated in DREAM.3D after Ref.\ \cite{miley2025microstructure}.
    From left to right, EI is the base equiaxed structure, EII is equiaxed with large grains, BII is the base structure with large inclusions, RI is rolled with highly anisotropic grains, and BI is the base structure with small inclusions.
    % Micrographs are 80 pixels on a side.
    }
    \label{fig:microstructures}
\end{figure}

\begin{definition}
Let $\paren{Z,d_Z}$ be a metric space and let $\mu$ and $\nu$ be probability measures on $Z$.
The {\bf\em energy distance} between $\mu$ and $\nu$ is
%Let $X, X'$ and $Y, Y'$ be random windows distributed according to probability distributions $\mu$ and $\nu$, respectively. 
%We define the {\bf\em energy distance} between $\mu$ and $\nu$ as the square root of Definition~\ref{def:theoretical-energy-distance}.
\begin{equation}\label{eq:theoretical-energy-distance}
    D(\mu, \nu) = \sqrt{2\mathbb{E}_{\mu\times \nu}[\metric_{Z}(X, Y)] - \mathbb{E}_{\mu\times\mu}[\metric_Z(X,X')] - \mathbb{E}_{\nu\times\nu}[\metric_Z(Y,Y')}]
\end{equation}
where $X,X' \sim \mu$ and $Y,Y' \sim \nu.$ 
\end{definition}

The energy distance is a metric on probability distributions on $\paren{Z,\metric_Z}$ if and only if $\metric_Z$ has strong negative type, which is stronger than requiring that $\metric_Z$ is conditionally negative definite.\cite{lyons2013distance} 
However, by Prop.\ \ref{prop:negative_definite} the one-to-one volume exchange metric is {\em not} conditionally negative definite and by Prop.\ \ref{prop:dissimilarity} the many-to-many volume exchange dissimilarity is not even a metric.
As such, the energy distance is not theoretically guaranteed to be a metric nor even be non-negative.
However, the use of Szekely's two sample energy statistic \cite[Definition 6.1]{szekely2013energystat} produces excellent energy distance results in practice.

The energy distance is estimated by letting $\mathrm X = \{x_1,\cdots,x_n\}$ and $\mathrm Y = \{y_1,\cdots,y_n\}$ be windows sampled by the procedure described above.
We first compute a pairwise dissimilarity matrix between the two collections of windows using one of the two volume exchange dissimilarities defined above.
An example of such a pairwise dissimilarity matrix is given in Fig.\ \ref{fig:energy-distance-blocks} where the $\metric_{DV}^1$ is evaluated between the window distributions for EI and EII.
From these pairwise distances, the following V-statistics are defined which represent the average dissimilarities between pairs of samples: 
\begin{align}
\mu_{\mathrm X \mathrm Y} &= \frac{1}{n^2}\sum_{i=1}^n\sum_{j=1}^n\metric_{DV}(x_i,y_j) \\\mu_{\mathrm X} &= \frac{1}{n^2}\sum_{i=1}^n\sum_{j=1}^n\metric_{DV}(x_i,x_j) \\
\mu_{\mathrm Y} &= \frac{1}{n^2}\sum_{i=1}^n\sum_{j=1}^n\metric_{DV}(y_i,y_j)
\end{align}
%$\mu_{\mathrm{XY}}$ represents the average dissimilarity of windows from $\mathrm X$ and $\mathrm Y$. $\mu_{\mathrm X}$ and $\mu_{\mathrm Y}$ represents the average dissimilarity of windows in $\mathrm X$ and $\mathrm Y$, respectively.
% We then estimate the squared 
The energy distance is then estimated from Eq.\ \ref{eq:theoretical-energy-distance} as
% \begin{equation}\label{eq:empirical-energy-distance}
%     \frac{1}{n^2}\left[\sum_{i,j=1}^n2\metric_{DV}(x_i,y_j) - \metric_{DV}(x_i,x_j)-\metric_{DV}(y_i,y_j) \right]
% \end{equation}
\begin{equation}\label{eq:empirical-energy-distance}
    \varepsilon(\mathrm X,\mathrm Y) =  \sqrt{2\mu_{\mathrm{XY}} - \mu_\mathrm X - \mu_\mathrm Y}
\end{equation}

\begin{figure}
    \centering
    \includegraphics[width=0.75\textwidth]{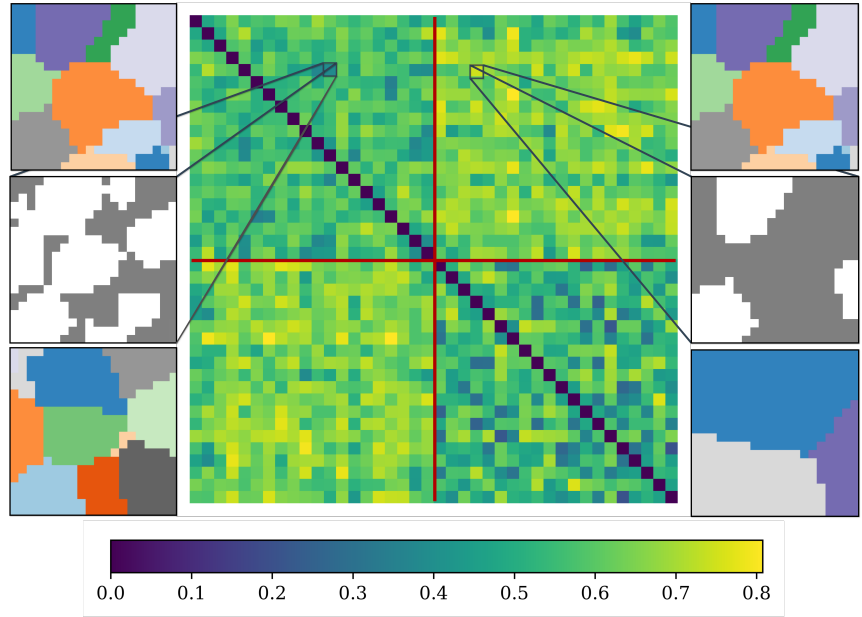}
    \caption{Pairwise dissimilarity matrix for 40 $2D$ windows with the first 20 windows (rows) sampled from EI and the second 20 sampled from EII, with the volume exchanges between two pairs of windows shown in gray on the sides. 
    % The first 20 rows/columns represent windows sampled from EI while the remaining rows/columns represent windows sampled from EII.
    Numerical values are the $\metric_{DV}^1$ volume exchange distance between the corresponding windows and are normalized by window area to be bounded between 0-1.
    Notice the resemblance to a block matrix with different average dissimilarity values as indicated by the red lines.
    The empirical energy distance in Eq.\ \ref{eq:empirical-energy-distance} effectively finds the difference between mean values of the diagonal and off-diagonal blocks.}
    \label{fig:energy-distance-blocks}
\end{figure}

Figure~\ref{fig:heatmaps} illustrates the pairwise energy distance estimates for samples of 1200 windows from the five synthetic microstructures. 
% +: Many-many is more(?) physically motivated, 
% -: M-M not a metric, expensive to compute
EII is the furthest from all other microstructures except $RI$, and this makes sense since EII has substantially larger grains and the volume exchange dissimilarities are sensitive to variations in cell area.
In contrast, observe that BII is the closest non diagonal entry to EI, which is likely due to both microstructures having the same base cell-structure while BII's large inclusions slightly drive up the average distance. 
However, notice that BI also shares the same primary structure as EI but has a significantly higher energy distance to EI relative to BII. 
This likely stems from the average cost to match a cell in EI to a cell in BI being inflated by the area of the inclusions that are embedded within the BI base grains.
Finally, the nonzero energy distance between EI and RI likely occurs because cells from EI must compress and stretch to match a cell in RI.

% This motivates, the energy distance between EI and RI is well suited for capturing differences in morphological anisotropy. 

One can notice that both the one-to-one volume exchange metric and many-to-many volume exchange dissimilarity yield similar heatmaps in Fig. \ref{fig:heatmaps}. 
This motivates the use of the one-to-one volume exchange metric as a computationally cheaper alternative to the many-to-many volume exchange dissimilarity.  

\begin{figure}
    \centering
    \includegraphics[width=1\textwidth]{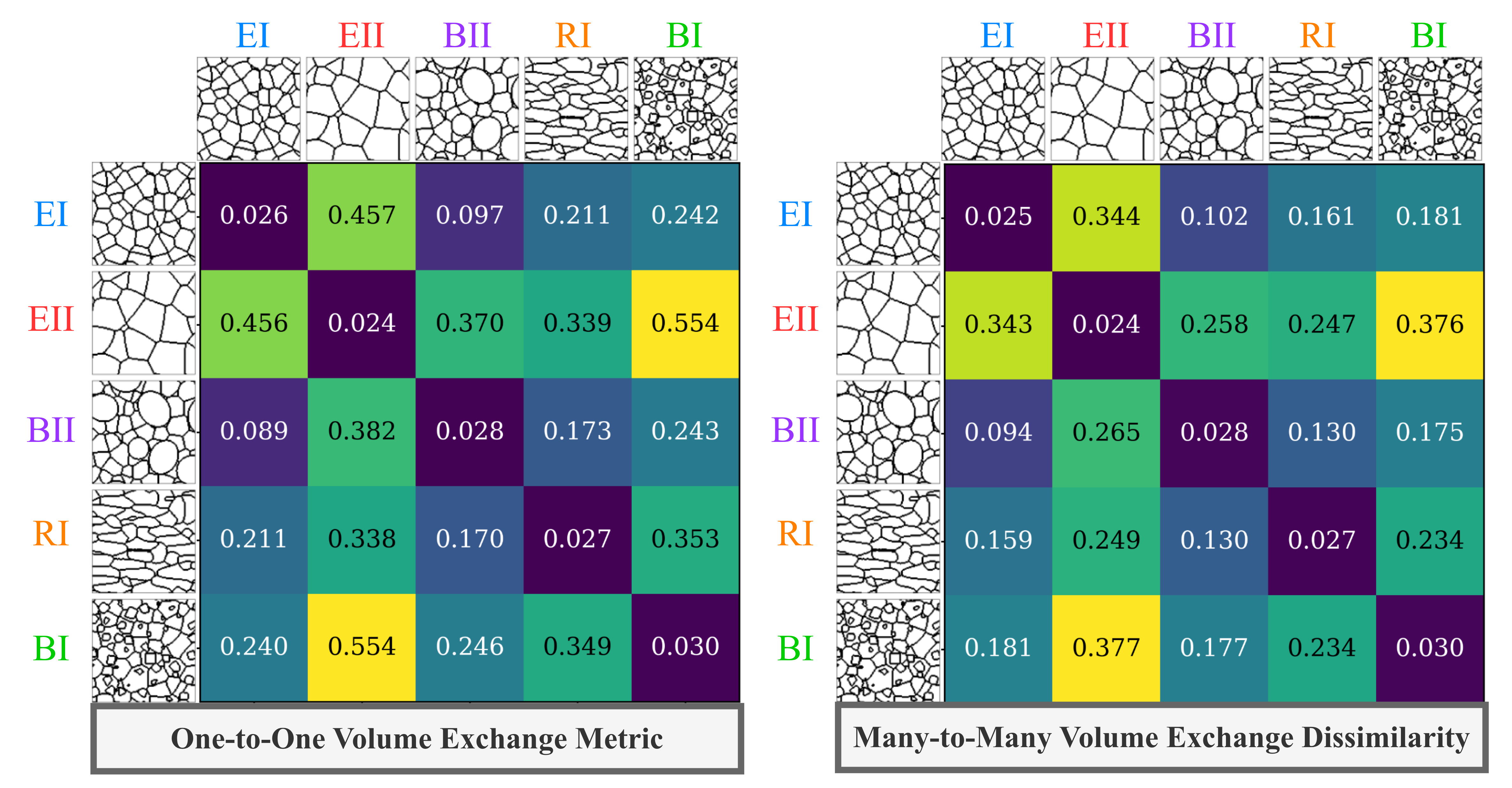}
    \caption{Dissimilarity matrices between window distributions.
    Empirical $2D$ window distributions were generated by sampling 1200 windows uniformly at random from the respective micrographs without replacement.
    Entries were computed using the energy distance in Eq.\ \ref{eq:empirical-energy-distance}.
    The left matrix uses $\metric_{DV}^1$ while the right matrix uses $\metric_{DV}^2$.
    Notice that both matrices satisfy the triangle inequality and that entries along the diagonal are consistently smaller than non-diagonal entries.
    \label{fig:heatmaps}}.
\end{figure}

Figure~\ref{fig:convergence-plots} shows the convergence of the energy distance statistics.  
%To validate the numerical stability of the dissimilarity matrices, we also computed convergence plots in figure (\ref{fig:convergence-plots}) where we plot the energy distances as a function over the number of sampled windows. 
Notice that as the number of samples is increased the relative ordering of energy distances remains stable even for small samples of windows while the self-distance decreases towards zero. Also, it appears that the limiting distance values satisfy the triangle inequality. 
% This stability naturally stems from the estimated expectations converging as our sample size increases.
% This supports the use of the combined volume exchange-energy distance for comparison of cell structures. 
The fact that self energy distances are less than all other distances support the use case of the volume exchange informed energy distance as a way to create  hypothesis tests for microstructure validation procedures. 
Finally, since the relative ordering of the microstructure distances are stable, one does not need to sample a high number of windows, like we do, in practice.

We also compared 3D windows sampled from the synthetic microstructures using the same pipeline. Only $\metric_{DV}^1$ was computed, as $\metric_{DV}^2$ was slow to compute.  These results can be found in Fig. \ref{fig:3d-vol-exchange}. The largest visual discrepancy between the results using 2D windows versus 3D windows comes from the pairwise distances of micrographs between RI and the remaining structures. This likely stems from 3D windows capturing RI's geometric anisotropy from all three axes.

%in the subplots that the ordering of the energy distance between distributions are stable throughout different sample sizes while the distance to one's own distribution converges to 0. This supports the use case of the energy distance utilizing volume exchange dissimilarities for applications in machine learning (classification, clustering, etc.).

% \begin{figure}
%     \centering
%     \includegraphics[width=1\textwidth]{convergenceplots2.png}
%     \caption{Subplots showing the energy distance from the indicated microstructure distribution to every microstructure distribution including itself.
%     % The energy distance from the respective distribution as a function of the number of samples.
%     The top row uses $\metric_{DV}^1$ and the bottom row uses $\metric_{DV}^2$.
%     Notice that the energy distance maintains a consistent ordering for at least 100 sampled windows.
%     Moreover, the distance appears to converge to zero only for self-distances.
%     }
%     \label{fig:convergence-plots}
% \end{figure}

% \begin{figure}[h]
%     \centering
%     \includegraphics[width=1\textwidth]{vertical_convergence_plots.png}
%     \caption{}
%     \label{fig:placeholder}
% \end{figure}

\begin{figure}
    \centering
    \includegraphics[width=0.75\textwidth]{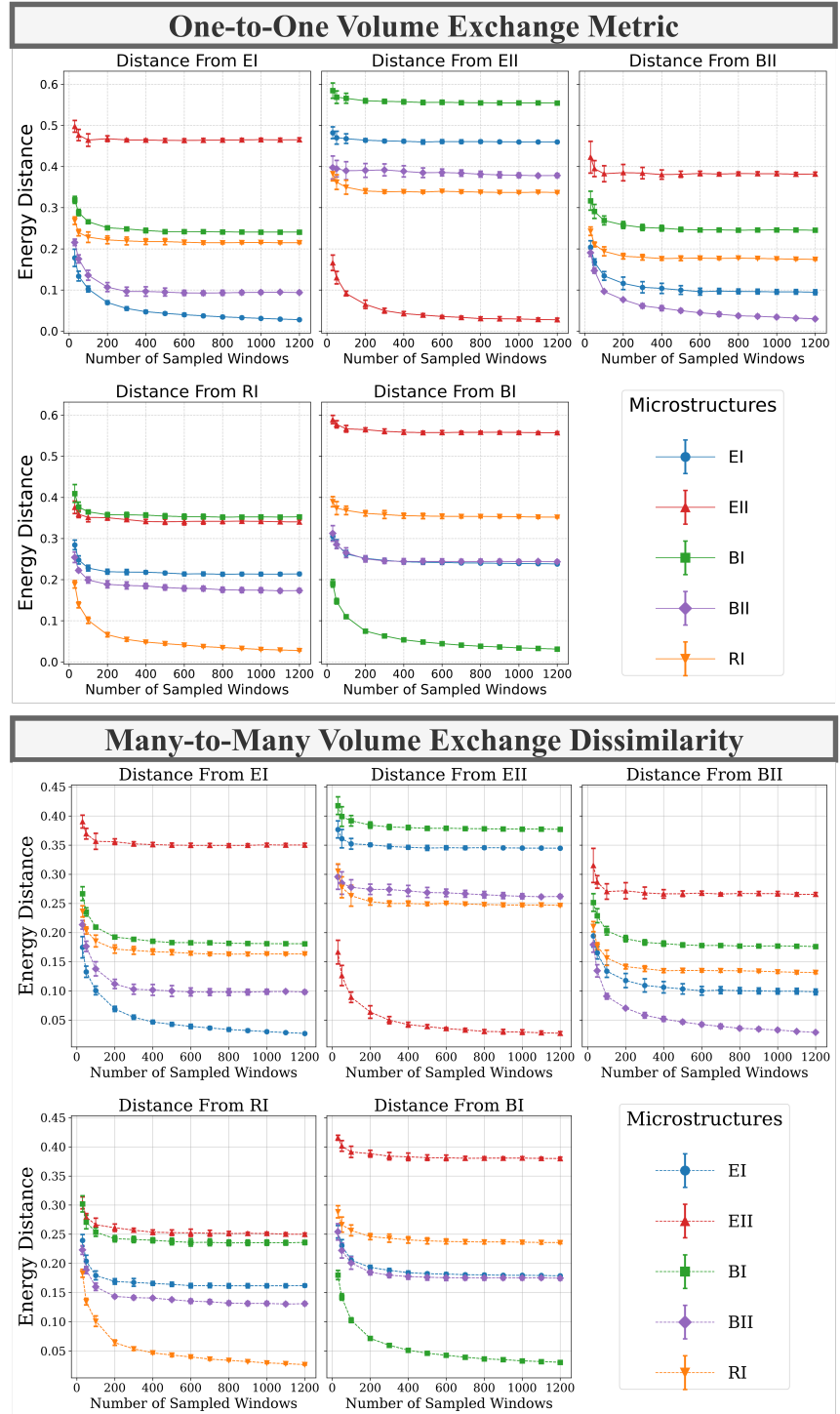}
    \caption{Subplots showing the energy distance from the indicated microstructure distribution to every microstructure distribution including itself.
    % The energy distance from the respective distribution as a function of the number of samples.
    The top plots use $\metric_{DV}^1$ and the bottom plots use $\metric_{DV}^2$.
    Each data point represents the average of 5 energy distances computed across different sets of 2D windows. 
    Error bars indicate $\pm 1$ standard deviation.
    Notice that the energy distance maintains a consistent ordering for at least 100 sampled windows.
    Moreover, the distance appears to converge to zero only for self-distances.}
    \label{fig:convergence-plots}
\end{figure}

\begin{figure}
    \centering
    \includegraphics[width=0.9\textwidth]{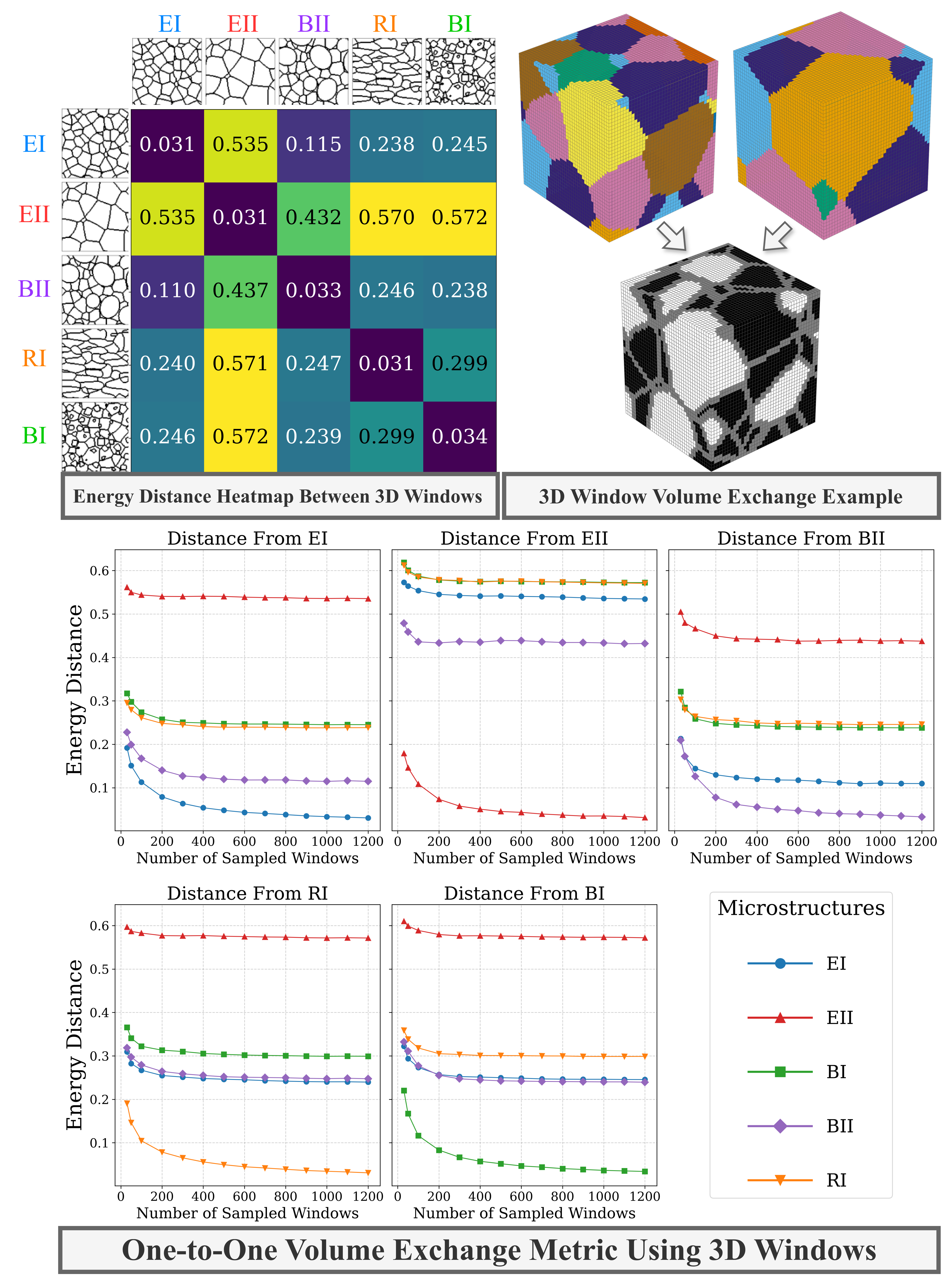}
    \caption{{\bf Top left:} energy distance heatmap using $\metric _{DV}^1$ between 3D window distributions. 1200 windows of dimension $20\times 20\times 20$ were sampled uniformly at random from the respective micrographs without replacement. {\bf Top right:} a visualization of $\metric_{DV}^1$ between 3D windows sampled from EI and EII. Note that for visualization purposes, the 3D windows were scaled uniformly to $40\times 40\times 40$ and an overlay of the union of the grain boundaries between the two windows is shown in gray. {\bf Bottom:} Convergence plots of the energy distance using $\metric_{DV}^1$ on 3D window distributions.}
    \label{fig:3d-vol-exchange}
\end{figure}

The code and data used to generate the Figs. \ref{fig:heatmaps}, \ref{fig:convergence-plots}, and \ref{fig:3d-vol-exchange} can be found on \href{https://github.com/jjung88/volume-exchange}{GitHub} along with additional documentation regarding the implementation.

% \begin{itemize}
% \item Figure on convergence curves for the Wasserstein distance on the 1-1 and many-to-many metrics. 
% \item The slow convergence was motivation to look at the Energy distance. 
% This slow convergence can be explained by Weed and Bach (Sharp asymptotic and finite-sample rates of convergence of empirical measures in Wasserstein distance). 
% Cite and talk about Theorem 1 in the paper (section 4.3). 
% Possibly(?) talk about the covering dimension $d_\ast$. 
% Covering numbers are hard to compute. 
% \item Do we want to estimate the dimension of our microstructures? 
% \item Figures (2) for Energy distance (1-1 and many-to many)
% \item Figure for heat maps (Wasserstein, energy, 1-1, many-to-many)
% With the heat maps it is easier to see that EII is further away on average to the rest of the structures. (look at EII row). 
% In contrast, the EI row is smallest on average. 
% \item Explaining more science?
% \end{itemize}

\section*{LLM Usage Disclosure} 

Any use of generative AI in this manuscript adheres to ethical guidelines for use and acknowledgment of generative AI in academic research \cite{porsdam2024guidelines}. Each author has made a substantial contribution to the work, which has been thoroughly vetted for accuracy, and assumes responsibility for the integrity of their contributions. LLMs were used to implement the code to compute the integer program in Equation~\ref{eqn:dv_many_to_many}, to generate batch scripts to parallelize the computation of pairwise distances, generate Tikz code for Figure \ref{fig:many-to-many-failure}, and to modify Figure~\ref{fig:volume_flux} to create Figure~\ref{fig:volume_flux_manymany}. They were also used to proofread a draft of the manuscript and check for mathematical errors. LLMs were not otherwise used at any stage of the theoretical work or the writing process. 

\section*{Acknowledgments}
S.P. and B.S.\ are grateful for the support of the National Science Foundation under Award No.\ 2232967.
J.K.M.\ is grateful for the support of the National Science Foundation under Award No.\ 2232968.

% Simple scripts were generated in Gemini to parallelize computation of the pairwise Volume Exchange dissimilarity matrices like in Fig. (\ref{fig:energy-distance-blocks}) across multiple CPU nodes on a HPC cluster.

% =======================================

% \bibliographystyle{plain}
\bibliographystyle{unsrtnat}
\bibliography{refs}

\end{document}